\documentclass[10pt,a4paper,reqno]{amsart}
\usepackage[T1]{fontenc} 
\usepackage[utf8]{inputenc}
\usepackage[english]{babel} 
\usepackage[a4paper]{geometry} 
\usepackage{amsthm}
\usepackage{amsmath} 
\usepackage{amssymb} 
\usepackage{amsfonts} 
\usepackage{rotating} 
\usepackage{graphicx} 
\usepackage{booktabs} 
\usepackage{url} 
\usepackage[shortlabels]{enumitem} 
\usepackage{tikz-cd} 
\usepackage{relsize} 
\usepackage{mathtools}
\usepackage{float}
\usepackage[hidelinks,colorlinks=true,allcolors=blue,citecolor=green!50!black]{hyperref}
\hypersetup{linktoc=all}
\usepackage{overpic} 
\usepackage{caption}
\usepackage[foot]{amsaddr} 

\usepackage{color}
\definecolor{dark-red}{RGB}{200,0,0}
\definecolor{dark-green}{RGB}{0,128,0}
\definecolor{light-blue}{RGB}{0,109,217}
\definecolor{light-purple}{RGB}{197,85,254}
\definecolor{light-gray}{RGB}{190,190,190}
\definecolor{q6red}{RGB}{181,41,102}
\definecolor{q5yellow}{RGB}{188,190,59}
\definecolor{q4violet}{RGB}{138,109,202}
\definecolor{rosapastell}{RGB}{227,162,166}
\definecolor{ly}{RGB}{214,202,59}

\usepackage[style=numeric,block=space,giveninits=true]{biblatex} 
\numberwithin{equation}{section} 

\AtBeginDocument{\renewcommand{\d}{\mathop{}\!\mathrm{d}}} 

\newcommand{\R}{\mathbb{R}}
\newcommand{\<}{\langle}

\newcommand{\ex}{\mathrm{ex}}
\newcommand{\en}{\mathrm{en}}
\renewcommand{\>}{\rangle}
\DeclareMathOperator{\id}{id}

\DeclareMathOperator{\tr}{tr}
\DeclareMathOperator{\Ai}{Ai}
\DeclareMathOperator{\Bi}{Bi}
\DeclareMathOperator{\D}{D}
\DeclareMathOperator{\pr}{pr}
\newcommand{\eval}{\big{\rvert}} 

\DeclarePairedDelimiter\abs{\lvert}{\rvert} 
\newcommand{\eps}{\varepsilon} 
\newcommand{\txta}{\textnormal{a}}
\newcommand{\txtr}{\textnormal{r}}
\newcommand{\txts}{\textnormal{s}}
\newtheorem{defn}{Definition}[section] 
\newtheorem{lemma}[defn]{Lemma} 
\newtheorem{thm}[defn]{Theorem} 
\newtheorem{prop}[defn]{Proposition}

\newtheorem{remark}[defn]{Remark}

\begin{document}

\title{The hyperbolic umbilic singularity in fast-slow systems}

\author[H. Jard\'on-Kojakhmetov]{Hildeberto Jard\'on Kojakhmetov$^{1}$}
\address{$^{1}$ Bernoulli Institute for Mathematics, Computer Science and Artificial Intelligence, University of Groningen, 9700 AK Groningen, The Netherlands}
\email{h.jardon.kojakhmetov@rug.nl}
\author[C. Kuehn]{Christian Kuehn$^{2}$}
\address{$^{2}$ Department of Mathematics, Technical University of Munich, Boltzmannstr. 3, 85748 Garching bei München, Germany}
\email{ckuehn@ma.tum.de}
\author[M. Steinert]{Maximilian Steinert$^{2}$}
\email{max.steinert@tum.de}

\begin{abstract}

Fast-slow systems with three slow variables and gradient structure in the fast variables have, generically, hyperbolic umbilic, elliptic umbilic or swallowtail singularities. In this article we provide a detailed local analysis of a fast-slow system near a hyperbolic umbilic singularity. In particular, we show that under some appropriate non-degeneracy conditions on the slow flow, the attracting slow manifolds jump onto the fast regime and fan out as they cross the hyperbolic umbilic singularity. The analysis is based on the blow-up technique, in which the hyperbolic umbilic point is blown up to a 5-dimensional sphere. Moreover, the reduced slow flow is also blown up and embedded into the blown-up fast formulation. Further, we describe how our analysis is related to classical theories such as catastrophe theory and constrained differential equations.

\end{abstract}

\maketitle


{ \hypersetup{linkcolor=black}
\tableofcontents
}
\section{Introduction and main result}

In this article we analyze a fast-slow system, also called a singularly perturbed differential equation, with two fast and three slow variables of the form
\begin{equation}
	\begin{split}
		\label{hyp_umb}
		\eps\dot{x} &= x^2 + ay + b + O(\eps) \\
		\eps\dot{y} &= y^2 + ax + c + O(\eps) \\
		\dot{a} &= g_a(x,y,a,b,c,\eps) \\
		\dot{b} &= g_b(x,y,a,b,c,\eps) \\
		\dot{c} &= g_c(x,y,a,b,c,\eps),
	\end{split}
\end{equation}
or equivalently, after suitable time change onto the fast time scale,
\begin{equation}
	\begin{split}
		\label{hyp_umb_f}
		x' &= x^2 + ay + b + O(\eps) \\
		y' &= y^2 + ax + c + O(\eps) \\
		a' &= \eps g_a(x,y,a,b,c,\eps) \\
		b' &= \eps g_b(x,y,a,b,c,\eps) \\
		c' &= \eps g_c(x,y,a,b,c,\eps),
	\end{split}
\end{equation}
for sufficiently small time scale separation parameter $0< \eps \ll 1$. More precisely, we focus on the dynamics of \eqref{hyp_umb_f} near the origin, which is the most degenerate singularity of \eqref{hyp_umb_f}. We refer to this singularity as hyperbolic umbilic singularity, since it arises from catastrophe theory, as will be shown below. For basic fast-slow systems terminology we refer to section \eqref{fs_terminology}. For more details and recent results, the reader is referred to, e.g. \cite{kuehn2015multiple,hjk2021survey,wechselberger2020geometric,de2021canard}.
 
Let us outline our motivation to study a hyperbolic umbilic singularity and in particular system \eqref{hyp_umb_f}.
The fast-slow system \eqref{hyp_umb_f} arises from the general class of fast-slow systems of the form
\begin{equation}
	\begin{split}
		\label{start0}
		z' &= \nabla_z V(z,\alpha) + O(\eps)\\
		\alpha' &= \eps g(z,\alpha,\eps),
	\end{split}
\end{equation}
where $z\in\mathbb R^n$, $\alpha\in\mathbb R^r$ and $V$ is a sufficiently smooth function. The formulation \eqref{start0} gives a (parameter-dependent) gradient field in the limit $\eps\to0$.
Note that the class of systems of the form \eqref{start0} contains any fast-slow system having only one fast variable. In particular, \eqref{start0} contains the fast-slow systems, which exhibit fold or cusp singularities, after suitable center manifold reduction and transformations \cite{krupa2001extending,broer2013cusp,wechselberger2012apropos,kojakhmetov2016cusp,szmolyan2001canards,krupa2001transcritical,kuehn2015multiple}. In this article, we are in general interested in a fast-slow system of the form \eqref{start0} near a singularity, for which at least two fast variables are necessary. More precisely, we are interested in a fast-slow system \eqref{start0} near a hyperbolic umbilic singularity. This singularity arises from elementary catastrophe theory \cite{arnold1973normal,michor1985cat} and needs at least two fast variables in \eqref{start0} for its occurrence. Additionally, this singularity occurs generically for at least three slow variables in \eqref{start0}. For a detailed explanation and the definition of this singularity in the setting of \eqref{start0} we refer to section \ref{sec:hyp_umb_singularity}, in particular to Definition \ref{sf:singularity}. The hyperbolic umbilic singularity, among others, already appeared in Takens' study of constrained differential equations \cite{takens1976}, which are related to the singular limit of \eqref{start0}. In fact, our main motivation is a list of ``normal forms'' provided by Takens \cite[{(4.10)}]{takens1976}, in which the hyperbolic umbilic singularity is contained. For detailed statements, we refer to section \ref{sec:hyp_umb_singularity}. Schematically, though, one can regard Takens' work \cite{takens1976}, to which we will refer to as ``Takens' program'', as a classification of gradient-type constrained differential equations, for which the constraints are given by the critical points of potential functions, that lead to a hierarchy of phase spaces (critical manifolds) as shown in figure \ref{fig:Hierarchy}. 
\begin{figure}[htbp!]
    \begin{tikzpicture}
        \node at (-5,0){
        \includegraphics[]{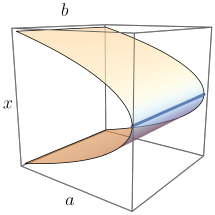}
        };
        \node at (0,0){
        \includegraphics[]{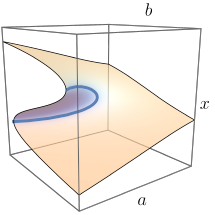}
        };
        \node at (5,0){
        \includegraphics[]{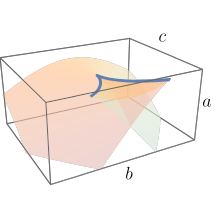}
        };
        \node at (-2.5,-5){
        \includegraphics[]{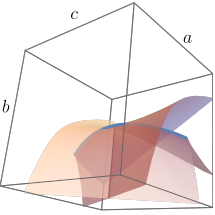}
        };
        \node at (2.5,-5){
        \includegraphics[]{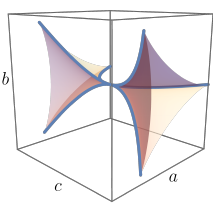}
        };
        \node at (-5,2){fold};
        \node at (0,2){cusp};
        \node at (5,2){swallowtail};

        \node at (-2.5,-3){hyperbolic umbilic};
        \node at (2.5, -3){elliptic umbilic};
    \end{tikzpicture}

    \caption{Takens' program is concerned with dynamical systems whose orbits are constrained to manifolds like the ones shown in this figure. Here, we depict Takens' classification of constraint manifolds as up to three parameters $(a,b,c)$ \cite{takens1976}. The hierarchy of these manifolds comes from the fact that previous singularities are contained in each new constraint as the number of parameters is increased. For example, the constraint manifold associated to the cusp singularity contains a curve of folds (blue) on which the cusp point is located. For the swallowtail, hyperbolic umbilic and elliptic umbilic, the constraint manifolds are 3-dimensional, so here we only depict a projection of their singularities in parameter space. Their singularities consist of surfaces of folds, which contain curves of cusps (blue) and so on.
    As we detail in Section \ref{sec:Takens}, these manifolds are defined by the classification of critical points of smooth functions, see for example also the Coxeter-Dynkin diagram in \cite[Corollary 8.7]{arnold1973normal}. It is worth mentioning that fast-slow systems related to fold and cusp critical manifolds (constraints) have already been studied. Still, to the best of our knowledge, our work is the first to address the dynamics in one of the umbilics.}
    \label{fig:Hierarchy}
\end{figure}

Roughly speaking, we choose the parameter-dependent potential
\begin{equation}
	\begin{split}
		\label{unfolding}
		V(x,y,a,b,c) = \frac{1}{3}x^3 + \frac{1}{3}y^3 + axy + bx + cy
	\end{split}
\end{equation} in \eqref{start0}, where $z = (x,y) \in \R^2$ are the fast variables and $\alpha = (a,b,c) \in \R^3$ are the slow variables. This choice leads to system \eqref{hyp_umb}. The potential $V$, given by \eqref{unfolding}, is a universal unfolding of the hyperbolic umbilic germ $\frac{1}{3}(x^3 + y^3)$ in the sense of catastrophe theory \cite{michor1985cat}, see also section \ref{sec:CT} below. The precise argument for our derivation of the fast-slow system \eqref{hyp_umb} from the general form \eqref{start0} is then given in section \ref{the_choice}.

\begin{remark}
    We notice that the chosen germ $\frac{1}{3}(x^3 + y^3)$ is permutation symmetric; also compare with appendix \ref{app:universal_unfolding}. This has interesting consequences, as we mention later in remark \ref{rem:folded-sings}.
\end{remark}

The main objective of this article is to provide a detailed analysis of the dynamics of the fast-slow system \eqref{hyp_umb} near the hyperbolic umbilic singularity at the origin, in a generic case. This analysis faces several challenges as such singularity has not been studied before in the context of fast-slow systems. To accomplish our objective, we employ geometric techniques from dynamical systems theory and especially from the theory of fast-slow systems \cite{kuehn2015multiple}. Our analysis is based on the blow-up technique, which has been used in numerous cases to desingularize fast-slow systems, e.g.\ \cite{krupa2001extending,szmolyan2001canards,krupa2001transcritical,broer2013cusp,kojakhmetov2016cusp,hjk2021survey,de2021canard}. We further mention that fast-slow systems and the blow-up technique are also useful in the context of regularization of Filipov (i.e. piece-wise smooth) systems \cite{panazzolo2017regularization,llibre2009study}. For now, much attention has been given to Filipov systems with one or two-dimensional discontinuity sets, and where fold and cusp catastrophes have already been shown to play an important role \cite{jeffrey2011geometry,perez2023slow,colombo2012bifurcations}. Although new phenomena appear in the context of Filipov systems, as one considers higher dimensional discontinuity sets, Takens' classification could also be relevant. Particularly, it may be interesting to know if there is a class of Filipov systems which after regularization lead to gradient-type fast-slow systems.

For the dynamics of \eqref{hyp_umb} near the hyperbolic umbilic, we focus on the continuation of attracting slow manifolds, denoted by $\mathcal{S}^{a}_{\eps}$, under the flow of the system. Attracting slow manifolds are obtained by Fenichel theory as ``perturbations'' from the singular limit $\eps = 0$ of the attracting region of the critical manifold $\mathcal{S}_{0}=\left\{ \nabla_z V(z,\alpha)=0\right\}$; for Fenichel theory we refer the reader to \cite{fenichel1979,wiggins1994normally,jones1995,kuehn2015multiple}. The aim is to analyze how attracting slow manifolds $\mathcal{S}^{a}_{\eps}$ of \eqref{hyp_umb} evolve near the hyperbolic umbilic singularity. Therefore, our analysis of \eqref{hyp_umb} will be of local nature. For sufficiently small $\nu > 0$, we define an entry- and exit section by
\begin{equation}
	\begin{split}
		\Delta^{\en} = \{ (x,y,a,b,c)\in\mathbb R^5\, | \, y = -\nu \}, \qquad \Delta^{\ex} = \{ (x,y,a,b,c)\in\mathbb R^5\, | \,x + y = 2\nu \}.
	\end{split}
\end{equation}
Let $J$ denote the closed line segment $\{ (x,y,a,b,c)\in\mathbb R^5\, | \, x \ge 0, y \ge 0, a = b = c = 0 \} \cap \Delta^{\ex}$ and denote by $j_x$ and $j_y$ the singletons containing the intersection of $J$ with the $x$-axis, respectively $y$-axis. Refer to figure \ref{fig:main_thm} for a sketch. Then, we track slow flow trajectories contained in (a choice of) an attracting slow manifold $\mathcal{S}^{a}_{\eps}$, which start at $\Delta^{\en}$ and come close to the hyperbolic umbilic singularity. These trajectories correspond to initial conditions in a subset $I_\eps \subset \Delta^{\en} \cap \mathcal{S}^{a}_{\eps}$ in the context of the following main result, which shows that these trajectories transition from the slow regime onto the fast regime and fan out. This fanning-out behavior is observed at the exit section $\Delta^{\ex}$. A sketch of this result is given in figure \ref{fig:main_thm} and the singular limit $\eps = 0$ is illustrated in figure \ref{fig:main_thm_sing_limit}.

\begin{thm}[Attracting slow manifolds jump and fan out at the hyperbolic umbilic]
	\label{thm:main_thm}
	Consider the fast-slow system \eqref{hyp_umb} and assume $g_b(0)>0$ and $ g_c(0) > 0$. For sufficiently small $\eps > 0$, there exist open sets $I_\eps \subset \Delta^{\en} \cap \mathcal{S}^{a}_{\eps}$ such that the transition maps $\Pi\colon I_\eps \to \Delta^{\ex}$ induced by the flow of \eqref{hyp_umb} are well-defined and satisfy the following assertions:
	\begin{enumerate}[(1)]
		\item $I_\eps$ and $\Pi(I_{\eps})$ are 2-dimensional manifolds and $I_{\eps}$ converges in Hausdorff distance to a single point $p$ for $\eps \to 0$.
		\item There exist open sets $L^x_{\eps},R_{\eps},L^{y}_{\eps}$ such that $I_{\eps} = L^x_{\eps} \cup R_{\eps} \cup L^{y}_{\eps}$ and $\Pi(R_\eps) \to J$, $\Pi(L^{x}_{\eps}) \to j_x$, $\Pi(L^{y}_{\eps}) \to j_y$ in Hausdorff distance for $\eps \to 0$.
		\item For $(x,y,a,b,c) \in \Pi(I_\eps)$ it holds that $a = O(\eps^{1/3}), b = O(\eps^{2/3}), c = O(\eps^{2/3})$.
		\item If $g_a(x,y,a,b,c,0) = a \cdot h(x,y,a,b,c)$ for sufficiently smooth $h$, then $a_i = O(\eps^{1/3})$ for $(x_i,-\nu,a_i,b_i,c_i) \in I_{\eps}$.
	\end{enumerate}
\end{thm}

\begin{figure}[h]
	\centering
	\begin{overpic}[width=.65\textwidth]{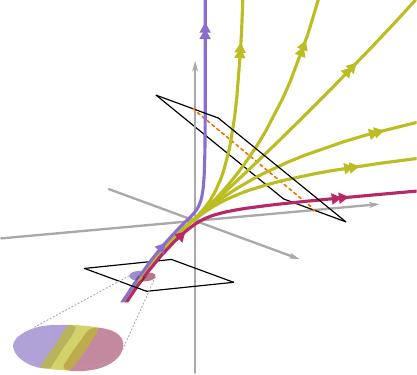}
		\put(73,26){\textcolor{gray}{$a$}}
		\put(90,38){\textcolor{gray}{$x$}}
		\put(44,75){\textcolor{gray}{$y$}}
		\put(59,59){\textcolor{orange}{$J$}}
		\put(73,36){\textcolor{orange}{$j_x$}}
		\put(43,67){\textcolor{orange}{$j_y$}}
		\put(24,28){$\Delta^{\en}$}
		\put(33,61){$\Delta^{\ex}$}
		\put(39,23){\textcolor{gray}{$I_{\eps}$}}
		\put(6,6){$L^{y}_{\eps}$}
		\put(15,5){$R_{\eps}$}
		\put(23,5){$L^{x}_{\eps}$}
	\end{overpic}
	\caption{Sketch of Theorem \ref{thm:main_thm}, describing the transition from a small region of initial conditions $I_\eps = \textcolor{q4violet}{L^{y}_{\eps}} \cup  \textcolor{q5yellow}{R_{\eps}} \cup \textcolor{q6red}{L^x_{\eps}} \subset \Delta^{\en}$ to $\Delta^{\ex}$: initial conditions in $L^{y}_{\eps}$ evolve close to the violet trajectory, $L^{x}_{\eps}$ evolves close to the red trajectory, whereas $R_{\eps}$ fans out in-between along the yellow trajectories. The trajectories were obtained by numerical integration of \eqref{hyp_umb} with $\eps = 10^{-3}$, $g_a = -1, g_b = 2, g_c = 1$ and suitably chosen initial conditions, such that they approach the hyperbolic umbilic at the origin and fan out. Most randomly sampled initial conditions near $I_{\eps}$ escape close to the $x$- or $y$-axis and only a few fan out in-between. This indicates that $R_{\eps}$ is small. The variables $b,c$ are suppressed in this sketch.}
	\label{fig:main_thm}
\end{figure}

\begin{remark}
    A relevant observation is that the hyperbolic umbilic does not present canard trajectories under the conditions of Theorem \ref{thm:main_thm}. 
\end{remark}

At this point, we provide a brief sketch of the proof to give an overview of the upcoming steps.

\begin{proof}[Sketch of the proof of Theorem \ref{thm:main_thm}]
	The first step is to understand the slow subsystem, which is the content of section \ref{sec:slow_singular_limit}. In particular, we study the so-called desingularized slow flow, which was already studied to some extent by Takens in \cite{takens1976}. The desingularized slow flow generically has a non-hyperbolic equilibrium at the hyperbolic umbilic point. Our study is then continued by blowing up the hyperbolic umbilic point in the fast formulation \eqref{hyp_umb_f} to a 5-d unit sphere $\mathbb{S}^5$. The desingularized slow flow \eqref{sf:desing}, which determines singular limit dynamics on the critical manifold, will be embedded into the blow-up analysis, which is carried out in section \ref{sec:desing_sf_blowup}.
	The transition map $\Pi$ and its properties are obtained by combining the results from an entry chart $\kappa_{-y}$, the rescaling chart $\kappa_{\eps}$ and an exit chart $\kappa_{\ex}$. The charts and their notation will be set up in the beginning of section \ref{sec:blowup}. In the entry chart, section \ref{sec:entry}, the main work is to obtain the flow nearby the hyperbolic umbilic in attracting center manifolds, which contain attracting slow manifolds. This will be accomplished by a center manifold reduction close to a point called $\zeta_2$, which arises from the desingularized slow flow, see Lemma \ref{lemma:CM_red}. In the rescaling chart $\kappa_\eps$, the key is to identify the relevant trajectories on the blow-up sphere, that lead attracting slow manifolds from $\zeta_2$ through the blow-up space. By perturbation arguments, this reduces to the analysis of two decoupled Riccati equations, which is a similar situation to the non-degenerate fold \cite{krupa2001extending}. This gives the flow close to the hyperbolic umbilic singularity, i.e.\ in the corresponding region in the rescaling chart. To enter this region, the transition $\Pi_1$ in section \ref{sec:entry_trans}, precisely the structure of \eqref{CM_red2}, requires that we limit the initial conditions suitably, which will eventually lead to $I_{\eps}$ shrinking to a point as $\eps \to 0$.
	In Proposition \ref{res:unique_gamma} we obtain an invariant 2-manifold $\Gamma$ on the blow-up sphere, leading the relevant trajectories in attracting slow manifolds $\mathcal{S}^{a}_\eps$ towards the fast regime in the exit chart $\kappa_{\ex}$. In the exit chart $\kappa_{\ex}$, we encounter three hyperbolic equilibria $q_{4,5,6}$, to which trajectories in $\Gamma$ connect. These three equilibria give three distinct forward asymptotic behavior of trajectories in $\Gamma$, which lead to the division of $I_\eps$ into three regions $L^x_{\eps}, R_{\eps}, L^{y}_{\eps}$.
	 Further, the three equilibria $q_{4,5,6}$ organize the transition onto the fast regime and lead to fanning-out, see Proposition \ref{ex:trans}. To obtain the fanning-out, the blown-up vector field is integrated in the invariant plane corresponding to the fast subsystem. Essentially, a $45^\circ$ rotation is included in $\kappa_{\ex}$, to have all escape directions visible in a single chart. The scalings of $\Pi(I_{\eps})$ originate from the convergence of $I_{\eps}$ to a point as $\eps \to 0$. The detailed proof is finished in section \ref{sec:main_proof}. For an overview of the blow-up analysis, we refer to figure \ref{fig:complete_transition}, in which all ingredients to the proof are visible.
\end{proof}

\begin{remark}
	Theorem \ref{thm:main_thm} does not cover the fast approach of trajectories towards attracting slow manifolds. However, from the fast subsystem one can roughly infer the basin of attraction of attracting slow manifolds. Compare to figure \ref{fig:main_thm_sing_limit}. Taking the stable manifold of an attracting slow manifold into account, we remark that Theorem \ref{thm:main_thm} does in this case only apply to suitable ``fast initial conditions'' located in fibers, whose base passes through $I_{\eps}$.
\end{remark}

\begin{figure}[h]
	\centering
	\begin{overpic}[width=.8\textwidth]{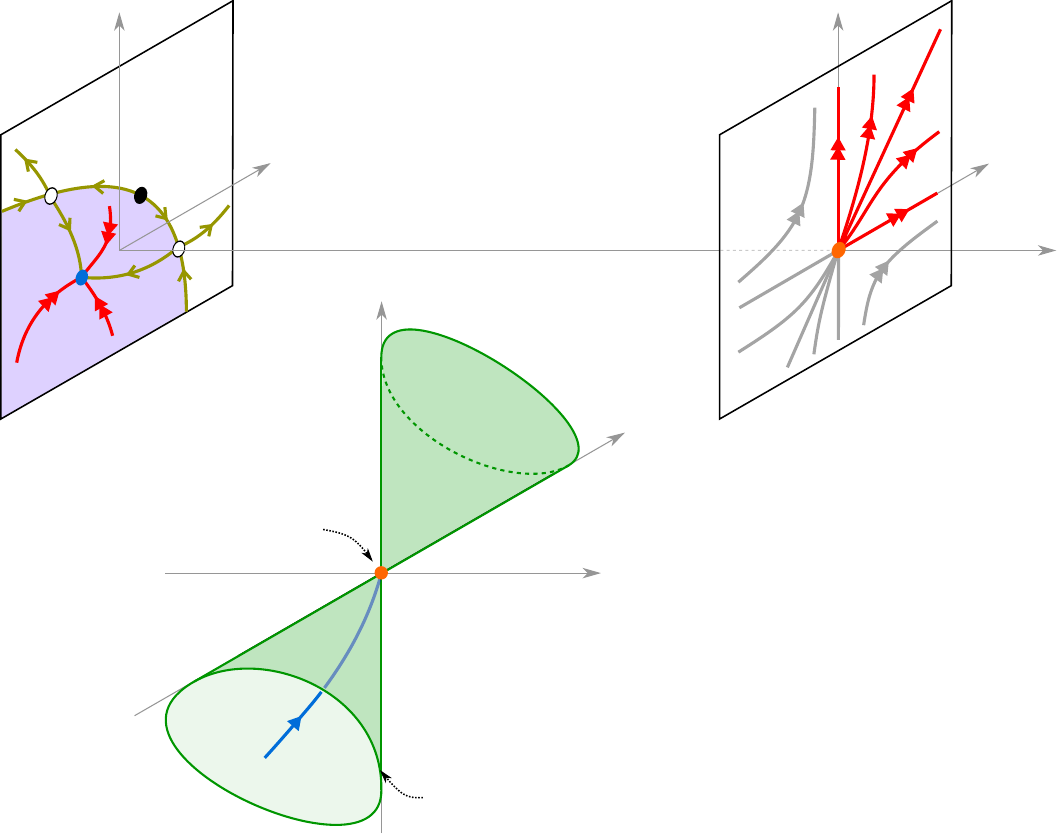}
		\put(93,57){\textcolor{gray}{$(a,b,c)$}}
		\put(94.5,63){\textcolor{gray}{$x$}}
		\put(81,77){\textcolor{gray}{$y$}}
		\put(27,63){\textcolor{gray}{$x$}}
		\put(12.7,77){\textcolor{gray}{$y$}}
		\put(57,23){\textcolor{gray}{$a$}}
		\put(60,37){\textcolor{gray}{$x$}}
		\put(33,48){\textcolor{gray}{$y$}}
		\put(28,56){\small \textcolor{gray}{$\leadsto$ appropriate slow flow $\leadsto$}}
		\put(41,18){\small \textcolor{gray}{slow subsystem}}
		\put(68,36){\rotatebox{30}{\small fast subsystem at $(0,0,0)$}}
		\put(0,36){\rotatebox{30}{\small fast subsystem at $(a,b,c)$}}
		\put(6,28){\small \textcolor{orange}{hyperbolic umbilic}}
		\put(41.5,3){\small \textcolor{dark-green}{double cone of singularities of $\mathcal{S}_0$}}
	\end{overpic}
	\caption{Singular limit sketch of \eqref{hyp_umb}. The upper two planar systems are fast subsystems for particular fixed choices of slow variables, i.e.\ ``parameters'' $(a,b,c)$. The right one is the fast subsystem at the hyperbolic umbilic, i.e.\ $(a,b,c) = (0,0,0)$. The slow subsystem on the critical manifold can be parametrized by $(x,y,a)$ and the green double cone represents the singularities. Fast segments of candidate trajectories for \eqref{hyp_umb} are shown in red. The blue trajectory is a particular slow candidate trajectory in the attracting part of the critical manifold, which approaches the hyperbolic umbilic singularity, depicted in orange. The blue segment can be concatenated with any of the red segments emanating from the hyperbolic umbilic, which corresponds to Theorem \ref{thm:main_thm} in the singular limit. The violet region depicts a basin of attraction for the blue sink, which is an instance of the attracting region of the critical manifold. Here, one can think of $(a,b,c)$ chosen such that this sink corresponds to the starting point of the blue slow candidate trajectory. Further, three fast candidate trajectories approaching the attracting critical manifold are depicted.}
	\label{fig:main_thm_sing_limit}
\end{figure}
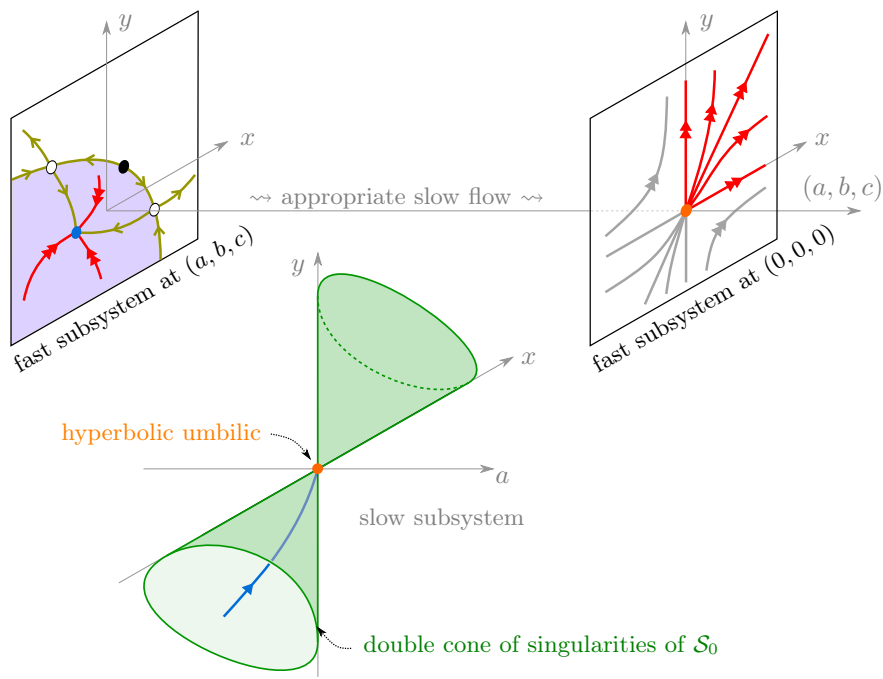

To motivate and illustrate the behavior of the system \eqref{hyp_umb} and in particular Theorem \ref{thm:main_thm}, we present a singular limit sketch of \eqref{hyp_umb} in figure \ref{fig:main_thm_sing_limit}, in which two particular fast subsystems and a particular slow flow in the slow subsystem are depicted. Note that \eqref{hyp_umb} has a 3-dimensional critical manifold and 2-dimensional fast subsystems. A detailed analysis of the slow and fast subsystems is carried out in sections \ref{sec:slow_singular_limit} and \ref{sec:fast_subsystem}. Figure \ref{fig:main_thm_sing_limit} illustrates an appropriate way of combining the fast- and slow subsystem, i.e.\ the concatenation of candidate trajectories.

\begin{remark}
	The assumption $g_b(0) > 0, g_c(0) > 0$ of Theorem \ref{thm:main_thm} will be motivated in the singular limit by Lemma \ref{sf:origin_lin}, which shows that there exists a candidate trajectory $\sigma$ approaching the hyperbolic umbilic point from inside the attracting critical manifold. This trajectory $\sigma$ is depicted in blue in the lower cone in figure \ref{fig:main_thm_sing_limit}. In fact, the idea of the blow-up analysis is to track trajectories close to $\sigma$ through the blow-up space. 
\end{remark}

The following remark points out an interesting observation, which will be apparent from the upcoming blow-up analysis and follows from the permutation symmetry of the chosen germ $\frac{1}{3}(x^3 + y^3)$ of the potential function $V$. It is related to the behavior of slow flow trajectories jumping at folds, away from the hyperbolic umbilic singularity.

\begin{remark}[Fast-slow systems with folded singularities appear in the blow-up of \eqref{hyp_umb_f}]\label{rem:folded-sings}
	In the blow-up analysis of \eqref{hyp_umb_f} in section \ref{sec:blowup}, folded saddles or folded centers appear generically on the blow-up sphere. This observation is made in section \ref{sec:desing_sf_blowup} and analyzed further in section \ref{sec:further}. These folded singularities seem to organize the transition of attracting slow manifolds at fold singularities of \eqref{hyp_umb_f}. Their further analysis does not lie in the scope of this article. However, these folded singularities, especially the folded centers, deserve future attention. In the folded saddle scenario, existing results can be applied, e.g.\ singular canards perturb to maximal canard solutions \cite{szmolyan2001canards}. 
\end{remark}

The article is structured as follows. In section \ref{sec:hyp_umb_singularity} we will briefly recall fast-slow terminology, formalize what we mean by a hyperbolic umbilic singularity in a fast-slow system and derive our system \eqref{hyp_umb}. This derivation is motivated by ``Takens' program'' \cite{takens1976}. For convenience we provide Takens' definition of constrained differential equations in appendix \ref{app:takens}. Sections \ref{sec:slow_singular_limit} and \ref{sec:fast_subsystem} contain an analysis of the slow and fast subsystem of \eqref{hyp_umb}. The analysis of the ``desingularized'' or ``reduced'' slow flow $Y$ in section \ref{ss:sfdesing} is fundamental for the upcoming blow-up analysis of system \eqref{hyp_umb}. The blow-up analysis is presented in section \ref{sec:blowup}, in which the origin of \eqref{hyp_umb_f} is blown up to a 5-d sphere. We remark that also the reduced slow flow $Y$ is blown-up and embedded into the blow-up analysis, which is carried out in section \ref{sec:desing_sf_blowup}. The combined analysis from three charts leads to the proof of Theorem \ref{thm:main_thm} in section \ref{sec:main_proof}. Eventually, in section \ref{sec:further} we present further observations made in the blow-up analysis, concerning the folded singularities on the blow-up sphere and the blown-up desingularized slow flow.

\section{The hyperbolic umbilic singularity and Takens' program}
\label{sec:hyp_umb_singularity}

In this section we provide the basic terminology and the context of our study. We begin by recalling some basic notion of fast-slow systems. For context, we then describe what we call ``the Takens' program'' and its relationship with catastrophe theory and fast-slow systems. Ultimately this explains the reason to choose \eqref{hyp_umb} for our analysis.

\subsection{Fast-slow systems terminology}
\label{fs_terminology}

Throughout this document, we restrict ourselves to gradient-like fast-slow systems. Consider a general fast-slow system of the form
\begin{equation}
	\begin{split}
		\label{start}
		\eps \dot{z} &= \nabla_z V(z,\alpha) + O(\eps)\\
		\dot{\alpha} &=  g(z,\alpha,\eps),
	\end{split}
\end{equation}
where $V:\mathbb R^n\times\mathbb R^r\to\mathbb R$ is smooth, $0<\eps\ll1$ is a small parameter, and $g:\mathbb R^n\times\mathbb R^r\times\mathbb R\to\mathbb R^r$ is smooth. Later on we shall restrict to the case $n=2$, but for now this is not necessary.

We briefly recall terminology from the theory of fast-slow systems \cite{kuehn2015multiple,jones1995}. The following two subsystems of \eqref{start} are fundamental:

\noindent\begin{minipage}{.5\textwidth}
\begin{equation}
	\begin{split}
		\label{start:slow}
		0 &= \nabla_z V(z,\alpha)\\
		\dot{\alpha} &= g(z,\alpha,0),
	\end{split}
\end{equation}
\end{minipage}%
\begin{minipage}{.5\textwidth}
\begin{equation}
	\begin{split}
		\label{start:fast}
		z' &= \nabla_z V(z,\alpha)\\
		\alpha' &= 0.
	\end{split}
\end{equation}
\end{minipage}
System \eqref{start:slow} is the slow subsystem (also called reduced problem), obtained from \eqref{start} by putting $\eps = 0$. System \eqref{start:fast} is the fast subsystem (also called layer problem). Let $\tau$ denote the time in \eqref{start}, also called the slow time scale. Then \eqref{start:fast} is obtained by a time change in \eqref{start} onto the fast timescale $t = \tau/\eps$ and putting $\eps = 0$. The critical manifold $\mathcal{S}_0 = \{ \nabla_z V = 0 \}$ appears as equilibria of the fast subsystem and as constrained phase space of the slow subsystem. Of central importance are the normally hyperbolic points of the critical manifold $\mathcal{S}_0$. In our setting, a point $p \in \mathcal{S}_0$ is normally hyperbolic if all eigenvalues of the Hessian $\D^2_z V$ have nonzero real part. By Schwarz's theorem the Hessian $\D^2_z V$ is symmetric, so its eigenvalues are real. The eigenvalues determine the stability type with respect to the fast subsystem. In this way, we shall denote the attracting (resp.\ repelling, resp.\ saddle-type) region of the critical manifold as $\mathcal{S}_0^\txta$ (resp.\ $\mathcal{S}_0^\txtr$, resp.\ $\mathcal{S}_0^\txts$). Accordingly $p\in\mathcal{S}_0$ is non-hyperbolic, if at least one eigenvalue of $\D^2_z V$ is zero.

Away from non-hyperbolic points of $\mathcal{S}_0$, one can employ Fenichel theory \cite{fenichel1979,wiggins1994normally,jones1995,kuehn2015multiple} to analyze \eqref{start}. Roughly speaking, Fenichel theory allows to conclude that normally hyperbolic regions of the critical manifold $\mathcal{S}_0$ persist, for sufficiently small $\eps > 0$, in the form of slow manifolds with the corresponding stability type. The analysis of \eqref{start} close to non-hyperbolic points of $\mathcal{S}_0$ is considerably more challenging, since Fenichel theory is not applicable. In our context, non-hyperbolic points of $\mathcal{S}_0$ will be referred to as singularities of \eqref{start}.
It is precisely the singularities, which in general give rise to certain phenomena, for example jumps, canards and relaxation oscillation, see for example \cite{kuehn2015multiple,hjk2021survey,de2021canard,wechselberger2020geometric} and references therein.

By the implicit function theorem, at a normally hyperbolic point, $\mathcal{S}_0$ is locally given as a graph over the slow $\alpha$-variables. This determines a flow in the normally hyperbolic regions of $\mathcal{S}_0$. This corresponds to the flow arising from the so called ``reduced'' or ``desingularized'' slow flow, also see section \ref{ss:sfdesing}. The fast subsystem gives dynamics off the critical manifold. In the singular limit $\eps = 0$, the fast- and slow subsystem \eqref{start:fast},\eqref{start:slow} are combined by concatenation of trajectories: At singularities of $\mathcal{S}_0$ one allows to concatenate fast and slow trajectory segments. This singular limit analysis is usually the guiding starting point for further analysis. This is the content of section \ref{sec:slow_singular_limit} and \ref{sec:fast_subsystem} in this article. We now come to the hyperbolic umbilic singularity, the choice of \eqref{hyp_umb} and its motivation.

\subsection{Takens' program}\label{sec:Takens}
As pointed out above, one is interested in singularities of the critical manifold $\mathcal{S}_0 = \{ \nabla_z V(z,\alpha) = 0 \}$. The question which ``types'' of singularities of $\mathcal{S}_0$ generically occur is in fact answered by catastrophe theory. For \eqref{start} with $\alpha \in \R^r$, $r \le 5$, we generically have that the singularities of $\mathcal{S}_0$ are Thom's seven elementary catastrophes plus four additional catastrophes of codimension 5, see \cite{michor1985cat}. 
In the article \cite{takens1976}, Takens introduced the notion of ``constrained differential equations'' and analyzed their behavior. We refer to appendix \ref{app:takens} for Takens' precise definition. In fast-slow terminology, Takens roughly studied the singular limit of \eqref{start}, in particular \eqref{start:slow}, and analyzed the effect of singularities of $\mathcal{S}_0$ on the slow flow restricted to the attracting part of the critical manifold. Takens carried out a local classification of the slow flow near singularities up to topological equivalence for $r \le 2$. Furthermore, Takens provided a list of ``normal forms'' for the singular limit of \eqref{start} near singularities for $r \le 3$. We refer to this as ``Takens' program'', which is in fact based on catastrophe theory, and serves as motivation for our analysis.

\subsection{Catastrophe theory framework}\label{sec:CT}
It will be convenient to briefly recall notions from catastrophe theory, for details we refer to \cite{michor1985cat,demazure2000,arnold1973normal}. 

We begin with some notations: Let $\mathcal{E}_n$ denote the local ring of germs (of smooth functions) at $0 \in \R^n$. Let $\mathfrak{m}_n$ denote the ideal of germs vanishing at 0. It follows that $\mathfrak{m}_n$ is the maximal ideal in $\mathcal{E}_n$, which means that there does not exist a proper ideal of $\mathcal{E}_n$ strictly containing $\mathfrak{m}_n$. Let $\mathfrak{m}_n^k$ denote the set of germs $f \in \mathfrak{m}_n$ such that the $(k-1)$-jet of $f$ vanishes, that is $\mathfrak{m}_n^k = \{ f \in \mathfrak{m}_n\, \rvert \,  j^{k-1} f = 0\}$. Let $\Delta(f)$ denote the Jacobi ideal of $f \in \mathfrak{m}_n^2$, which is the ideal generated by the partial derivatives of $f$ over $\mathcal{E}_n$, formally $\Delta(f) = \< \partial_1 f ,\dots, \partial_n f \>_{\mathcal{E}_n}$. 

Next, let the codimension of $f$ be defined by the dimension of $\mathfrak{m}_n / \Delta(f)$ as an $\R$-vector space. Further, two germs $f,g \in \mathfrak{m}_n$ are right-equivalent, if they agree up to a local coordinate change, that is, there exists a local diffeomorphism $\psi$ with $\psi(0) = 0$ such that $f = g \circ \psi$. 

Finally, let us recall the notion of unfoldings. An $r$-parameter unfolding $F$ of $f \in \mathfrak{m}_n$ is a germ of a smooth function $F\colon \R^n \times \R^r\to \R$, for which $F(\cdot,0) = f$. An $r$-parameter unfolding $F$ of $f$ is \emph{versal}, if any other $s$-parameter unfolding $G$ of $f$ can be expressed in terms of $F$ (``$G$ is induced by $F$''), that is, there exist smooth germs $\phi \colon (\R^{n+s},0) \to (\R^{n+r},0), \bar{\phi}\colon (\R^{s},0) \to (\R^{r},0), \mu\colon (\R^{s},0) \to (\R,0)$ such that $\phi(\cdot,0) = \id_{\R^{n}}$, $\pr_2 \circ \phi = \bar{\phi} \circ \pr_2$ and $G = F \circ \phi + \mu \circ \pr_2$. Here we denote by $\pr_2$ the natural projection onto the unfolding parameters, which discards the spatial variables. Two $r$-parameter unfoldings $F,G$ of $f$ are isomorphic, if $G$ is induced by $F$ and $\phi$ and $\bar{\phi}$ are local diffeomorphisms.
In fact, a germ $f \in \mathfrak{m}_n^2$ has a versal unfolding if and only if its codimension is finite \cite{michor1985cat}. A versal unfolding with $r$-parameters is universal, if the number $r$ is minimal. In fact, the minimal $r$ is given by the codimension of $f$. We point out that universal unfoldings of a fixed germ $f$ are not unique, for an example refer to appendix \ref{app:universal_unfolding}. Further, note that universal unfoldings are also called miniversal unfoldings in slightly different contexts.

A fundamental result in catastrophe theory is the classification up to right-equivalence of germs in $\mathfrak{m}_n^2$ with low codimension. Here we present a version of this result up to codimension 4, which is also known under the name ``Thom's seven elementary catastrophes''. Recall that a germ $f$ has a non-degenerate critical point at the origin if $\nabla f(0) = 0$ and the Hessian $\D^2 f(0)$ is invertible. Recall as well that such $f$ with non-degenerate critical point can locally be written in simple form by the Morse Lemma.

\begin{thm}[Thom's seven elementary catastrophes {{\cite{michor1985cat,arnold1973normal,broecker1975germs}}}]
	\label{thom}
	Let $f \in \mathfrak{m}_n^2$ with codimension $\le 4$. Then $f$ is either non-degenerate (i.e.\ Morse) or right-equivalent to one of the following germs, up to sign and up to addition of a non-degenerate quadratic form in other variables.
		\begin{table}[H]
			\centering
			\begin{tabular}{lllc}
				\toprule
				catastrophe & germ & universal unfolding & codimension \\
				\midrule
				fold & $x^3$ & $x^3 + ax$ & 1\\
				cusp & $x^4$ & $x^4 + ax^2 + bx$ & 2\\
				swallowtail & $x^5$ & $x^5 + ax^3 + bx^2 + cx$ & 3\\
				hyperbolic umbilic & $x^3 + y^3$ & $x^3 + y^3 + axy + bx + cy$ & 3\\
				elliptic umbilic & $x^3 - xy^2$ & $x^3 - xy^2 + a(x^2 + y^2) + bx + cy$ & 3\\
				butterfly & $x^6$ & $x^6 + ax^4 + bx^3 + cx^2 + dx$ & 4\\
				parabolic umbilic & $x^2y + y^4$ & $x^2y + y^4 + ax^2 + by^2 + cx + dy$ & 4\\
				\bottomrule
			\end{tabular}
			\caption{The seven elementary catastrophes. The third column contains a choice of universal unfolding for each corresponding germ in the second column.}
			\label{the7}
		\end{table}
\end{thm}

\begin{remark}
	Note that Theorem \ref{thom} says that for any germ $f$ with $f(0) = 0, \nabla f (0) = 0$ and codimension $\le 4$, we can find local coordinates $(x_1,\dots,x_n)$ centered at the origin such that $f$ is of the form
	\begin{equation}
		\begin{split}
			\pm h(x_1,\dots,x_k) + \sum_{i > k} \pm x_i^2,
		\end{split}
	\end{equation}
	where either $h$ is (up to right-equivalence) a germ appearing in the above list with $k \in \{1,2\}$, or $h = 0$ and $k = 0$.
\end{remark}

\subsection{Singularities of \eqref{start}}
In view of the above discussion, we propose the following natural definition to classify singularities of \eqref{start}. More precisely, we want to classify the singularities of the critical manifold $\mathcal{S}_0 = \{ \nabla_z V(z,\alpha) = 0 \}$ of \eqref{start} without considering any dynamics. Due to the scope of this paper, we are particularly interested in the catastrophes of codimension $3$, especially the hyperbolic umbilic singularity.

\begin{defn}[Singularities of the critical manifold of \eqref{start}]
	\label{sf:singularity}
	Consider the fast-slow system \eqref{start} with $z \in \R^n, \alpha \in \R^r$. We say that \eqref{start} has a $\ast$ singularity at $(\bar{z},\bar{\alpha}) \in \mathcal{S}_0$, if $\mathcal{S}_0$ is locally an $r$-dimensional manifold at $(\bar{z},\bar{\alpha})$ and the germ $f_{\bar{z},\bar{\alpha}} := V(\cdot + \bar{z},\bar{\alpha}) - V(\bar{z},\bar{\alpha}) \in \mathfrak{m}_n^2$ is right-equivalent to
	\begin{equation}
		\begin{split}
			\pm h(z_1,\dots,z_k) + \sum_{i > k} \pm z_i^2,
		\end{split}
	\end{equation}
	where $h$ is a germ for the $\ast$ catastrophe in table \ref{the7}, and $k \in \{1,2\}$ accordingly.
\end{defn}

To this definition, we shall give the following remarks.
\begin{remark}\leavevmode 

	\begin{enumerate}[(a),noitemsep]
		\item Definition \ref{sf:singularity} does not concern dynamics on or nearby $\mathcal{S}_0$.
		\item Normally hyperbolic points of $\mathcal{S}_0$ correspond to germs $f_{\bar{z},\bar{\alpha}}$ that are non-degenerate (i.e. Morse).
		\item If $r \le 4$, for generic $V$ the critical manifold is indeed $r$-dimensional and the only singularities that occur are the catastrophes in table \ref{the7} \cite{michor1985cat}. A similar statement holds for $r \le 5$ by extending table \ref{the7} to codimension 5 catastrophes \cite{arnold1973normal,michor1985cat}.
		\item The number $k$ is given by the corank of $f_{\bar{z},\bar{\alpha}}$, which is the corank of its Hessian at the origin.
		\item Definition \ref{sf:singularity} does not cover non-degeneracy conditions for the singularity. An example for this is the 1-d choice $V(z) = z^3$, which leads to the set $\{ z = 0\}$ of fold singularities.
		\item In the case that a choice of $V$ leads to a transcritical singularity or a pitchfork singularity, the critical manifold fails to be a manifold near the singularities.
		\item Finite determinacy of germs can lead to conditions for singularities on the level of jets. An example for this is the hyperbolic umbilic germ $f = x^3 + y^3$. One can check that this germ is 3-determined, which implies that any choice of $V$, such that $f_{\bar{z},\bar{\alpha}}$ has the same 3-jet as $f$ and the critical manifold is indeed locally a manifold, gives a hyperbolic umbilic singularity at $({\bar{z},\bar{\alpha}})$. Note that one can also take the 3-determined germ $f = x^3 + xy^2$ to represent the hyperbolic umbilic singularity, among many other possible choices.
	\end{enumerate}
\end{remark}

\subsection{The choice of \eqref{hyp_umb}}
\label{the_choice}
We are interested in the hyperbolic umbilic singularity, which is the first catastrophe in table \ref{the7}, for which two fast directions are needed. Our choice is motivated by Takens' work \cite[(4.10)]{takens1976} and we will argue similarly to Takens by giving the following ``derivation'' for the fast-slow system \eqref{hyp_umb}:

Let us call $V(z,\alpha)$ from \eqref{start} a (parameter-dependent) potential. The catastrophes of codimension 3 from Theorem \ref{thom} occur at isolated points for generic 3-parameter potentials. Thus we restrict to \eqref{start} with $\alpha = (a,b,c) \in \R^3$. There exists an open and dense set of potentials $\mathcal{J}$, such that $V \in \mathcal{J}$ in fact gives a versal unfolding of the singularity it unfolds and the critical manifold is indeed a 3-d critical manifold \cite[(9.17)]{michor1985cat}. Assume that such generic potential $V$ unfolds a hyperbolic umbilic singularity at the origin. Then the germ $f := V(\cdot,0) - V(0,0) \in \mathfrak{m}_n^2$ is right-equivalent to $g := x^3 + y^3 + \sum_{2<j\le n} \pm z_j^2$, where $z_j$ are, if $n > 2$, the remaining coordinates of $\R^n$. In other words, we have $g = f \circ \psi$ for a local diffeomorphism $\psi$. We remark that at this point we make a choice, namely, we choose the germ $x^3 + y^3$ to represent the hyperbolic umbilic. This choice is not canonical and it influences the family of fast subsystems we eventually obtain. For now we neglect the fast dynamics, but we will argue for our choice after Remark \ref{guckenheimer}.

The potential $V$ is (up to shift) a versal unfolding of $f$. A (uni-)versal unfolding of $g$ is given by $G := x^3 + y^3 + a xy + bx + cy + \sum_{2 < j \le n} \pm z_j^2$.  Then one can check that $V(\psi(\cdot),\cdot)$ is a versal unfolding of $g$ and any two versal unfoldings with the same number of parameters are isomorphic \cite{michor1985cat}. Thus there exists a local diffeomorphism $\phi \colon (\R^{n + 3},0) \to (\R^{n + 3},0)$ with $\phi(\cdot,0) = \psi$ and $\pr_2 \circ \phi = \bar{\phi} \circ \pr_2$ for a local diffeomorphism $\bar{\phi}\colon (\R^3,0) \to (\R^3,0)$, such that $V = G \circ \phi$ up to addition of a constant and a smooth function in the parameters. Hence there are local coordinates $(x,y,z_3,\dots,z_n,a,b,c)$ centered at 0, such that the critical manifold $\{ \nabla_z V = 0\} \subset \R^{n+3}$ of \eqref{start} near the origin can be written as the set satisfying:
\begin{equation}
	\begin{split}
		0 &= 3x^2 + ay + b  \\
		0 &= 3y^2 + ax + c  \\
		0 &= \pm2 z_3 \\
		&\vdots\\
		0 &= \pm2 z_n.
	\end{split}
\end{equation}
For the slow subsystem the $z_j$ directions are irrelevant. In the family of fast subsystems the $z_j$ directions are hyperbolic directions. Thus, by a center manifold argument, we discard the $z_j$ variables. This leads to the slow subsystem we are interested in. Again, we point out that this slow subsystem choice is carried out similarly in Takens' program \cite{takens1976}. Finally, we take the fast-slow system of the form \eqref{start}, which contains the above constructed slow subsystem.

For convenience we scale the germ $x^3 + y^3$ to $x^3/3 + y^3/3$ and hence consider a universal unfolding given by $V := x^3/3 + y^3/3 + axy + bx + cy$.
This leads to a fast-slow system \eqref{hyp_umb}, i.e.\ a fast-slow system of the form
\begin{equation}
	\begin{split}
		\label{hyp_umb2}
		\eps\dot{x} &= x^2 + ay + b + O(\eps) \\
		\eps\dot{y} &= y^2 + ax + c + O(\eps) \\
		\dot{a} &= g_a(x,y,a,b,c,\eps) \\
		\dot{b} &= g_b(x,y,a,b,c,\eps) \\
		\dot{c} &= g_c(x,y,a,b,c,\eps),
	\end{split}
\end{equation}
which precisely has, at the origin, a hyperbolic umbilic singularity in the sense of Definition \ref{sf:singularity}. 

\begin{remark}
We observe that \eqref{hyp_umb2} only has higher order terms multiplying $\varepsilon$. This is because we assume that the leading part in the fast variables comes from the universal unfolding of the hyperbolic umbilic singularity. In general, other higher-order terms do not preserve the hyperbolic umbilic singularity; see also remark \ref{guckenheimer} and the discussion following remark \ref{guckenheimer} below. Notice, however, that other higher-order terms would be irrelevant in the blow-up analysis of section \ref{sec:blowup}, as they would appear as higher-order terms in $\bar r$. 
Finally, it  would be interesting to investigate if \eqref{hyp_umb2} can serve as a fast-slow normal form.
\end{remark}

\begin{remark}
	By replacing the potential $V$ by $-V$, the attracting and repelling regions of the critical manifold swap their stability types. Therefore, on the level of fast-slow systems of the form \eqref{start}, the sign in front of $V$ cannot be neglected in general. An example for this is the cusp \cite{broer2013cusp,kojakhmetov2016cusp}. However, for the hyperbolic umbilic singularity, the setting of $-V$ is analogous to $V$, as will be apparent from the upcoming analysis.
\end{remark}

We remark that so far we neglected the effect of the above constructed local diffeomorphism on the fast dynamics, i.e.\ the family of fast subsystems \eqref{start:fast}, which is a family of gradient vector fields. First, an important remark about the fast subsystems of \eqref{hyp_umb2} seems to be in order. This remark addresses the fact that versality in the sense of catastrophe theory on the level of potential functions does in general not give versality in the sense of unfoldings of vector fields, i.e.\ on the level of flows of the associated gradient fields.

\begin{remark}[Versality in the sense of unfoldings of vector fields]
	\label{guckenheimer}
	The family of gradient systems $Z_{a,b,c}$ given by the family of fast subsystems of \eqref{hyp_umb2}, which is precisely the planar family \eqref{planar_fast}, does \emph{not} give a versal unfolding of the vector field $Z_{0,0,0} = x^2 \partial_x + y^2 \partial_y$ in the space of (gradient) vector fields. Here, we mean by versal unfolding of a vector field the following: Let $X_0$ be a vector field, then we call a family of vector fields $X_{\mu}$ containing $X_0$ versal if for any other family $Y_{\lambda}$ containing $X_0$, there exist a homeomorphism $h(\lambda) = \mu$ between the parameter spaces and a family of homeomorphisms $H_{\lambda}$, not necessarily depending continuously on $\lambda$, such that $Y_{\lambda}$ is topologically equivalent to $X_{h(\lambda)}$ via $H_{\lambda}$. The latter is also known as fiber topological equivalence \cite{kuznetsov2004}. The proof that $Z_{a,b,c}$ does not give a versal unfolding of $Z_{0,0,0}$ is given by Guckenheimer in \cite{guckenheimer1973}: small perturbations of $Z_{0,0,0}$ in the space of gradient vector fields produce phase portraits, which contain a saddle connection. However, the family $Z_{a,b,c}$ does not contain phase portraits with saddle connections, refer to the analysis of the fast subsystem in section \ref{sec:fast_subsystem}. Thus $Z_{a,b,c}$ cannot be versal in the above sense. In \cite{guckenheimer1973} Guckenheimer uses the name ``universality'' instead of versality. The minimal number of parameters for a versal unfolding of $Z_{0,0,0}$ in the above sense is 4, and a proof of this claim is contained in \cite{vegter1982}. The notion of a versal unfolding of a vector field given above is called an almost universal unfolding in \cite{vegter1982}, where ``almost'' refers to the fact that in this case $H_{\lambda}$ does not depend continuously on $\lambda$ in a neighborhood of the origin. 
\end{remark}

\begin{figure}[h]
	\begin{overpic}[width=.6\textwidth]{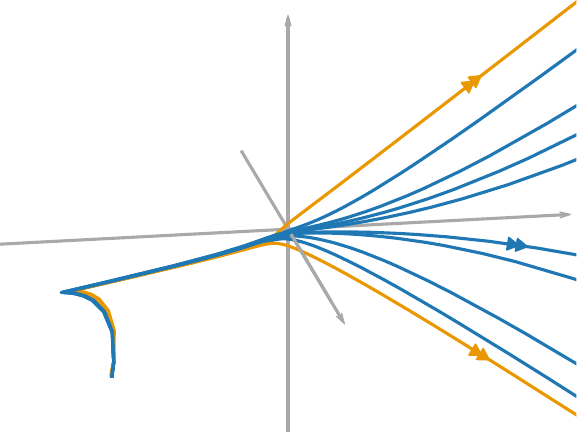}
		\put(97,40){\textcolor{gray}{$x$}}
		\put(46,71){\textcolor{gray}{$y$}}
		\put(61,17){\textcolor{gray}{$a$}}
	\end{overpic}
	\caption{Numerical integration of \eqref{start0} with $z = (x,y)$, $\alpha = (a,b,c)$, $V = y^3/3 + yx^2 + a(y^2 - x^2) + bx + cy$ and $g=(-1,1,2)$. According to Definition \ref{sf:singularity}, this system has a hyperbolic umbilic singularity at the origin, but corresponds to a different choice of germ representing a hyperbolic umbilic catastrophe, i.e.\ $y^3/3 + yx^2$, and an associated universal unfolding (see appendix \ref{app:universal_unfolding}). Ten initial conditions close to the point $(x,y,a,b,c) = (-1.5,-2,1,-2,-1)$ were chosen, which fan out near the origin. Similar to figure \ref{fig:main_thm}, most randomly sampled initial conditions close to this point would escape close to the two orange trajectories. Besides that, the fast approach towards attracting slow manifolds is visible here: the transition from fast to slow dynamics happens at the kink on the left.}
	\label{fig:alternative_unfolding}
\end{figure}

In the above derivation in section \ref{the_choice}, we have chosen the germ $x^3 + y^3$ (up to scaling) to represent the hyperbolic umbilic catastrophe, which leads to \eqref{hyp_umb2}, which is identical to \eqref{hyp_umb}. As already mentioned, in general this choice might affect the fast dynamics, more precisely the fast subsystems.
For example, we could have chosen the germ $y^3 + yx^2$, which is right-equivalent to $x^3+y^3$, see e.g. appendix \ref{app:universal_unfolding}. However, the associated gradient fields to these germs are not topologically equivalent: $x^2 \partial_x + y^2 \partial_y$ has two hyperbolic sectors and two parabolic sectors, whereas $2xy \partial_x + (3y^2 + x^2)\partial_y$ only has two hyperbolic sectors and no parabolic sectors \cite{takens1971manifolds,khesin1990}. However, by a result from Khesin \cite{khesin1990,khesin1986}, the family of fast subsystems of our chosen system \eqref{hyp_umb2} is generic in the following sense:

In \cite{khesin1990} it is shown that there are seven three-parameter families of gradient fields, whose fiber topological equivalence classes (up to orientation of the orbits) form an open and dense set in the space of three-parameter families of gradient fields, which have a degenerate critical point at the origin. Three out of these seven families unfold a hyperbolic umbilic singularity in the sense of catastrophe theory. From the upcoming analysis in section \ref{sec:fast_subsystem} it follows that the family of fast subsystems of \eqref{hyp_umb2} is equivalent to family number 5 of Theorem 1 in Khesin's article \cite{khesin1990}, since they topologically have the same bifurcation set and do not show saddle connections. Also see Corollary 2 in \cite{khesin1986}. The main point is that the fast subsystem of \eqref{hyp_umb2} is generic in the set of three-parameter families of gradient fields, which have a degenerate critical point at the origin. Further, the phase portraits in Khesin's family number 6 and family number 5 in \cite{khesin1990} differ only at parameter values zero. At vanishing parameters, the phase portraits of these two families differ precisely in the way the phase portraits of the gradients of the germs $x^3+y^3$ and $y^3 + yx^2$ above differ. Hence, the analysis of the fast subsystem in section \ref{sec:fast_subsystem} for parameters $(a,b,c) \neq (0,0,0)$ does also apply to Khesin's family number 6. Moreover, we conjecture that the choice of family number 6 (i.e.\ choosing the germ $y^3 + yx^2$ and a corresponding universal unfolding for $V$) would still result in a fanning-out behavior similar to Theorem \ref{thm:main_thm}. This conjecture is supported by a numerical simulation of \eqref{start0} with potential $V = y^3/3 + yx^2 + a(y^2 - x^2) + bx + cy$, shown in figure \ref{fig:alternative_unfolding}. Also see appendix \ref{app:universal_unfolding} for this potential. 

In summary, in this section we have argued and justified the choice of \eqref{hyp_umb} as an appropriate representative of any fast-slow system of the form \eqref{start0} near an isolated hyperbolic umbilic singularity. We now proceed with the singular limit analysis of the fast-slow system \eqref{hyp_umb}, or equivalently \eqref{hyp_umb_f}.

\section{Critical manifold and slow flow}
\label{sec:slow_singular_limit}

The slow subsystem of \eqref{hyp_umb} is given by
\begin{equation}
	\begin{split}
		\label{slow_sub}
		0 &= x^2 + ay + b  \\
		0 &= y^2 + ax + c  \\
		\dot{a} &= g_a(x,y,a,b,c,0) \\
		\dot{b} &= g_b(x,y,a,b,c,0) \\
		\dot{c} &= g_c(x,y,a,b,c,0).
	\end{split}
\end{equation}

\subsection{The critical manifold and its singularities}
The 3-dimensional critical manifold $\mathcal{S}_0$ of \eqref{hyp_umb} is given by
\begin{equation}
	\begin{split}
		\label{crit_mfld}
		\mathcal{S}_0 = \{ x^2 + ay + b = 0, y^2 + ax + c = 0 \}.
	\end{split}
\end{equation}
Since the fast variables of \eqref{hyp_umb} are of gradient type, the singular points of $\mathcal{S}_0$ are precisely the points, where $\D^2 V$ is not invertible. Explicitly, these are all points on $\mathcal{S}_0$, where
\begin{equation}
	\begin{split}
		\label{sing_cone}
		\det \D^2 V = \det \begin{bmatrix}
			2x & a \\
			a & 2y
		\end{bmatrix} = 4 xy - a^2 = 0.
	\end{split}
\end{equation}

The projection of the 3-manifold $\mathcal{S}_0$ onto the ambient $(x,y,a)$-coordinates is a bijection. This gives a parametrization for $\mathcal{S}_0$, given by
\begin{equation}
	\begin{split}
		\label{parametrization}
		\Psi\colon \R^3 \to \R^5,\, (x,y,a) \mapsto (x,y,a,b = -x^2-ay, c = -y^2-ax).
	\end{split}
\end{equation}
Via $\Psi$ the singular points \eqref{sing_cone} appear as a double-cone, see figure \ref{double_cone}.
\begin{figure}[h]
	\centering
	\begin{overpic}[width=.5\textwidth]{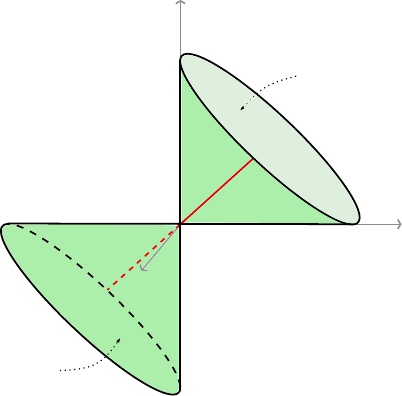}
		\put(8,5){$\mathcal{S}_0^{\txta}$}
		\put(76,78){$\mathcal{S}_0^{\txtr}$}
		\put(20,65){$\mathcal{S}_0^{\txts}$}
		\put(38,27){\textcolor{gray}{$a$}}
		\put(95,37){\textcolor{gray}{$x$}}
		\put(39,93){\textcolor{gray}{$y$}}
	\end{overpic}
	\caption{Stratification of $\mathcal{S}_0$: The critical manifold itself corresponds to the ambient $(x,y,a)$-space. The double cone determined by \eqref{sing_cone} represents the singular points of the critical manifold $\mathcal{S}_0$ via $\Psi$, compare to figure \ref{fig:main_thm_sing_limit}. The hyperbolic umbilic singularity is located at the origin. The 2-manifold given by the double cone away from the origin consists of a surface of fold points (green) and a line of cusp points (red). The attracting region $\mathcal{S}^{\txta}_0$ is the region inside the lower cone, the repelling region $\mathcal{S}^{\txtr}_0$ is the region inside the upper cone.}
	\label{double_cone}
\end{figure}
The following Lemma establishes the different stability regions $\mathcal{S}_0^{\txta},\mathcal{S}_0^{\txtr},\mathcal{S}_0^{\txts}$ of $\mathcal{S}_0$ with respect to the fast subsystem.
\begin{lemma}
	The critical manifold $\mathcal{S}_0$ has three connected normally hyperbolic parts, each having different stability properties with respect to the fast subsystem: An attracting part $\mathcal{S}_0^{\txta}$, a saddle-type part $\mathcal{S}_0^{\txts}$ and a repelling part $\mathcal{S}_0^{\txtr}$. Via $\Psi$, such regions are given by $\mathcal{S}^{\txta}_0 = \{4xy - a^2 > 0, x + y < 0\},\; \mathcal{S}^{\txts}_0 = \{4xy - a^2 < 0\},\textnormal{ and } \mathcal{S}^{\txtr}_0 = \{4xy - a^2 > 0, x + y > 0\}$.
\end{lemma}
\begin{proof}
	Recall that $\D^2 V$ can be viewed as linearization of the fast subsystem. The regions are determined by the signs of the eigenvalues of $\D^2 V$, which are determined by the signs of $\det \D^2 V = 4 xy - a^2$ and $\tr \D^2 V = 2 (x + y)$. This directly leads to the above regions.
\end{proof}

We now aim to classify the singular points of $\mathcal{S}_0$. Let $\pi$ be the natural projection $\pi \colon (x,y,a,b,c) \mapsto (a,b,c)$ onto the parameter subspace.
Then $\tilde{\pi} = \pi \circ \Psi$,
\begin{equation}
	\begin{split}
		\label{pi_tilde}
		\tilde{\pi}\colon (x,y,a) \mapsto (a,-x^2-ay,-y^2-ax)
	\end{split}
\end{equation}
gives a smooth map $\mathcal{S}_0 \to \R^3$ of 3-manifolds, which is the restriction of $\pi$ to $\mathcal{S}_0$ in the coordinates $\Psi$ from \eqref{parametrization}. In catastrophe theory $\tilde{\pi}$ is called the catastrophe map. The singular points of $\mathcal{S}_0$ correspond to singularities of $\tilde{\pi}$, that is, points at which the linearization is singular. Further, the set of singular points is generically stratified into submanifolds of singularities of the same type \cite{arnold2012singularities,boardman1967,zeeman1976,kojakhmetov2014constrained}. In our case, the singular points are stratified as follows:
\begin{lemma}
	\label{classif_singular_points}
	Besides the hyperbolic umbilic singularity at the origin, there is a line of cusp points given by $x = y = a/2$, and the remaining surface $4xy - a^2 = 0$ consists of fold points.
\end{lemma}
\begin{proof}
	Let $\Sigma^{1}(\tilde{\pi})$ be the set of points, where the kernel of $\d\tilde{\pi}$ has dimension 1. Then $\Sigma^{1}(\tilde{\pi})$ is precisely the 2-manifold determined by $4xy - a^2 = 0$ and $(x,y) \neq 0$. Folds and cusps are contained in there.
	Restrict $\tilde{\pi}$ to $\Sigma^{1}(\tilde{\pi})$. Then $\d \tilde{\pi}\eval_{\Sigma^1(\tilde{\pi})}$ has a one dimensional kernel along $x = y = a/2$, which gives the cusp line. The situation is depicted in figure \ref{double_cone}. The notation $\Sigma^1(\cdot)$ refers to the (Thom-Boardman) symbols $\Sigma^{I}$, refer to \cite{arnold2012singularities,boardman1967}.
\end{proof}

\subsection{The slow flow and its desingularization}
\label{ss:sfdesing}
The right hand side $g$ of the slow subsystem \eqref{slow_sub} determines dynamics on the critical manifold $\mathcal{S}_0$ through the equations \eqref{crit_mfld} of the critical manifold. This dynamics is refered to as slow flow on $\mathcal{S}_0$. The goal of this section is to analyze this slow flow, more precisely its interaction with the hyperbolic umbilic point. Since we are intuitively dealing with a flow on the manifold $\mathcal{S}_0$, the natural approach would be to use charts and push forward this flow to euclidean space. However, we point out that the slow flow is in general not a flow in the usual sense of vector fields on manifolds. In \eqref{slow_sub}, the behavior of the slow flow is linked to the geometry of $\mathcal{S}_0$. More precisely, away from the singular points of $\mathcal{S}_0$, the slow flow behaves like a usual flow of a vector field on $\mathcal{S}_0$. But this behavior breaks down at the singular points of $\mathcal{S}_0$. For example, existence and uniqueness of integral curves might fail at singular points. The procedure to obtain integral curves for the slow flow on $\mathcal{S}_0$ was carried out by Takens in \cite{takens1976}. Let us outline this important procedure.

Recall that $g$ is the right hand side of \eqref{slow_sub}. A natural idea is to ``transform'' $g$ via $\tilde{\pi} \colon \mathcal{S}_0 \to \R^3$, which maps points in $\mathcal{S}_0$ parametrized by $(x,y,a)$ via $\Psi$ to their parameters $(a,b,c)$. Note that $\tilde{\pi}$ is a local diffeomorphism away from the singular points, which allows to pullback $g$. In fact, one can ``lift'' the slow flow $g$ to a vector field $Y$ on $\mathcal{S}_0$ formally given by
\begin{equation}
	\begin{split}
		\label{slow_flow_desing}
		Y = \det \d \tilde{\pi} \cdot (\d \tilde{\pi})^{-1} g(x,y,\tilde{\pi}(x,y,a)),
	\end{split}
\end{equation}
although $\d \tilde{\pi}$ is not invertible everywhere. The reason for this is that by Cramer's rule
\begin{equation}
	\begin{split}
		(\d \tilde{\pi})^{-1} = \frac{1}{\det \d \tilde{\pi}} \mathrm{adj}(\d \tilde{\pi}),
	\end{split}
\end{equation}
where $\mathrm{adj}(\cdot)$ denotes the transpose of the cofactor matrix. Therefore \eqref{slow_flow_desing} reduces to $Y = \mathrm{adj}(\d \tilde{\pi})g(x,y,\tilde{\pi}(x,y,a))$.
Away from singular points of $\mathcal{S}_0$, $Y$ is just a rescaling of the vector field $(\d \tilde{\pi})^{-1} g$. The factor $\det \d \tilde{\pi} = 4xy - a^2$ serves as desingularization, which leads to the term ``desingularized slow flow'' for $Y$. Note that $Y$ is smooth if $g$ is. It is important to note that by the desingularization factor $\det \d \tilde{\pi}$ in \eqref{slow_flow_desing}, the orientation of trajectories is reversed in the region $\{ \det \d \tilde{\pi} < 0\}$. This time reversion in the region $\{ \det \d \tilde{\pi} < 0\}$ on $\mathcal{S}_0$ must not be neglected, since only after this we obtain the correct flow on the critical manifold away from singular points. For illustrating examples of the desingularization procedure \eqref{slow_flow_desing} we refer to \cite{kojakhmetov2014constrained}.

In our case of the hyperbolic umbilic, \eqref{slow_flow_desing} leads to the following desingularized slow flow:
\begin{lemma}
	\label{takens_desing}
	The desingularized slow flow $Y$ for \eqref{hyp_umb} is given by
	\begin{equation}
		\begin{split}
			\label{sf:desing}
			\dot{x} &= a \tilde{g}_c(x,y,a) - 2y \tilde{g}_b(x,y,a) + (xa - 2y^2) \tilde{g}_a(x,y,a)\\
			\dot{y} &= a \tilde{g}_b(x,y,a) - 2x \tilde{g}_c(x,y,a) + (ya - 2x^2) \tilde{g}_a(x,y,a)\\
			\dot{a} &= (4xy - a^2) \tilde{g}_a(x,y,a),
		\end{split}
	\end{equation}
	where the time orientation of trajectories needs to be reversed in the saddle-type region $\mathcal{S}_0^\txts = \{ 4xy - a^2 < 0 \}$. The notation $\tilde{g}_\bullet(x,y,a) = g_\bullet(x,y,\tilde{\pi}(x,y,a),0) = g_\bullet (x,y,a,-x^2-ay,-y^2-ax,0)$ is used.
\end{lemma}
\begin{proof}
	Computation of $\mathrm{adj}(\d \tilde{\pi}) g(x,y,\tilde{\pi}(x,y,a),0)$. The region $\{\det \d \tilde{\pi} < 0 \}$ where the orientation of trajectories must be reversed coincides with the saddle-type region $\mathcal{S}^{\txts}_0$.
\end{proof}

Recall that $Y$ induces dynamics on $\mathcal{S}_0$. For the remainder of this article we let $Y$ denote the vector field \eqref{sf:desing}, where the orientation of integral curves in $\mathcal{S}_0^{\txts}$, i.e.\ outside of the double cone of singularities, is reversed. 

We are interested in the dynamics close to the hyperbolic umbilic singularity, located at the origin. From now on, let
\begin{equation}
	\begin{split}
		A_0 := g_a(0),\, B_0 := g_b(0),\, C_0 := g_c(0).
	\end{split}
\end{equation}
Note that $Y$ \eqref{sf:desing} has an equilibrium at the origin. We can classify this equilibrium of $Y$ according to the following lemma. This is similar to \cite{kojakhmetov2014constrained}.

\begin{figure}[h]
	\centering
	\begin{overpic}[width=.5\textwidth]{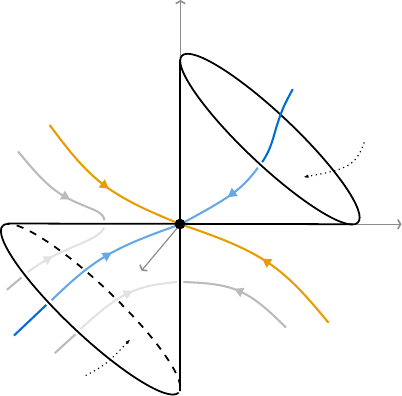}
		\put(31,30){\textcolor{gray}{$a$}}
		\put(95,37){\textcolor{gray}{$x$}}
		\put(39,93){\textcolor{gray}{$y$}}
		\put(15,4){$\mathcal{S}_0^{\txta}$}
		\put(89,66){$\mathcal{S}_0^{\txtr}$}
		\put(30,65){$\mathcal{S}_0^{\txts}$}
		\put(2,18){\textcolor{light-blue}{$\sigma$}}
	\end{overpic}
	\caption{Four trajectories of the desingularized slow flow $Y$ \eqref{sf:desing} approach the origin in the case $B_0, C_0 > 0$, which correspond to the 1-d stable/unstable manifolds at the origin before desingularization. In the vicinity of the origin one of such manifolds is contained in $\mathcal{S}_0^{\txta}$ (denoted by $\sigma$) and $\mathcal{S}_0^{\txtr}$, and the other in $\mathcal{S}_0^{\txts}$. Also compare to figure \ref{fig:main_thm_sing_limit}, where $\sigma$ appears as blue trajectory. A $1$-d center manifold, not shown, also exists and is tangent to the cone at the origin. In particular, the hyperbolic umbilic is not a funnel in the sense of Takens \cite{takens1976}, and thus we shall concentrate on tracking $\sigma$ and its perturbations. This property of the hyperbolic umbilic is evident in the blow-up analysis. Away from the center manifold, almost all trajectories of the slow flow leave $\mathcal{S}_0^{\txta}$ via folds (and cusps), as depicted by means of two trajectories in gray.}
	\label{desing_sketch}
\end{figure}

\begin{lemma}
	\label{sf:origin_lin}
	 For any $g = (g_a,g_b,g_c)$ in \eqref{hyp_umb}, the origin is a non-hyperbolic equilibrium of the desingularized slow flow $Y$ \eqref{sf:desing}. The linearization at the origin has eigenvalues $0$ and $\pm2\sqrt{B_0 C_0}$. Taking the time reversion in $\mathcal{S}_0^\txts$ into account, the local behavior near the origin is characterized by the following arrangement of invariant manifolds according to the values of $B_0$ and $C_0$:
	 \begin{align*}
	 		B_0,C_0 > 0: & \text{ two 1-d stable manifolds and 1-d center manifold}\\
			B_0,C_0 < 0: & \text{ two 1-d unstable manifolds and 1-d center manifold}\\
			B_0,C_0 \text{ have opposite sign}: & \text{ 1-d center manifold and rotational behavior around it}\\
			B_0 = 0 \text{ or } C_0 = 0: & \text{ degenerate case}
	 \end{align*}
	 The first three cases occur generically.
\end{lemma}
\begin{proof}
	The linearization at the origin is given by
	\begin{equation}
		\begin{split}
			\begin{bmatrix}
				0 & -2 B_0 & C_0\\
				-2 C_0 & 0 & B_0\\
				0 & 0 & 0
			\end{bmatrix},
		\end{split}
	\end{equation}
	which has the eigenvalues $0, \pm 2\sqrt{B_0 C_0}$.
	In the case $B_0, C_0 > 0$ the eigenvalues and corresponding eigenvectors of the linearization are given by
	\begin{equation}
		\begin{split}
			0\colon  \left(\frac{B_0}{2C_0},\frac{C_0}{2B_0},1\right),\,
			-2 \sqrt{B_0C_0} \colon \left(\sqrt{B_0/C_0},1,0\right),\,
			2 \sqrt{B_0C_0} \colon \left(-\sqrt{B_0/C_0},1,0\right).
		\end{split}
	\end{equation}
	In the case $B_0, C_0 < 0$ we have
	\begin{equation}
		\begin{split}
			0\colon  \left(\frac{B_0}{2C_0},\frac{C_0}{2B_0},1\right),\,
			-2 \sqrt{B_0C_0} \colon \left(-\sqrt{B_0/C_0},1,0\right),\,
			2 \sqrt{B_0C_0} \colon \left(\sqrt{B_0/C_0},1,0\right).
		\end{split}
	\end{equation}
	Note the swapped eigenvectors corresponding to the latter two nonzero eigenvalues. This is due to the fact that $C_0 \sqrt{B_0/ C_0} = - \sqrt{B_0 C_0}$ if $B_0, C_0 < 0$. In the case $B_0, C_0 > 0$, there exists a 1-d stable manifold contained in $\mathcal{S}_0^{\txta} \cup \mathcal{S}_0^{\txtr}$ close to the origin. Further, there exists a 1-d unstable manifold contained in $\mathcal{S}_0^{\txts}$ close to the origin. Due to the time reversion in $\mathcal{S}_0^{\txts}$ we obtain four trajectories approaching the origin. See figure \ref{desing_sketch}.
	Similarly, we obtain four trajectories emanating from the origin if $B_0, C_0< 0$. Note that in both cases the 1-d center manifolds are ``tangent'' to the double cone surface at the origin.
\end{proof}

Lemma \ref{sf:origin_lin} topologically determines the slow flow dynamics away from the 1-d center manifold(s) in case of $B_0C_0 \neq 0$, close to the origin.
Motivated by Lemma \ref{sf:origin_lin}, we shall restrict the forthcoming analysis to the generic sub-case $B_0,C_0 > 0$. The simple reason for this is that the flow on the stable manifolds is directed towards the hyperbolic umbilic point in this case. In contrast, for $B_0,C_0 < 0$ 
there are no trajectories approaching the hyperbolic umbilic point, except possibly on the center manifold.
In the case of $B_0,C_0$ having opposite sign (i.e.\ $B_0 C_0 < 0$), we obtain a pair of conjugated purely imaginary eigenvalues at the hyperbolic umbilic. To desingularize this case, the approach of \cite{hayes2016geometric} might be applicable.

\begin{remark}[The non-hyperbolic equilibrium of $Y$ at the origin]
For any $g = (g_a,g_b,g_c)$ in \eqref{hyp_umb}, respectively \eqref{slow_sub}, the vector field $Y$ \eqref{sf:desing} has a non-hyperbolic equilibrium at the origin. We emphasize that this is an equilibrium of the desingularized slow flow $Y$ and not of the original system. Roughly speaking, this occurs due to two fast variables in \eqref{hyp_umb} with appropriate non-linear right-hand side for $\eps = 0$, in our case second-order polynomials. The equilibrium at the origin is due to the presence of $\mathrm{adj}(\d \tilde{\pi})$ in \eqref{slow_flow_desing} and the non-hyperbolicity is due to the quadratic factor $\det \d \tilde{\pi} = 4xy - a^2$ in the third equation of \eqref{sf:desing}, coming from the desingularization process.
This is in contrast to the case of folds, cusps and the swallowtails, where the desingularized slow flow is generically nonzero at the most degenerate singularity \cite{takens1976}. In the case of folded singularities \cite{szmolyan2001canards}, the desingularized slow flow generically has a hyperbolic equilibrium at the folded singularity.
\end{remark}

From now on, we assume $B_0,C_0 > 0$, if not stated otherwise. Note that by swapping the coordinate labels $b \leftrightarrow c$ and $x \leftrightarrow y$ in \eqref{hyp_umb}, we obtain \eqref{hyp_umb} with swapped $g_b \leftrightarrow g_c, f_x \leftrightarrow f_y$ and permuted first two arguments in these functions. This suggests that in the singular limit the case $g_b = g_c$ might play a special role. Indeed, if $g_b = g_c$ then $\{ x = y \}$ is invariant for \eqref{sf:desing}. We will see later that the case $B_0 = C_0$ will play a special role on the blow-up sphere.

Let $\sigma$ denote the unique trajectory, which approaches the origin off the center manifold in forward time from inside $\mathcal{S}_0^{\txta}$ sufficiently close to the origin. Refer to figure \ref{desing_sketch}. This trajectory will play an important role in the forthcoming analysis, as we aim to track it through the hyperbolic umbilic.

\section{Fast subsystem}
\label{sec:fast_subsystem}

In this section we analyze the fast subsystem of \eqref{hyp_umb}: its equilibria, their bifurcations and heteroclinic connections.
This analysis is naturally the next step to analyze the singular limit of a fast-slow system. The analysis will show the interplay of the fast fibers with the critical manifold $\mathcal{S}_0$, since equilibria of the fast subsystem are instances of $\mathcal{S}_0$.
Considering \eqref{start:fast}, we have the parameter space $(a,b,c)$ and the planar gradient system
\begin{equation}
	\begin{split}
		\label{planar_fast}
		x' &= x^2 + ay + b\\
		y' &= y^2 + ax + c.
	\end{split}
\end{equation}
Let $Z_{a,b,c}(x,y)$ denote the vector field determined by \eqref{planar_fast}. This system was already analyzed to certain extent in \cite{guckenheimer1973,vegter1982,khesin1990}.

\subsection{Equilibria and their bifurcations} The vector field $Z_{a,b,c}$ has at most four equilibria, since, for $a \neq 0$, they correspond to intersections of the two parabolas 
\begin{equation}
	\begin{split}
		\label{parabolas}
		0 &= x^2 + ay + b\\
		0 &= y^2 + ax + c.
	\end{split}
\end{equation}
One parabola is symmetric with respect to the $y$-axis, the other symmetric with respect to the $x$-axis. Depending on the sign of $a$, the parabolas either open right and up or open left and down. For $a=0$, each of the parabolas degenerates into a pair of parallel lines.

Linearization of $Z_{a,b,c}(x,y)$ gives $\det \mathrm{D}Z_{a,b,c}(x,y) = 4xy - a^2,\; \tr \mathrm{D} Z_{a,b,c}(x,y) = 2(x+y)$. Note that for each $a$, these two expressions determine the regions, in which the phase space $(x,y)$ of the vector field $Z_{a,b,c}(x,y)$ contains sinks, sources, saddles or fold/cusp bifurcations. These regions appear precisely as in taking a section $a = \text{const}$ through figure \ref{double_cone}. Also refer to figure \ref{4config}, in which these regions are separated by the green dashed curves.

\begin{figure}[h]
	\centering
	\begin{overpic}[width=\textwidth]{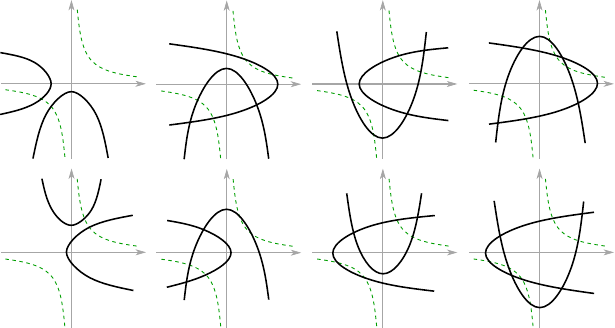}
		\put(0,50){\small (A)}
		\put(26,50){\small (B)}
		\put(51,50){\small (C)}
		\put(77,50){\small (D)}
		\put(0,1){\small (A)}
		\put(26,1){\small (B)}
		\put(51,1){\small (C)}
		\put(77,1){\small (D)}
	\end{overpic}
	\caption{The generic intersections of the two parabolas \eqref{parabolas} give the generic configurations (A),(B),(C),(D) of equilibria of $Z$. The dashed green curves indicate the fold curves for fixed $\abs{a} > 0$, which divide the plane in three regions: sinks occur in the lower left region, sources in the upper right region, saddles in between.}
	\label{4config}
\end{figure}

Thus, there are four generic configurations of the equilibria of $Z_{a,b,c}(x,y)$, which are shown in figure \ref{4config}: (A) no equilibria, (B) a saddle and a sink, (C) a saddle and a source and (D) two saddles, a source and a sink. Note that (B) can only occur for $a > 0$, while (C) can only occur for $a < 0$. The fold/cusp bifurcations, which happen upon variation of the parameters $a,b,c$, visually correspond to transitions between the configurations in figure \ref{4config}. That is, by suitably shifting the parabolas along their symmetry-axes by varying $b,c$ and changing the parabolas' ``width'' by varying $a$.

The following symmetries allow to simplify the study of the family of vector fields $Z_{a,b,c}$ and its bifurcations.

\begin{lemma}
\label{sym}
The family of vector fields $Z$, as defined in \eqref{planar_fast}, has the following symmetries:
\begin{itemize}
	\item[(S1)] $Z_{a,c,b} \circ A = A \circ Z_{a,b,c}$, where $A$ is reflection along $x = y$. That is, $Z_{a,b,c}$ and $Z_{a,c,b}$ are conjugate via $A$. In particular, the stability types of the reflected equilibria remain invariant.
	\item[(S2)] $Z_{-a,b,c} \circ R = Z_{a,b,c}$, where $R$ is rotation by $180^{\circ}$. Integral curves of $Z_{-a,b,c}$ are obtained from integral curves of $Z_{a,b,c}$ by rotation via $R$ and time reversion. In particular, under $a \mapsto -a$, the positions of equilibria of $Z$ are rotated by $R$ and saddles remain saddles, but sinks become sources and vice versa.
\end{itemize}
\end{lemma}
\begin{proof}
	(S1) follows by direct computation using $A(x,y) = (y,x)$. 
	For (S2), observe that since $R$ amounts to multiplication by $(-1)$, $Z_{-a,b,c}$ and $-Z_{a,b,c}$ are conjugate via $R$, which implies the statements.
\end{proof}
We will simply refer to the properties (S1) and (S2) from Lemma \ref{sym} by (S1) and (S2) in the remainder of this section.
Note how the symmetries appear in figure \ref{4config}: switching from the upper configuration (A) to the lower configuration (A) amounts to (S2), similarly for (D). Switching from the upper row to the lower row in case of (B) and (C) amounts to (S1).

Recall the catastrophe set or bifurcation set, which is given by projection of the set of singularities (represented by the double cone) onto the parameter space $(a,b,c)$.
By the symmetries (S1) and (S2) from Lemma \ref{sym}, the bifurcation set in the parameter space is symmetric to the planes $a = 0$ and $b = c$. Furthermore, the bifurcation set is qualitatively well known, see e.g.\ \cite{poston1996catastrophe,broecker1975germs}.
This leads to the qualitative bifurcation diagram depicted in figure \ref{bif_fast}. The bifurcation set divides the parameter space into four connected regions (A),(B),(C) and (D), which correspond to the generic configurations in figure \ref{4config}.
Topological equivalence of the associated vector fields $Z$ in these regions follows from the absence of saddle connections:
\begin{lemma}
	\label{saddle_connection}
	In each of the open and connected regions (A),(B),(C),(D) the family of vector fields $Z_{a,b,c}(x,y)$ has topologically equivalent flows.
\end{lemma}
\begin{proof}
	The vector field $Z_{a,b,c}(x,y)$ is a gradient. By the Andronov-Pontryagin criterion for planar structural stability we only need to consider the region (D), where a saddle to saddle connection is possible. But the vector fields $Z_{a,b,c}(x,y)$ in (D) do not show a saddle connection. This was already shown implicitly in \cite{guckenheimer1973}. Here we give an alternative argument, similar to \cite{khesin1990}: by \cite{chicone1979} saddle connections in quadratic planar gradient systems can only occur along straight lines. Assume that $Z_{a,b,c}(x,y)$ has an invariant straight line $y = mx + d$ connecting the two saddles. Note that necessarily $m < 0$, see figure \ref{4config}. Then the zero level set of $p(x,y) = y' - mx' = y^2 - mx^2 + a(x - my) + c -mb$ needs to contain $y = mx +d$ for all $x$-values between the saddles. The gradient $\nabla p$ vanishes in a single point, so away from this point $p(x,y) = 0$ determines a smooth curve. If the straight line and $p(x,y) = 0$ locally agree, there would exist a real valued function $\lambda(x)$, such that $\nabla p(x,mx + d) = \lambda(x) [-m \quad 1]$, the latter being the normal vector to the invariant line. This implies $\lambda(x) = 2x - a/m$ and $\lambda(x) = m (2x - a) + 2d$. But due to $m < 0$ this cannot hold, even locally.
\end{proof}

\begin{figure}[h]
	\centering
	\begin{overpic}[width=0.89\textwidth]{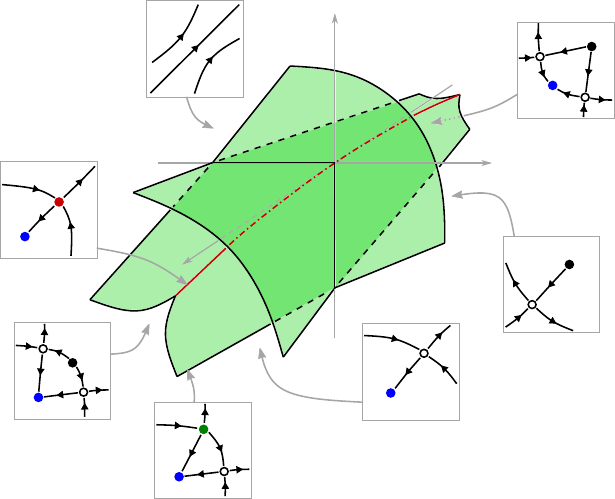}
		\put(28,40){\textcolor{gray}{$a$}}
		\put(81,54){\textcolor{gray}{$b$}}
		\put(55,80){\textcolor{gray}{$c$}}
		\put(0,56){\textcolor{dark-red}{cusp}}
		\put(35,17){\textcolor{dark-green}{fold}}
		\put(24,62){(A)}
		\put(59,30){(B)}
		\put(93,44){(C)}
		\put(2,30){(D)}
		\put(84,79){(D)}
	\end{overpic}
	\caption{Qualitative bifurcation diagram of $Z_{a,b,c}$ \eqref{planar_fast} in parameter space: The bifurcation set is given by projection of the singular points of $\mathcal{S}_0$ onto $(a,b,c)$, and corresponds to the hyperbolic umbilic point at the origin, a curve of cusp points (red) and two surfaces of fold points (green), which intersect along the negative $b$- and $c$-axes. The phase portraits show the qualitative behavior of the fast subsystem at bifurcations and between them. Folds are shown in green, cusps in red, sinks in blue, sources in black and saddles in white. Naturally, sinks/sources and saddles correspond to points in $\mathcal{S}_0^\txta$/$\mathcal{S}_0^\txtr$ and $\mathcal{S}_0^\txts$, respectively.}
	\label{bif_fast}
\end{figure}

\subsection{Feasible jumps}
We now aim to combine the analysis of the fast and slow subsystem away from the hyperbolic umbilic singularity. The goal is to see which concatenations of singular candidate trajectories are possible at all, due to the geometry of $\mathcal{S}_0$ and the fast subsystems. In the following we will focus on the possibilities at jump points, i.e. singular points, where singular candidate trajectories on the critical manifold ``jump'' onto the fast subsystem. As in lower dimensions, one expects that generic fold points indeed serve as jump points: that is, if the slow flow is transverse to the surface of fold points. Refer to the ``normal switching condition'' in \cite{kuehn2015multiple}. In fact, at the fold or cusp points, we can locally employ a center manifold reduction to reduce to one fast variable. This allows to directly apply the existing theory, but only locally. We now aim to obtain global insight. Since the fast subsystem is $2$-dimensional and the critical manifold is $3$-dimensional, a visualization of feasible jumps as in, e.g.\ \cite{kojakhmetov2016cusp}, is not directly applicable. Moreover, in case there is only one fast variable present, the effect of the fast subsystem for candidate trajectories is usually straight forward, since it amounts to projection along the fast fibers.
Here, the fast fibers are 2-d with dynamics determined by a gradient vector field. The essential question for jumps is: \emph{To which equilibria can the fast subsystem lead candidate trajectories, emanating from a fold or cusp point?}

To answer this question, we need to consider $Z_{a,b,c}(x,y)$ at parameter values $(a,b,c)$ that correspond to a fold or a cusp bifurcation. We then need to check which candidate trajectories, emanating from these ``static'' bifurcating equilibria are feasible. That is, we need to check the forward asymptotic behavior of trajectories, which are backward asymptotic to these bifurcation equilibria.
Consider the generic configurations (A),(B),(C),(D), which correspond to regions in parameter space as indicated in the bifurcation diagram in figure \ref{bif_fast}. These regions are separated by the bifurcation set. We call a fold bifurcation of type (A) $\leftrightarrow$ (B), if it occurs in $Z_{a,b,c}$ for parameter variation $(a,b,c)$ from region (A) to region (B) or vice versa. Similarly, we use this notation for all other bifurcations according to the bifurcation diagram in figure \ref{bif_fast}. Note that if $a = 0$ we can have two simultaneuous folds and two ``overlapping'' cusps at $a = b = c = 0$, i.e.\ the hyperbolic umbilic.

In the following we will encounter fast candidate trajectories, which ``escape'' towards infinity. To treat this behavior, let us setup a suitable ``exit'' section, which solutions pass before escaping. For the vector field $Z_{a,b,c}$ consider the section
\begin{equation}
	\begin{split}
		\label{fast_subsystem_exit}
		\Delta^{\ex} = \{ x + y = 2\nu \},
	\end{split}
\end{equation}
for $\nu > 0$.
Then there is an open neighbourhood $K$ of the origin in parameter $(a,b,c)$-space, such that all equilibria of $Z_{a,b,c}$ for $(a,b,c)\in K$ are located below $\Delta^{\ex}$. This assumption is merely to simplify the following statements. In the scope of the current section $\nu > 0$ can be chosen arbitrarily large. In this setting, we state the following Lemma, which constructs positively invariant regions, in which jumping candidate trajectories need to be contained. Eventually, this leads to feasible jumps along the fast subsystem.

\begin{lemma}
	\label{bif_transitions}
	Consider $Z_{a,b,c}$ at parameter values $(a,b,c)\in K$ contained in the bifurcation set. 
	\begin{enumerate}[(a)]
		\item If $(a\neq 0,b,c)$ does \textbf{not} correspond to a fold or cusp bifurcation $(B) \leftrightarrow (D)$, then any trajectory emanating from the bifurcation equilibrium arrives at the section $\Delta^{\ex}$.
		\item If $(a\neq 0,b,c)$ corresponds to a fold or cusp bifurcation $(B) \leftrightarrow (D)$, then every trajectory emanating from the bifurcation equilibrium either arrives at $\Delta^{\ex}$ or is forward asymptotic to a sink or a saddle.
		\item If $(a = 0, b ,c)$: Any trajectory emanating from a fold equilibrium located on the negative $x$- or $y$-axis arrives at $\Delta^{\ex}$. Any trajectory emanating from a fold equilibrium located on the positive $x$- or $y$-axis is forward asymptotic to the simultaneuous fold on the negative axis or arrives at $\Delta^{\ex}$. Any trajectory of $Z_{0,0,0}$ emanating from the origin arrives at $\Delta^{\ex}$.
		\item From $\Delta^{\ex}$ any trajectory escapes towards infinity and cannot connect to equilibria.
	\end{enumerate}
\end{lemma}

\begin{proof}
	This follows by looking at the nullclines of $Z_{a,b,c}$, i.e.\ the configuration of the parabolas \eqref{parabolas}, for parameters in the bifurcation set. The proof is given by means of figure \ref{proof_figs}, which contains four qualitative phase portraits, to which all cases except $Z_{0,0,0}$ can be reduced. We explain the reduction and figure \ref{proof_figs}:
	Assume $a \neq 0$. Consider the phase portrait of $Z_{a,b,c}$ at parameter values as mentioned above. Due to (S1) from Lemma \ref{sym}, the argument for the folds reduces to four qualitatively distinct phase portraits corresponding to folds (A) $\leftrightarrow$ (B), (A) $\leftrightarrow$ (C), (B) $\leftrightarrow$ (D) and (C) $\leftrightarrow$ (D). By (S2), we can take the phase portrait for a fold (B) $\leftrightarrow$ (D) and obtain the phase portrait for a fold (C) $\leftrightarrow$ (D) by $180^{\circ}$ rotation and time reversion. Similarly for folds (A) $\leftrightarrow$ (B) and (A) $\leftrightarrow$ (C). Hence only two fold phase portraits contain all the information we need for the proof. In figure \ref{proof_figs}, these are the upper two phase portraits: Trajectories being backward asymptotic to the folds (green) must be contained in the violet regions, which are positively invariant. Trajectories being forward asymptotic to the folds must be contained in the orange regions, which are negatively invariant. After time reversion and $180^{\circ}$ rotation the orange region corresponds to a positively invariant region containing trajectories emanating from folds in reverse time. 
	\begin{center}
		\begin{overpic}[width=.8\textwidth]{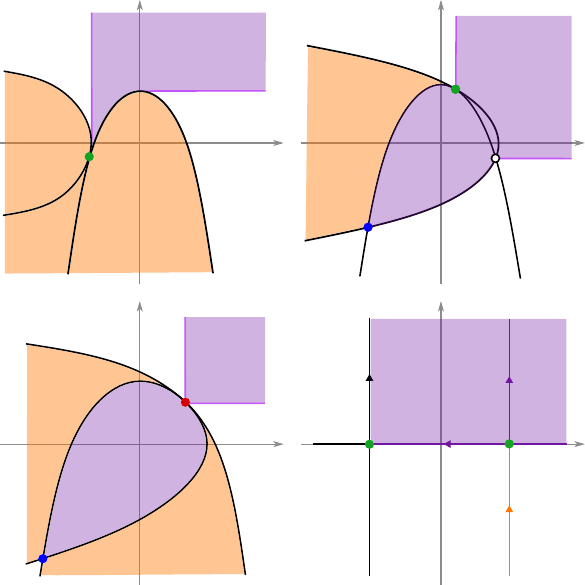}
			\put(98,21){\small \textcolor{gray}{$x$}}
			\put(98,73){\small \textcolor{gray}{$x$}}
			\put(46,21){\small \textcolor{gray}{$x$}}
			\put(46,73){\small \textcolor{gray}{$x$}}
			\put(21,98){\small \textcolor{gray}{$y$}}
			\put(21,47){\small \textcolor{gray}{$y$}}
			\put(72,98){\small \textcolor{gray}{$y$}}
			\put(72,47){\small \textcolor{gray}{$y$}}
			\put(38,53){\small $x' = 0$}
			\put(90,53){\small $x' = 0$}
			\put(43,1){\small $x' = 0$}
			\put(1,43){\small $y' = 0$}
			\put(1,89){\small $y' = 0$}
			\put(53,94){\small $y' = 0$}
			\put(10,84){\small $\rightarrow$}
			\put(5,95){\small $\nearrow$}
			\put(30,79.2){\small $\uparrow$}
			\put(34,65.5){\small $\uparrow$}
			\put(35,85){\small \textcolor{light-purple}{$\nearrow$}}
			\put(15.5,89){\small \textcolor{light-purple}{$\nearrow$}}
			\put(31,95){\small \textcolor{light-purple}{(A) $\leftrightarrow$ (B)}}
			\put(77.9,93){\small \textcolor{light-purple}{$\nearrow$}}
			\put(90,73.4){\small \textcolor{light-purple}{$\nearrow$}}
			\put(83.5,94.5){\small \textcolor{light-purple}{(B) $\leftrightarrow$ (D)}}
			\put(67,88.4){\small $\rightarrow$}
			\put(55,58.6){\small $\rightarrow$}
			\put(74,65){\small $\leftarrow$}
			\put(66.5,75){\small $\downarrow$}
			\put(70,94){\small $\nearrow$}
			\put(92,35){\small $\nearrow$}
			\put(92,13){\small $\nearrow$}
			\put(74,35){\small $\nwarrow$}
			\put(74,13){\small $\nwarrow$}
			\put(39,31.5){\small \textcolor{light-purple}{$\nearrow$}}
			\put(31.4,39){\small \textcolor{light-purple}{$\nearrow$}}
			\put(25,40){\small $\nearrow$}
			\put(5,58){\small $\nearrow$}
			\put(19,54.3){\small \textcolor{orange}{(A) $\leftrightarrow$ (C)}}
			\put(53,85){\small \textcolor{orange}{(C) $\leftrightarrow$ (D)}}
			\put(15,28){\small $\downarrow$}
			\put(15,38.6){\small $\rightarrow$}
			\put(39.5,14){\small $\uparrow$}
			\put(27.8,15){\small $\leftarrow$}
			\put(22,62){\small $\nwarrow$}
			\put(30,5){\small $\nwarrow$}
			\put(74,55){\small $\nwarrow$}
			\put(22,23){\small $\swarrow$}
			\put(74,75){\small $\swarrow$}
			\put(57,75){\small $\searrow$}
			\put(5,75){\small $\searrow$}
			\put(5,23){\small $\searrow$}
			\put(32,43){\small \textcolor{light-purple}{(B) $\leftrightarrow$ (D)}}
			\put(39.5,40){\small \textcolor{light-purple}{cusp}}
			\put(10,3){\small \textcolor{orange}{(C) $\leftrightarrow$ (D) cusp}}
		\end{overpic}
		
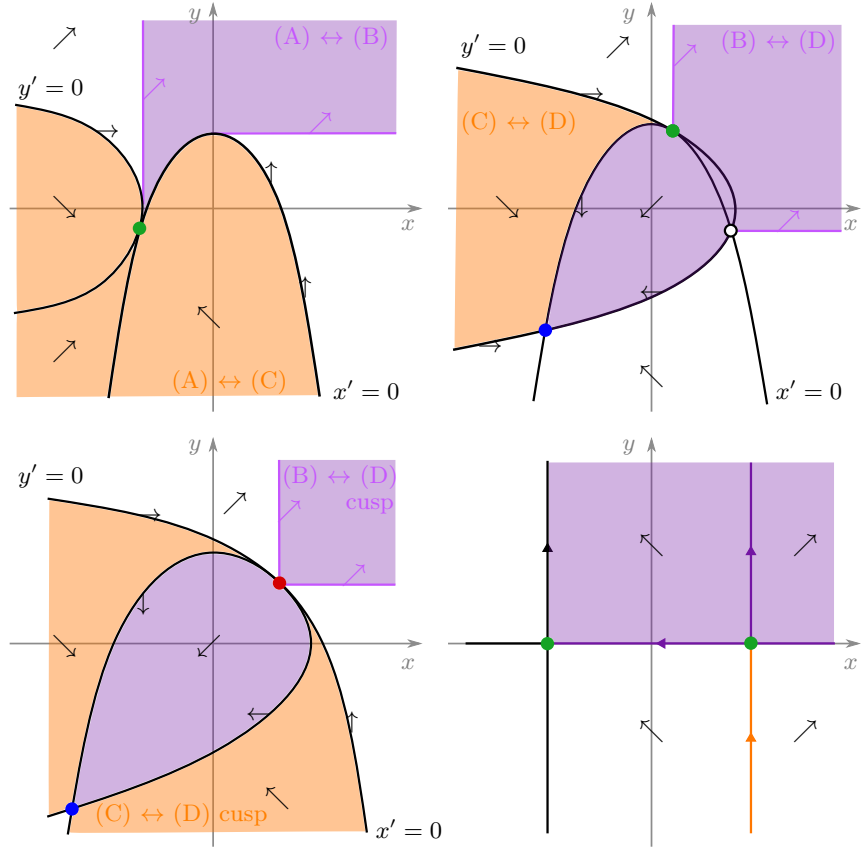
\captionof{figure}{Qualitative phase portraits for the proof of Lemma \ref{bif_transitions}: Up to (S1) and (S2), these four phase portraits qualitatively contain all relevant phase portraits of $Z_{a,b,c}$ at parameter values contained in the bifurcation set. The black parabolas (which degenerate into straight lines for $a = 0$) are the nullclines, the arrows indicate the signs of $x',y'$, i.e.\ the direction of the flow. The violet regions define positively invariant regions, in which trajectories emanating from the bifurcating equilibrium (fold in green or cusp in red) are contained. After time reversion the orange regions play a similar role. Folds are shown in green, cusps in red, sinks in blue, saddles in white. Recall that $\Delta^{\ex} = \{ x + y = 2\nu \}$.}
		\label{proof_figs}
	\end{center}
	The upper left portrait treats folds (A) $\leftrightarrow$ (B) with the violet positively invariant region and, via time reversion and $180^{\circ}$ rotation (S2), folds of type (A) $\leftrightarrow$ (C) with orange.  The upper right portrait treats (B) $\leftrightarrow$ (D) with the violet positively invariant region and (C) $\leftrightarrow$ (D) with the orange negatively invariant region, where again time reversion and $180^{\circ}$ rotation (S2) needs to be employed. Recall that (S2) turns a sink into a source.
	Two additional configurations of the nullclines lead to cusps, which can be reduced again by (S2) to a single phase portrait containing all the necessary information. This is the lower left phase portrait in figure \ref{proof_figs}. It treats cusps (B) $\leftrightarrow$ (D) with the violet region, cusps (C) $\leftrightarrow$ (D) with the orange region via (S2).
	The case $a = 0$ follows similarly, the nullcline parabolas just degenerate into (pairs of) straight lines. The case $Z_{0,0,0}$ follows easily. By (S1), the remaining cases with $a = 0$ can be reduced to a single phase portrait. This is the lower right portrait in figure \ref{proof_figs}, which does only show the violet and orange regions corresponding to the fold on the right hand side. Note that in $a = 0$ away from the origin two folds happen simultaneously. The case of $Z_{0,0,0}$, which is the fast subsystem at the hyperbolic umbilic itself, follows easily by looking at its phase portrait. Assertion (d) follows from Remark \ref{rem:escape}.
\end{proof}

\begin{remark}[The geometry near folds and cusps]
	\label{rem:foldcusp}
	Near the singularly perturbed cusp in dimension three \cite{broer2013cusp,kojakhmetov2016cusp} the critical manifold is ``S-shaped'' due to two intersecting fold curves. Further, along the fold curves, an attracting branch is ``folded over'' a repelling branch. Combined with the ``S-shaped'' critical manifold, this can lead to singular limit relaxation oscillation type behavior near the cusp. In our case, however, the above analysis shows that at folds a branch of $\mathcal{S}^{\txts}_0$ is folded over either a branch of $\mathcal{S}^{\txta}_0$ or $\mathcal{S}^{\txtr}_0$, given that $a \neq 0$. See for example figure \ref{bif_fast}. Moreover, at folds near the cusp curves, where $\mathcal{S}^{\txts}_0$ is folded over $\mathcal{S}^{\txta}_0$, there are only instances of $\mathcal{S}^{\txts}_0$ and $\mathcal{S}^{\txtr}_0$ present in the fast subsystem, except for the fold itself. Hence, at such folds we can in principle only jump to $\mathcal{S}^{\txts}_0$. These are folds of type (C) $\leftrightarrow$ (D), and Lemma \ref{bif_transitions} shows that the fast subsystem does not allow jumps to $\mathcal{S}^{\txts}_0$. Thus, a direct jump back to $\mathcal{S}_0$ is impossible, and therefore a singular limit relaxation oscillation behavior as in the 3-d cusp is impossible. In particular, a singular limit return mechanism from $\mathcal{S}^{\txta}_0$ back to $\mathcal{S}^{\txta}_0$ must go through the hyperbolic umbilic.
	At folds near the cusp curves, where $\mathcal{S}^{\txts}_0$ is folded over $\mathcal{S}^{\txtr}_0$, there are only instances of $\mathcal{S}^{\txts}_0$ and $\mathcal{S}^{\txta}_0$ present in the fast subsystem, except the fold itself. These are folds of type (B) $\leftrightarrow$ (D). Lemma \ref{bif_transitions} does not exclude jumps from these folds to $\mathcal{S}^{\txts}_0$ or $\mathcal{S}^{\txta}_0$.
\end{remark}

\begin{remark}[The exit section $\Delta^{\ex}$ and escape towards infinity]
	\label{rem:escape}
	The vector field $Z_{a,b,c}(x,y)$ might not be transverse to $\Delta^{\ex}$ at isolated points in the regions $x' < 0, y' > 0$ and $x' > 0, y' < 0$. However, this section still works as suitable ``exit'' section: Any section of the form $\{x = \text{const}\}$ or $\{ y = \text{const}\}$ is transverse to the flow in these regions, which we can place ``before'' $\Delta^{\ex}$. Together with $\Delta^{\ex}$ this then gives ``triangular''-like compact regions, which solutions need to exit in finite time through $\Delta^{\ex}$ by Poincar\'e-Bendixson. Further, figure \ref{proof_figs} shows that in this case trajectories can only escape towards infinity and cannot connect to any equilibrium. 
\end{remark}

\begin{remark}[Asymptotic estimates from the positively/negatively invariant regions]
	\label{kappa_a_asymptotics}
	By employing the nullclines and the direction of $Z_{a,b,c}(x,y)$ along the diagonal $x = y$, it is easy to obtain forward asymptotic estimates for trajectories, which jump onto the fast subsystem at folds.  At folds (C) $\leftrightarrow$ (D), any solution $(x,y)$ emanating from the corresponding fold equilibrium satisfies the following: if $b < c$: $x = O(\sqrt{y})$ while $y \to +\infty$, if $c < b$: $y = O(\sqrt{x})$ while $x \to +\infty$. At folds (A) $\leftrightarrow$ (B), any solution $(x,y)$ emanating from the fold equilibrium satisfies: if $b < c$: $y > x$ and $y \to +\infty$, if $c < b$: $x < y$ and $x \to +\infty$. In the upcoming blow-up analysis, these estimates can be employed to connect the corresponding orbits on the equator of the blow-up sphere to their ``exit equilibria'' in the charts $\kappa_{\pm a}$.
\end{remark}

Now, the jump possibilities given by Lemma \ref{bif_transitions} can be visualized by attaching the qualitative fast subsystems to the corresponding regions on the double cone from figure \ref{double_cone}, which represent the singularities of $\mathcal{S}_0$. This is shown in figure \ref{double_cone_fast}. In principle, this allows to concatenate candidate trajectories of the singular limit of system \eqref{hyp_umb}, if we add the desingularized slow flow $Y$ to the figure \ref{double_cone_fast} and respect the statements from Lemma  \ref{bif_transitions}. An example how to read figure \ref{double_cone_fast} and concatenate trajectories is given in the proof of Proposition \ref{prop:fast_jumps}.

\begin{figure}[h]
	\centering
	\begin{overpic}[width=0.85\textwidth]{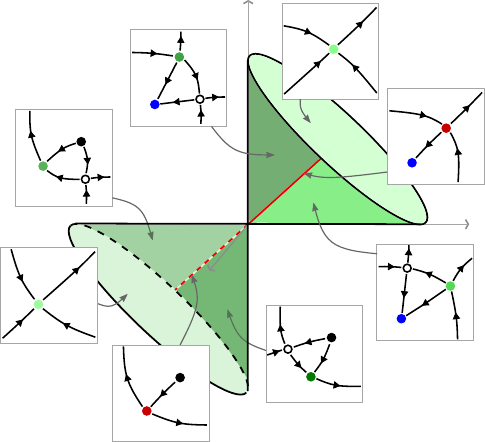}
		\put(45,34){\textcolor{gray}{$a$}}
		\put(96,47){\textcolor{gray}{$x$}}
		\put(54,90){\textcolor{gray}{$y$}}
		\put(32,25){\textcolor{blue}{$\mathcal{S}_0^{\txta}$}}
		\put(70,61){$\mathcal{S}_0^{\txtr}$}
		\put(37,55){\textcolor{gray}{$\mathcal{S}_0^{\txts}$}}
	\end{overpic}
	\caption{Qualitative fast subsystems corresponding to the singularities on the double cone for $a \neq 0$, each fast subsystem phase portrait is attached to a set of singularities on the cone. The $a$-axis is directed towards the observer. Cusp points are depicted in red, fold points in green, sinks in blue, sources in black and saddles in white. Recall that equilibria appearing in the fast subsystem correspond to instances of $\mathcal{S}_0^{\txta}, \mathcal{S}_0^{\txtr}, \mathcal{S}_0^{\txts}$ or singular points of $\mathcal{S}_0$, as indicated by the color coding. The phase portraits only contain few heteroclinic connections and ``escaping'' trajectories, but give a qualitatively complete picture according to Lemma \ref{bif_transitions}. For future reference, it is worth noting that the only trajectories that can jump towards $\mathcal{S}_0^\txta$ are those jumping at the upper cone.}
	\label{double_cone_fast}
\end{figure}

Roughly speaking, the feasible jumps from Lemma \ref{bif_transitions} and figure \ref{double_cone_fast} are summarized in the following proposition.
\begin{prop}[Feasible singular limit jumps]
	\label{prop:fast_jumps}
	The fast-slow system \eqref{hyp_umb} allows the following singular limit jump behavior via fast segments:
	\begin{enumerate}[(a)]
		\item It is impossible to jump via a fast candidate trajectory from folds or cusps enclosing $\mathcal{S}_0^{a}$ back onto the critical manifold $\mathcal{S}_0$.
		\item From folds enclosing $\mathcal{S}_0^\txtr$ for $a > 0$, there exist fast candidate trajectories jumping to $\mathcal{S}_0^{\txta}$ and $\mathcal{S}_0^{\txts}$.
		\item From cusps next to $\mathcal{S}_0^\txtr$ for $a > 0$, there exist fast candidate trajectories jumping to $\mathcal{S}_0^{\txta}$.
		\item From folds enclosing $\mathcal{S}_0^\txtr$ for $a = 0$, there exists a single fast candidate trajectory jumping to a fold enclosing $\mathcal{S}_0^{\txta}$.
		\item Any other jumping candidate trajectory escapes towards infinity and cannot arrive at $\mathcal{S}_0$.
	\end{enumerate} 
\end{prop}
\begin{proof}
	(a): From figure \ref{4config} one can deduce that folds and in particular cusps next to $\mathcal{S}^{\txta}_0$ are of type (C) $\leftrightarrow$ (D), in particular for $a < 0$, compare to figure \ref{bif_fast}. Then Lemma \ref{bif_transitions} implies assertion (a), which is clarified by Remark \ref{rem:foldcusp}. In figure \ref{double_cone_fast} this is visualized as follows: Starting inside $\mathcal{S}^{\txta}_0$, we assume that a suitable slow flow brings us to a fold surrounding the attracting region. Then we look to the four attached fast subsystems, which qualitatively show the fast dynamics at these folds or cusps. In each of these portraits we need to follow trajectories emanating from the fold equilibrium (green) or cusp (red), which all escape towards infinity.\\
	(b) and (c): Similarly, this follows from Lemma \ref{bif_transitions} by checking that a heteroclinic connection from the fold equilibrium to the saddle indeed exists. In figure \ref{double_cone_fast}, this is visualized by considering the four fast subsystems attached to the upper cone. We see that in this case connections to the sink (blue) and saddle (white) exist for $a > 0$.\\
	(d): This occurs qualitatively in the setting of the lower right phase portrait of figure \ref{proof_figs}, which is (d) in Lemma \ref{bif_transitions}.\\
	(e): In the language of Lemma \ref{bif_transitions} this corresponds to arrival at the exit section $\Delta^{\ex} = \{ x+ y = 2\nu \}$. In the scope of this section $\nu > 0$ can be chosen arbitrarily.
\end{proof}

The fact that there cannot be jumps from fold points enclosing $\mathcal{S}_0^{\txta}$ back to $\mathcal{S}_0^{\txta}$ was already proven in \cite{kojakhmetov2014constrained}. However, note that the latter did not assume any dynamics in the fast fibers, since there Takens' setting of constrained differential equations was used.
Proposition \ref{prop:fast_jumps} shows that a return mechanism from $\mathcal{S}_0^{\txta}$ back to $\mathcal{S}_0^{\txta}$ solely via folds or cusps is impossible. However, via a singular canard through the hyperbolic umbilic point one can get a candidate trajectory from $\mathcal{S}_0^{\txta}$ to $\mathcal{S}_0^{\txtr}$ (or $\mathcal{S}_0^{\txts}$) and possibly jump back to $\mathcal{S}_0^{\txta}$ from the upper cone of singularities. This was already observed in Remark \ref{rem:foldcusp}. By Lemma \ref{sf:origin_lin} a singular canard through the hyperbolic umbilic cannot occur for $B_0 C_0 > 0$.

Moreover, the above results give the idea that generically the hyperbolic umbilic point will  behave ``like a catapult''. In particular, any trajectory jumping at a singularity seems to escape generically along the positive $x$- or $y$-direction. Note that for trajectories approaching folds on the upper cone from the saddle-type region $\mathcal{S}_0^\txts$ there is the possibility for a relaxation oscillation in the singular limit: Hitting a fold on the upper cone in $a > 0$, the trajectory might then jump to the saddle-type region ``on the other side'' of the cone, from where, upon hitting another fold on the front of the upper cone, the trajectory might jump back to the starting point. A rigorous analysis of these relaxation oscillations, i.e.\ the question wether they possibly persist for $\eps > 0$ is beyond the scope of this document. Once an orbit jumps onto the attracting part upon hitting a fold point, we observe the ``catapult behavior'' again, if we assume that the singularities serve indeed as jump points. Generically, a fold point is indeed a jump point. In general, there possibly exist folded singularities on the double cone, which can give rise to canard trajectories passing to nearby regions of $\mathcal{S}_0$ with distinct stability properties, see e.g.\ \cite{szmolyan2001canards,kuehn2015multiple}. 

Next, we are going to investigate the dynamics in the vicinity of the hyperbolic umbilic point by means of a blow-up, which is the main part of the present article.

\section{Blow-up analysis}
\label{sec:blowup}

We are interested in the behavior of attracting slow manifolds of \eqref{hyp_umb} in a vicinity of the hyperbolic umbilic point for small $\eps > 0$. We investigate \eqref{hyp_umb} by means of the following vector field $X$, which is of the form
\begin{equation}
	\begin{split}
		\label{hyp_umb_exp}
		x' &= x^2 + ay + b + \eps f_x(x,y,a,b,c,\eps)\\
		y' &= y^2 + ax + c + \eps f_y(x,y,a,b,c,\eps)\\
		a' &= \eps g_a(x,y,a,b,c,\eps)\\
		b' &= \eps g_b(x,y,a,b,c,\eps)\\
		c' &= \eps g_c(x,y,a,b,c,\eps)\\
		\eps' &= 0.
	\end{split}
\end{equation}
We assume that the functions $f_x,f_y,g_a,g_b,g_c$ are sufficiently smooth. Recall that we use the notation $A_0 = g_a(0), B_0 = g_b(0), C_0 = g_c(0)$. According to the hypothesis of Theorem \ref{thm:main_thm}, and motivated by Lemma \ref{sf:origin_lin}, we shall be concerned with the case $B_0,C_0 > 0$. 

\subsection{Preparations for blowing up}
In the following we blow-up the vector field $X$ at the origin. That is, we ``replace'' the origin by the 5-d sphere $\mathbb S^5$. First, we set up the notation that is heavily used in the following analysis. With these we also briefly recall the blow-up method, see \cite{hjk2021survey,kuehn2015multiple} for further details. 
More precisely, we transform the vector field $X$ \eqref{hyp_umb_exp} by the weighted (quasi-homogeneous) blow-up transformation
\begin{equation}
	\begin{split}
		\label{global_blowup}
		\Phi\colon &B:= \mathbb S^5 \times [0,\tilde{r}) \subset \R^7 \to \R^6\\
		&(\bar{x},\bar{y},\bar{a},\bar{b},\bar{c},\bar{\eps},\bar{r}) \mapsto (\bar{r}\bar{x},\bar{r}\bar{y},\bar{r}\bar{a},\bar{r}^{2}\bar{b},\bar{r}^{2}\bar{c}, \bar{r}^{3}\bar{\eps}),
	\end{split}
\end{equation}
where $\bar{x}^{2} + \bar{y}^2 + \bar{a}^{2} + \bar{b}^{2} + \bar{c}^{2} + \bar{\eps}^2 = 1$. The map $\Phi$ is a diffeomorphism for $\bar{r} > 0$, to where we can pullback the vector field $X$. There exists a sufficiently smooth extension $\bar{X}$ on $B$, such that the pushforward of $\bar{X}$ is $X$, i.e.\ $\Phi_{\ast} \bar{X} = X$. Note that by this procedure, we have ``blown up'' the hyperbolic umbilic singularity at the origin to a sphere $\mathbb S^5 \times \{0\}$. Typically, by blowing up a sufficiently degenerate singularity, $\bar{X}$ vanishes on the blow-up sphere $ \{\bar{r} = 0\} \simeq \mathbb S^5$. The key point of the blow-up method is that one aims to rescale $\bar{X}$ (dividing out a common factor) in such a way, that distinct (semi-)hyperbolic equilibria on $\{\bar{r} = 0\}$ appear. By this, the orbit structure of $\bar{X}$ in $\{\bar{r} > 0\}$ remains unchanged. Often this rescaling is taken out locally, i.e.\ by studying $\bar{X}$ locally in charts. The general goal is then to infer the dynamics in $\{\bar{r} > 0\}$ close to the sphere from the dynamics on the invariant blow-up sphere $\{\bar{r} = 0\}$. 

In the following we refer to $\mathbb S^5$ as the blow-up sphere and identify it with $\mathbb S^5 \times \{0\} \subset \mathbb S^5 \times [0,\tilde{r}) = B$, where we call $B$ the blow-up space. Note that the blow-up space and the blow-up sphere are (sub-)manifolds. To work on them in a convenient way one usually introduces charts, in which $\Phi$ takes the form of a ``directional blow-up''. 

Let us denote these ``usual'' charts by $\kappa_{\pm \bullet}$, where $\bullet$ corresponds to a coordinate direction. Then
\begin{equation}
	\begin{split}
		\Phi_{\pm \bullet} := \Phi \circ \kappa_{\pm \bullet}^{-1}
	\end{split}
\end{equation}
is simply given by putting $\bar{\bullet} = \pm 1$ in \eqref{global_blowup} and renaming the variables accordingly. This notation comes from the fact that in $\kappa_{\pm \bullet}$ the blow-up space $B\rvert_{\pm \bar{\bullet} > 0}$ is visible. Let us give a brief example to clarify this notation: The chart $\kappa_{-y}$ is precisely designed in such a way, that the blow-up $\Phi$ in chart $\kappa_{-y}$ is given by
\begin{equation}
	\begin{split}
		\Phi_{-y} \colon (x_1,r_1,a_1,b_1,c_1,\eps_1) \mapsto (r_1 x_1, - r_1, r_1 a_1, r_1^{2} b_1, r_1^{2} c_1, r_1^{3} \eps_1),
	\end{split}
\end{equation}
where $(x_1,r_1,a_1,b_1,c_1,\eps_1)$ are the local coordinates determined by $\kappa_{-y}$. Note how $\Phi_{-y}$ is obtained from $\Phi$ \eqref{global_blowup} by putting $\bar{y} = -1$ and relabelling the coordinates.

\begin{remark}
	The blow-up exponents in \eqref{global_blowup} are chosen in such a way, that they match the quasi-homogeneity of $X$ in the fast variables and allow for rescaling in the slow variables. For example, in the chart $\kappa_{- y}$, all terms in the first component in the pullback of $X$ via $\Phi_{-y}$ (i.e.\ the equation for $x_1'$) share the same power of $r_1$ as factor. This allows to rescale and obtain nontrivial dynamics on $\{ \bar{r} = 0 \}$. Moreover, with this exponent choice, the critical manifold of $X$ appears as 2-submanifold on the blow-up sphere $\{ \bar{r} = 0 \}$.
\end{remark}

We aim to study $\bar{X}$ by means of the directional blow-ups $\Phi_{\pm \bullet}$ in the charts $\kappa_{\pm \bullet}$. For reference, we state all blow-up maps $\Phi_{\bullet}$ that are to be used in the forthcoming analysis. This also fixes the naming conventions for the coordinates determined by the charts $\kappa_\bullet$. The blow-up maps we are going to use explicitly are:
\begin{equation}\label{eq:maps}
\begin{aligned}
		\Phi_{-y}\colon& x = r_1 x_1, &&y = -r_1, &&a = r_1 a_1,&& b = r_1^2 b_1,\, c = r_1^{2} c_1,\, \eps = r_1^3 \eps_1;\\
		\Phi_{\eps}\colon& x = r_2 x_2, &&y = r_2 y_2, &&a = r_2 a_2,&& b = r_2^2 b_2,\, c = r_2^{2} c_2,\, \eps = r_2^3;\\
		\Phi_{a}\colon& x = r_3 x_3, &&y = r_3 y_3, &&a = r_3,&& b = r_3^2 b_3,\, c = r_3^{2} c_3,\, \eps = r_3^3 \eps_3;\\
		\Phi_{-a}\colon& x = r_4 x_4, &&y = r_4 y_4, &&a = -r_4,&& b = r_4^2 b_4,\, c = r_4^{2} c_4,\, \eps = r_4^3 \eps_4;\\
		\Phi_{\ex}\colon& x = r_5(1 + x_5), &&y = r_5(1-x_5), &&a = r_5 a_5,&& b = r_5^2 b_5, \, c = r_5^2 c_5,\, \eps = r_5^3 \eps_5.
\end{aligned}
\end{equation}
We emphasize the slightly ``unusual'' blow-up $\Phi_{\ex}$, which is roughly obtained by including a $45^{\circ}$ degree rotation in the fast variables. For further details on $\Phi_{\ex}$ refer to the exit chart analysis in section \ref{s:exit_chart}. 

We now give an overview of the forthcoming blow-up analysis, which also describes which charts/blow-ups we will use. The goal is to understand the dynamics of the fast-slow vector field $X$ close to the origin. Notice that $X$ has a manifold of equilibria $\mathcal{S}_0 \times \{\eps = 0\}$, i.e., the critical manifold. By Fenichel theory \cite{fenichel1979,jones1995} compact submanifolds of the normally hyperbolic regions of $\mathcal{S}_0$ perturb for sufficiently small $\eps > 0$ to (non-unique) invariant slow manifolds, on which the flow converges to the slow flow on the critical manifold for $\eps \to 0$. These invariant slow manifolds appear in sections $\eps = \text{constant}$ of $4$-dimensional center manifolds along $\mathcal{S}_0^{\txta} \times \{\eps = 0\}, \mathcal{S}_0^{\txtr} \times \{\eps = 0\}$ and $\mathcal{S}_0^{\txts} \times \{\eps = 0\}$. Trajectories, which approach slow manifolds, will generically approach slow manifolds obtained as perturbations of the attracting region $\mathcal{S}_0^{\txta}$. Refer to figure \ref{double_cone} for the region $\mathcal{S}_0^{\txta}$. In the blow-up of $X$, the hyperbolic umbilic singularity is ``replaced'' by the $5$-d sphere $\mathbb S^5 \simeq B\rvert_{\{\bar{r} = 0\}}$. In all of the charts, we will see certain regions of phase space and away from $\{\bar{r} = 0\}$ all objects, e.g. invariant manifolds or the critical manifold, appear in a locally diffeomorphic manner. In the singular limit $\eps = 0$, the blow-up space is restricted to $\{\bar{\eps} = 0 \}$, and the blow-up sphere $\mathbb S^5$ reduces to $B\rvert_{\{\bar{\eps} = 0,\bar{r} = 0\}} \simeq \mathbb{S}^4$, to which we will refer as equator of the blow-up sphere. 

\begin{remark}
	The following notational convention will be used: objects in original coordinates, which are visible in a blow-up $\Phi_{\bullet}$ will be indexed by the corresponding index used for the new coordinates in $\Phi_{\bullet}$. For example, $\mathcal{S}_0^{\txta}$ seen in $\Phi_{-y}$ (i.e.\ the blow-up in $\kappa_{-y}$) will be denoted by $\mathcal{S}_{0,1}^{\txta}$. Similarly, the section $\Delta^{\ex}$ seen in $\Phi_{\ex}$ will be denoted by $\Delta^{\ex}_5$. However, we will always simply use the variable $t$ to denote time (i.e.\ the variable to parametrize integral curves) in each of the blown-up vector fields and we remark that the meaning of $t$ will be clear from the context.
\end{remark}

In the following we present the analysis in the entry chart $\kappa_{-y}$, the rescaling chart (also called central chart) $\kappa_{\eps}$ and the exit chart $\kappa_{\ex}$. The charts $\kappa_{\pm a}$ are only relevant in section \ref{sec:further} and are not used to prove Theorem \ref{thm:main_thm}. However, in \eqref{eq:maps}, $\Phi_{\pm a}$ appear before the exit chart blow-up $\Phi_{\ex}$, since in principle one would study transitions from the entry chart $\kappa_{-y}$ to $\kappa_{\pm a}$ in order to analyze the dynamics away from the hyperbolic umbilic singularity.
To cover all transitions onto the fast regime in a single exit chart, we use a slightly ``unusual'' chart $\kappa_{\ex}$, which includes, up to scaling, a $45^{\circ}$ rotation of the fast variables of $X$. In our notation, $\kappa_{\ex}$ intuitively corresponds to ``$\kappa_{x+y}$''. The corresponding blow-up map $\Phi_{\ex}$ can then also be roughly understood as blowing up $X$ ``in direction of $x+y$''. This choice is motivated by \eqref{fast_subsystem_exit} and Lemma \ref{bif_transitions}. After the analysis in the three charts $\kappa_{-y}$, $\kappa_{\eps}$, $\kappa_{\ex}$ we are in position to prove Theorem \ref{thm:main_thm}.
Finally, we also present a brief analysis of the dynamics on the blow-up sphere in the two charts $\kappa_{\pm a}$ in section \ref{sec:further}, which are useful to understand the desingularized slow flow near the hyperbolic umbilic singularity and the flow on the blow-up sphere close to the equator $\{\bar{\eps} = \bar{r} = 0\}$.

\subsection{Analysis in the entry chart $\kappa_{-y}$}
\label{sec:entry}

Blowing up the vector field $X$ \eqref{hyp_umb_exp} by $\Phi_{-y}$, as defined in \eqref{eq:maps}, and desingularizing accordingly by dividing out the common factor $r_1$, we obtain the blow-up in $\kappa_{-y}$:
\begin{equation}
	\begin{split}
		\label{blowup_-y}
		x_1' &=  x_1^2 - a_1 + b_1 + x_1 (1+ a_1 x_1 + c_1) + r_1\eps_1 (f_x + x_1 f_y) \\
		r_1' &= -r_1(1 + a_1 x_1 + c_1) - r_1^2\eps_1 f_y \\
		a_1' &= a_1(1 + a_1 x_1 + c_1) + r_1\eps_1 (g_a + a_1 f_y)\\
		b_1' &= 2b_1 (1 + a_1 x_1 + c_1) + \eps_1 g_b + 2r_1 \eps_1 b_1 f_y\\
		c_1' &= 2c_1 (1 + a_1 x_1 + c_1) + \eps_1 g_c + 2r_1 \eps_1 c_1 f_y\\
		\eps_1' &= 3 \eps_1(1 + a_1 x_1 + c_1) + 3r_1\eps_1^2 f_y.
	\end{split}
\end{equation}
In \eqref{blowup_-y} we use the shorthand $f_x $ for $ f_x \circ \Phi_{-y}$, similarly for $f_y,g_a,g_b,g_c$.
Recall that dividing out the common factor $r_1$ amounts to a time reparametrization of trajectories away from the blow-up sphere.
The system \eqref{blowup_-y} has invariant subspaces $\{r_1 = 0\}, \{\eps_1 = 0\}$ and $\{r_1 = 0, \eps_1 = 0\}$. 
The subspace $\{r_1 = 0\}$ corresponds to dynamics on the blowup sphere $\mathbb S^5$. Analogously, the subspace $\{\eps_1 = 0\}$ corresponds to the singular limit and the subspace $\{r_1 = 0, \eps_1 = 0\}$ to the part of the $4$-dimensional equator of $\mathbb S^5$, where $\bar{y} < 0$ holds. Note the constant of motion $r_1^3\eps_1$, induced by the blown-up invariant sets $\{\eps = \text{const}\}$ of $X$.

In $\{r_1 = 0, \eps_1 = 0\}$ the system \eqref{blowup_-y} reduces to
\begin{equation}
	\begin{split}
		x_1' &=  x_1^2 + (1+ a_1 x_1 + c_1) x_1 - a_1 + b_1  \\
		a_1' &= a_1(1 + a_1 x_1 + c_1)\\
		b_1' &= 2b_1 (1 + a_1 x_1 + c_1) \\
		c_1' &= 2c_1 (1 + a_1 x_1 + c_1),
	\end{split}
\end{equation}
which has two isolated equilibria
\begin{equation}
	\begin{split}
		\label{q1q2}
		q_1 &:= (x_1 = 0, a_1 = 0, b_1 = 0, c_1 = 0 )\\
		q_2 &:= (x_1 = -1, a_1 = 0, b_1 = 0, c_1 = 0 )
	\end{split}
\end{equation}
and a 2-manifold of equilibria determined by
\begin{equation}
	\begin{split}
		\label{s_on_sphere}
		x_1^2 + b_1 - a_1 &= 0\\
		1 + a_1 x_1 + c_1 &= 0.
	\end{split}
\end{equation}

In $\{\eps_1 = 0\}$ the system \eqref{blowup_-y} simplifies to
\begin{equation}
	\begin{split}
		\label{-y_singular_limit}
		x_1' &=  x_1^2 + (1+ a_1 x_1 + c_1) x_1 - a_1 + b_1 \\
		r_1' &= -r_1(1 + a_1 x_1 + c_1) \\
		a_1' &= a_1(1 + a_1 x_1 + c_1) \\
		b_1' &= 2b_1 (1 + a_1 x_1 + c_1) \\
		c_1' &= 2c_1 (1 + a_1 x_1 + c_1) .
	\end{split}
\end{equation}
In addition to the isolated equilibria $q_1,q_2$ on the blow-up sphere, system \eqref{-y_singular_limit} has a 3-manifold of equilibria
\begin{equation}
	\begin{split}
		\label{E1}
		\mathcal{S}_{0,1} = \{ x_1^2 + b_1 - a_1 = 0,\, 1 + a_1 x_1 + c_1 = 0 \},
	\end{split}
\end{equation}
where the 2-manifold determined by \eqref{s_on_sphere} appears as slice $\mathcal{S}_{0,1}\eval_{r_1 = 0}$. Via $\Phi_{-y}$ the manifold $\mathcal{S}_{0,1}\eval_{r_1 > 0}$ corresponds to the critical manifold $\mathcal{S}_0\eval_{y < 0}$. One should recall that the hyperbolic umbilic point is an object of codimension 3 in the critical manifold. In blown-up coordinates, it corresponds to the intersection of the 3-dimensional critical manifold with the 5-dimensional blow-up sphere $\mathbb S^5$, which gives a 2-dimensional manifold. This 2-manifold will appear in multiple charts and is contained in the singular limit. That is, it is contained in the equator $\{\bar{\eps} = 0, \bar{r} = 0\} \simeq \mathbb S^4$ of the blow-up sphere $\mathbb S^5$, which precisely is given by \eqref{s_on_sphere} in the chart $\kappa_{-y}$. The manifold $\mathcal{S}_{0,1}$ is a graph over the $(x_1,r_1,a_1)$-coordinates. 

\begin{figure}
	\centering
	\begin{overpic}[width=.65\textwidth]{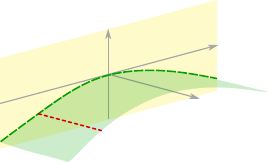}
		\put(60,10){$\mathcal{S}_{0,1}^{\txta}$: attracting}
		\put(88,38){$\mathcal{S}_{0,1}^{\txts}$: saddle-type}
		\put(76,23){\textcolor{gray}{$r_1$}}
		\put(83,44){\textcolor{gray}{$a_1$}}
		\put(39,51){\textcolor{gray}{$x_1$}}
		\put(10,5){\rotatebox{-16}{\textcolor{dark-green}{folds}}}
		\put(25,16){\rotatebox{-16}{\textcolor{dark-red}{cusps}}}
		\put(3,34){\rotatebox{15}{hyperbolic umbilic}}
	\end{overpic}
	\caption{Stratification of the blown-up critical manifold in chart $\kappa_{-y}$, given by Lemma \ref{E1_regions}.}
	\label{fig:critical_mfld_-y}
\end{figure}

\begin{lemma}
	\label{E1_regions}
	The manifold $\mathcal{S}_{0,1}\eval_{r_1 \ge 0}$ is subdivided into an attracting region $\mathcal{S}_{0,1}^{\txta} := \{4x_1 + a_1^2 < 0, r_1 > 0\}$, a saddle-type region $\mathcal{S}_{0,1}^{\txts} := \{4x_1 + a_1^2 > 0, r_1 > 0\}$, a fold surface given by $\{4x_1 + a_1^2 = 0, r_1 > 0\}$, which contains the cusp line $\{x_1 = -1, a_1 = -2, r_1 > 0\}$. The 2-manifold $\mathcal{S}_{0,1}\eval_{r_1 = 0}$ represents the blown-up hyperbolic umbilic singularity.
\end{lemma}
\begin{proof}
	Lemma \ref{classif_singular_points} blown up by $\Phi_{-y}$. See figure \ref{fig:critical_mfld_-y}.
\end{proof}

For the following it is convenient to rectify the critical manifold $\mathcal{S}_{0,1}$ by introducing a coordinate transformation $T \colon x_1 = x_1, r_1 = r_1, a_1 = a_1, u = x_1^2 + b_1 - a_1, v = 1 + a_1x_1 + c_1, \eps_1 = \eps_1$. Applying this transformation to \eqref{blowup_-y} leads to
\begin{equation}
	\begin{split}
		\label{entry:rectification}
		x_1' &= u + x_1 v + r_1 \eps_1 (f_x + x_1 f_y)\\
		r_1' &= -r_1 v - r_1^2 \eps_1 f_y\\
		a_1' &= a_1 v + r_1 \eps_1 (g_a + a_1 f_y)\\
		u' &= 2x_1 u + a_1 v + 2uv + \eps_1 g_b + O(r_1 \eps_1)\\
		v' &= a_1 u -2v + 2v^2 + \eps_1 g_c + O(r_1 \eps_1)\\
		\eps_1' &= 3 \eps_1 v + 3r_1 \eps_1^2 f_y.
	\end{split}
\end{equation}
Again, we use the shorthand $f_x = f_x \circ \Phi_{-y} \circ T^{-1}$, similarly for $f_x,f_y,g_a,g_b,g_c$. Each of the $O(r_1\eps_1)$ terms is of the form $r_1 \eps_1 \tilde{f}$, for a sufficiently smooth function $\tilde{f}$.
The effect of $T$ is that the critical manifold $\mathcal{S}_{0,1}$ in \eqref{entry:rectification} is now given by $\{u = 0, v= 0, \eps_1 = 0\}$.

\begin{lemma}
	\label{lemma:lin_along_critical_manifold}
	Along $\mathcal{S}_{0,1}^{\txta}$ there exist 4-dimensional attracting center manifolds $W^c(\mathcal{S}_{0,1}^{\txta})$, which can be extended to $r_1 = 0$. Via $\Phi_{-y}$ the center manifolds along $\mathcal{S}_0^{\txta} \times \{ \eps = 0\}$ (in original coordinates) correspond to instances of $W^c(\mathcal{S}_{0,1}^{\txta})\rvert_{r_1 > 0}$. The attracting slow manifolds obtained by Fenichel theory coincide with intersections of the form $W^c(\mathcal{S}_{0,1}^{\txta}) \cap \{ r_1^3 \eps_1 = \text{const} \}$.
\end{lemma}
\begin{proof}
	Linearization of \eqref{entry:rectification} along $\mathcal{S}_{0,1} \simeq \{\eps_1 = 0, u = 0, v = 0\}$ is given by the matrix
	\begin{equation}
		\begin{split}
			\label{entry_lin}
			\begin{bmatrix}
				0 & 0 & 0 & 1 & x_1 & O(r_1)\\
				0 & 0 & 0 & 0 & -r_1 & O(r_1^2)\\
				0 & 0 & 0 & 0 & a_1 & O(r_1)\\
				0 & 0 & 0 & 2x_1 & a_1 & B_0+ O(r_1)\\
				0 & 0 & 0 & a_1 & -2 & C_0+ O(r_1)\\
				0 & 0 & 0 & 0 & 0 & 0
			\end{bmatrix}.
		\end{split}
	\end{equation}
	The matrix \eqref{entry_lin} has a quadruple zero eigenvalue and two further eigenvalues $\lambda_{1,2} = x_1 - 1 \pm \sqrt{ (1 - x_1)^2 + 4 x_1 + a_1^2}$. Notice that $\lambda_{1,2} < 0$ along $\mathcal{S}_{0,1}^{\txta}$ and $\lambda_1 > 0 > \lambda_2$ along $\mathcal{S}_{0,1}^{\txts}$.
	Note that along the fold surface given by $\left\{4x_1 + a_1^2 = 0\right
	\}$ we have $\lambda_1 = 0, \lambda_2 = 2(x_1 - 1)$, while $x_1 \le 0$. Since $\eps = r_1^{3}\eps_1$ is a constant of motion in $\kappa_{-y}$, the (blown-up) attracting slow manifolds are contained in intersections of the form $W^c(\mathcal{S}_{0,1}^{\txta}) \cap \{ r_1^3 \eps_1 = \text{const} \}$.
\end{proof}

\begin{remark}
	The equilibria $q_1, q_2$ from \eqref{q1q2} arise from interaction with the fast subsystem.
	For further illustration, refer to figure \ref{fast_blowup} and the explanation next to \eqref{fast_restricted}.
\end{remark}

\subsection{Blow-up of the desingularized slow flow $Y$}
\label{sec:desing_sf_blowup}

The desingularized slow flow $Y$ \eqref{sf:desing} constitutes a singular limit slow flow on the critical manifold $\mathcal{S}_0$. It is worth recalling the necessary time reversion for $Y$ in the saddle region $\mathcal{S}_0^\txts$ due to the desingularization process. A natural idea is to incorporate $Y$ into our blow-up analysis, i.e.\ the blow-up of the fast formulation $X$ \eqref{hyp_umb_exp}. By this we aim to study the singular limit interaction of $Y$ with the blown-up hyperbolic umbilic singularity. That is, we will ``embed'' $Y$ into the blow-up of $X$. Note that this step was also implicitly carried out in the analysis of a fold singularity or cusp singularity \cite{krupa2001extending,broer2013cusp,kojakhmetov2016cusp}. However, in these cases, the desingularized slow flow was regular and therefore considerably simpler to embed into the blow-up analysis. In contrast to this, we are dealing with a desingularized slow flow $Y$, which has a non-hyperbolic equilibrium at the hyperbolic umbilic singularity. Even in the studies of folded singularities \cite{szmolyan2001canards,mitry2017folded,de2021canard}, these appear as hyperbolic singularities of the desingularized slow flow.

Recall the blow-up transformation $\Phi$ given by \eqref{global_blowup}. Let $\Theta\colon \mathbb S^2 \times [0,\bar{r}) \to \R^3, (\bar{x},\bar{y},\bar{a},\bar{r}) \mapsto (\bar{r}\bar{x},\bar{r}\bar{y},\bar{r}\bar{a})$ be the blow-up arising from $\Phi$ restricted to the $(x,y,a)$-coordinates. Define the blown-up $\bar{Y}$ of the desingularized slow flow $Y$ by
\begin{equation}
	\begin{split}
		\label{slow_flow_complete_blowup}
		\Theta_{\ast}\bar{Y} = Y.
	\end{split}
\end{equation}
In fact, $\bar{Y}$ has nontrivial dynamics on the blow-up sphere $\mathbb S^2$ without employing a time rescaling. Note that one can use the charts $\kappa_{\pm x, \pm y, \pm a}$ to study $\bar{Y}$, as this just amounts to blowing up $Y$ by the corresponding directional blow-ups induced by $\Theta$. We denote the directional blow-ups induced by $\Theta$ with the same subscripts, i.e.\ $\Theta_{\bullet}$.

Since $Y$ determines the slow flow on the critical manifold $\mathcal{S}_0$, in blown-up coordinates $\bar{Y}$ determines the slow flow on the blown-up critical manifold in each of the charts $\kappa_{\pm x, \pm y, \pm a}$.
By this we mean the following: Looking at the blow-up of $X$ in chart $\kappa_\bullet$, i.e.\ blowing up $X$ with $\Phi_{\bullet}$, the critical manifold appears as a manifold of equilibria. The slow flow on this manifold of equilibria is then determined by the blow-up of $Y$ by $\Theta_{\bullet}$. Note that since the critical manifold is parametrized by $(x,y,a)$-coordinates, the coordinates used for parametrizing the critical manifold in the chart $\kappa_{\bullet}$ match with the coordinates used for blowing up $Y$ by $\Theta_{\bullet}$. This directly allows to ``embed'' the blown-up slow flow $\bar{Y}$ in each of the above charts into the blow-up of $X$.
For example, in the chart $\kappa_{-y}$, we have that the critical manifold appears as the manifold $\mathcal{S}_{0,1}$, which is parametrized by the coordinates $(x_1,r_1,a_1)$. These are precisely the coordinates, which are used for blowing up $Y$ by $\Theta_{-y}$. We will start with the analysis of $\bar{Y}$ in $\kappa_{-y}$ in this section. Similarly, one can employ the other charts to complete the analysis of $\bar{Y}$. Eventually, this leads to a blown-up version of figure \ref{desing_sketch}, i.e.\ a ``usual'' blow-up analysis of the vector field $\bar{Y}$, with time reversion in $\mathcal{S}_0^\txts$. A sketch of the flow of $\bar{Y}$ is shown in figure \ref{sf:desing_complete_blowup}.

To embed $Y$ into $\bar{X}$ in chart $\kappa_{-y}$, we take the desingularized slow flow $Y$ \eqref{sf:desing} and transform it by
\begin{equation}
	\begin{split}
		\label{restricted_blow-up}
		\Theta_{-y}\colon\, x =r_1 x_1,\, y = -r_1,\, a = r_1 a_1.
	\end{split}
\end{equation}
This gives
\begin{equation}
	\begin{split}
		\label{sf:-y} 
		\dot{x}_1 &= a_1 \tilde{g_c} - 2 \tilde{g_b} + r_1(a_1 x_1 - 2)\tilde{g_a} + x_1 J\\
		\dot{r}_1 &= - r_1 J\\
		\dot{a}_1 &= a_1 J - r_1(4x_1 + a_1^2)\tilde{g_a},
	\end{split}
\end{equation}
where $J = J(x_1,r_1,a_1) = a_1 \tilde{g}_b - 2x_1 \tilde{g}_c - r_1(a_1 + 2x_1^2)\tilde{g}_a$ and the shorthand $\tilde{g}_\bullet$ for $\tilde{g}_\bullet \circ \Theta_{-y}$ is used with $\tilde{g}_\bullet$ defined in \eqref{sf:desing}.
System \eqref{sf:-y} represents a vector field on the manifold $\mathcal{S}_{0,1}$ in the coordinates $(x_1,r_1,a_1)$ and, simultaneously, a ``usual'' blow-up of $Y$ in negative $y$-direction.
The invariant set $\{r_1 = 0 \}$ of \eqref{sf:-y} corresponds to the blown-up hyperbolic umbilic point. Refer to figure \ref{fig:critical_mfld_-y}, in which \eqref{sf:-y} now induces a flow.
Recall that for $Y$ we need to reverse the time orientation of trajectories in the saddle-type region. This we have to take into account for $\bar{Y}$ as well: In \eqref{sf:-y} we reverse the time-orientation of orbits in the saddle-type region $\left\{4x_1 + a_1^2 > 0\right\}$, now including $r_1 = 0$, in order to obtain a consistent flow.
\begin{remark}
	Another way to motivate the time reversion in $\{r_1 = 0, 4x_1 + a_1^2 > 0\}$ will be apparent from the charts $\kappa_{\pm a}$ in section \ref{sec:further}. In these charts, the regions $\{r_1 = 0, a_1 > 0, 4x_1 + a_1^2 > 0\}$ and $\{r_1 = 0, a_1 < 0, 4x_1 + a_1^2 > 0\}$ appear as saddle-type regions of 2-d critical manifolds. Hence, its desingularized slow flow undergoes a time reversion in such regions.
\end{remark}

Recall from our notation that $\tilde{g}_a(0) = A_0$, $\tilde{g}_b(0) = B_0$ and $\tilde{g}_c(0) = C_0$.
The vector field \eqref{sf:-y} has three equilibria in $\{r_1 = 0 \}$, which are located at
\begin{equation}
	\begin{split}
		\label{zetas}
		\zeta_1 &:= (x_1 = -B_0^2/C_0^2, a_1 = -2B_0/C_0),\\ \zeta_{2,3} &:= (x_1 = \mp \sqrt{B_0/C_0}, a_1 = 0).
	\end{split}
\end{equation}
Note that $\zeta_{2}$ and $\zeta_3$ only exist if $B_0,C_0$ have identical sign.
The equilibrium point $\zeta_1$ lies on $\left\{4x_1 + a_1^2 = 0\right\}$, which is the extension of the fold surface from Lemma \ref{E1_regions} to $\left\{r_1 = 0\right\}$. In case of $B_0 = C_0 \neq 0$, $\zeta_1$ extends the cusp line to $\left\{r_1 = 0\right\}$. In the vicinity of $\zeta_2$ lies the attracting part of the (blown-up) critical manifold $\mathcal{S}_{0,1}^{\txta}$, while $\zeta_3$ lies next to the saddle-type part $\mathcal{S}_{0,1}^{\txts}$. See figure \ref{fig:slow_flow_-y}. 
We now analyze in further detail the equilibria $\zeta_{1,2,3}$, since they organize the flow.

\begin{lemma}
	\label{sf:-y:analysis}
	Let $B_0,C_0 > 0$. With reversed time orientation in $\left\{4x_1 + a_1^2 > 0\right\}$, the points $\zeta_{2,3}$ are hyperbolic equilibria with the following stability:
	In the invariant plane $\{r_1 = 0\}$ $\zeta_{2,3}$ are sources. At each of $\zeta_{2,3}$ there exist unique 1-d stable manifolds transverse to the plane $\{r_1 = 0\}$. If the time reversion is neglected, the $x_1$-axis is invariant and establishes a heteroclinic connection from $\zeta_2$ to $\zeta_3$.
\end{lemma}
\begin{proof}
	Linearization of \eqref{sf:-y} at a point $(x_1,r_1 = 0,a_1)$ leads to
	\begin{equation}
		\begin{split}
			\label{sf:lin}
			\begin{bmatrix}
				B_0 a_1 - 4 C_0 x_1 & * & C_0 + x_1 B_0\\
				0 & -B_0 a_1 + 2 C_0 x_1 & 0\\
				-2 C_0 a_1 & * & 2B_0 a_1 - 2 C_0 x_1
			\end{bmatrix}.
		\end{split}
	\end{equation}
	The stars $*$ denote unspecified entries. At $\zeta_2$ this leads to eigenvalues $4\sqrt{B_0C_0}, 2 \sqrt{B_0C_0}$ in the invariant plane $\{r_1 = 0\}$ and $-2\sqrt{B_0C_0}$ with eigenvector transverse to $\{r_1 = 0\}$. At $\zeta_3$ there are eigenvalues $-4\sqrt{B_0C_0}, -2 \sqrt{B_0C_0}$ in the invariant plane $\{r_1 = 0\}$ and $2\sqrt{B_0C_0}$ with eigenvector transverse to it. Taking into account the reversed time orientation for $4x_1 + a_1^2 > 0$ implies the assertions. The hetercolinic connection along the invariant $x_1$-axis to $\zeta_3$ is clear. 
\end{proof}

\begin{figure}[h]
	\centering
	\begin{overpic}[width = 0.8\textwidth]{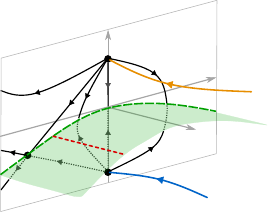}
		\put(9,25){$\zeta_1$}
		\put(42,60){$\zeta_3$}
		\put(42,11){$\zeta_2$}
		\put(80,12){$\mathcal{S}_{0,1}^{\txta}$}
		\put(85,55){$\mathcal{S}_{0,1}^{\txts}$}
		\put(74,28){\textcolor{gray}{$r_1$}}
		\put(78,53){\textcolor{gray}{$a_1$}}
		\put(35,64){\textcolor{gray}{$x_1$}}
		\put(62,7){\textcolor{light-blue}{$\sigma_1$}}
		\put(1,54){\rotatebox{13}{\textcolor{gray}{$\mathcal{S}_{0,1}\rvert_{r_1 = 0}$}}}
	\end{overpic}
	\caption{Sketch of the blown-up desingularized slow flow \eqref{sf:-y} for $B_0 > C_0  > 0$ close to $\mathcal{S}_{0,1}\rvert_{r_1 = 0}$. Time reversion is taken into account above the green surface of fold points, that is in $\mathcal{S}_{0,1}^{\txts}$. The black trajectories indicate the flow on the invariant plane $\{r_1 = 0\} \simeq \mathcal{S}_{0,1}\rvert_{r_1 = 0}$, containing $\zeta_{1,2,3}$. The trajectory $\sigma_1$ (the blown-up version of $\sigma$) approaches $\zeta_2$ from inside $\mathcal{S}_{0,1}^{\txta}$ and the orange trajectory reaches $\zeta_3$ from inside $\mathcal{S}_{0,1}^{\txts}$, see Lemma \ref{sf:origin_lin}. The center manifold based at $\zeta_1$ is not depicted.}
	\label{fig:slow_flow_-y}
\end{figure}

Recall the trajectory $\sigma$ from Lemma \ref{sf:origin_lin}, which represents the stable manifold of the origin in the attracting part of $\mathcal{S}_0$.
Let $\sigma_1$ denote $\sigma$ in the present blown-up coordinates.

\begin{remark}
	Recall figure \ref{desing_sketch}. Via $\Theta_{-y}$ the equilibrium $\zeta_2$ serves as arrival point for the trajectory $\sigma$ at the blown-up hyperbolic umbilic. The equilibrium $\zeta_{3}$ serves as arrival point for the unique trajectory, which approaches the hyperbolic umbilic point from the saddle-type region $y < 0$. Similarly, there will exist two more equilibria $\zeta_4$ and $\zeta_5$ of $\bar{Y}$ on the blow-up sphere, which correspond to the two other trajectories, which arrive the hyperbolic umbilic from $y > 0$.
\end{remark}

Refer to figure \ref{fig:slow_flow_-y} for a sketch of the flow of \eqref{sf:-y} close to the invariant plane $\{r_1 = 0\}$. Note that Lemma \ref{sf:-y:analysis} shows that the trajectory $\sigma_1$ is transverse to the plane $\{r_1 = 0\}$. This implies that $\sigma_1$ is transverse to $\{ r_1 = \nu \}$, for sufficiently small $\nu > 0$.

There exists the third equilibrium $\zeta_1$ in $\{r_1 = 0\}$, which lies on $4x_1 + a_1^2 = 0$, the extension of the fold surface to $r_1 = 0$. Employing the linearization \eqref{sf:lin} gives

\begin{lemma}
	Let $B_0 C_0 \neq 0$. In $\{r_1 = 0\}$, without taking the time reversion into account, $\zeta_1$ has eigenvalues $\pm 2\sqrt{B_0 C_0}$. Hence $\zeta_1$ is a saddle if $B_0 C_0 > 0$ and $\zeta_1$ is a center if $B_0 C_0 < 0$. There exists a 1-d center manifold at $\zeta_1$, transverse to $\{r_1 = 0\}$.
\end{lemma}

In fact, in the charts $\kappa_{\mp a}$, there appear fast-slow systems on the blow-up sphere, in which $\zeta_1$ is a folded saddle, respectively a folded center. For more details on the folded singularities arising in the blow-up analysis we refer to section \ref{sec:further}. 

\begin{remark}
	For the following analysis we are interested in the flow close to $\zeta_2$, since the trajectory $\sigma_1$ arrives there. In particular, we will not investigate the role of the center manifold at $\zeta_1$ further.
\end{remark}

Until this point, we have analyzed the blow-up $\bar{Y}$ in the chart induced by $\kappa_{-y}$. For a full picture of $\bar{Y}$ it remains to study $\bar{Y}$ in other charts to obtain the complete dynamics of $\bar{Y}$ close to the blow-up sphere $\mathbb{S}^2$. In $\kappa_{-x}$ (and similarly in $\kappa_{x}$ or $\kappa_{y}$) the findings are similar to $\kappa_{-y}$ and we omit the presentation here. 

\begin{remark}
	The heteroclinic connections from $\zeta_{2,3}$ to $\zeta_1$ as shown in figure \ref{fig:slow_flow_-y} are implied by the analysis in chart $\kappa_{-a}$ by Lemma \ref{kappa_-a_heterocl}. In particular, the stable/unstable manifolds of $\zeta_1$ in $\{r_1 = 0\}$ are transverse to the extended fold surface for $B_0 \neq C_0$.
\end{remark}

Employing several charts to analyze $\bar{Y}$, we obtain the complete blow-up of $\bar{Y}$, as visualized in figure \ref{sf:desing_complete_blowup}. Once again, we emphasize the time reversion in the saddle-type region, which needs to be taken into account. The qualitative flow in figure \ref{sf:desing_complete_blowup} is implied by the analysis in the charts $\kappa_{\pm a}$ in section \ref{sec:further}, in particular \ref{ss:further:y}. On the blow-up sphere there exist three further equilibria which we denote by $\zeta_{4}$, $\zeta_5$ and $\zeta_6$. The equilibria $\zeta_{5,6}$ are analogous to $\zeta_{2,1}$, but are located next to the repelling region of the critical manifold. The role of the equilibrium $\zeta_4$ is analogous to $\zeta_3$. We remark that figure \ref{sf:desing_complete_blowup} is the blown-up version of figure \ref{desing_sketch}.

\begin{figure}[h]
	\centering
	\begin{overpic}[width=.6\textwidth]{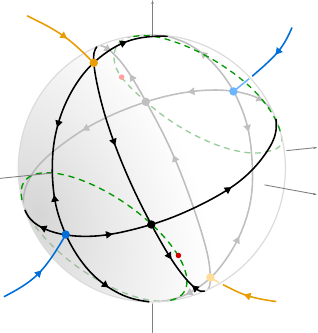}
		\put(8,10){\textcolor{light-blue}{$\sigma$}}
		\put(90,37){\textcolor{gray}{$-\bar{a}$}}
		\put(93,51){\textcolor{gray}{$\bar{x}$}}
		\put(48,97){\textcolor{gray}{$\bar{y}$}}
		\put(46.5,36.5){$\zeta_1$}
		\put(21,32){$\zeta_2$}
		\put(22,80){$\zeta_4$}
		\put(44,73){\textcolor{light-gray}{$\zeta_6$}}
		\put(63.5,22){\textcolor{light-gray}{$\zeta_3$}}
		\put(66.5,67.5){\textcolor{light-gray}{$\zeta_5$}}
	\end{overpic}
	\caption{Complete blow-up $\bar{Y}$ of the desingularized slow flow $Y$ \eqref{sf:desing} for $B_0 > C_0 > 0$ at the origin, with time reversion in the saddle-type region taken into account. The green dashed curves depict the fold curves, which bound the attracting/repelling/saddle-type regions on the blow-up sphere. We remark that this figure shows a blown-up version of figure \ref{desing_sketch}, but for visualization reasons the negative $a$-direction points here towards the observer. On the blow-up sphere $\mathbb{S}^2$, only trajectories contained in the stable/unstable manifolds of the folded saddles $\zeta_1,\zeta_6$ and trajectories on the equator $\bar{a} = 0$ are depicted. These organize the flow on $\mathbb{S}^2$. For further details refer to section \ref{sec:further}, in particular \ref{ss:further:y}. The red points depict the cusp points, i.e.\ where the cusp curves touch the blow-up sphere.}
	\label{sf:desing_complete_blowup}
\end{figure}

\begin{remark}
	The idea of the following analysis is to track small perturbations of the trajectory $\sigma \simeq \sigma_1$ for small $\eps > 0$ through the blow-up. The importance of $\sigma \simeq \sigma_1$ is that in the case $B_0,C_0 > 0$ it approaches the blow-up sphere at $\zeta_2$ from inside $\mathcal{S}^{\txta}_0$, see figures \ref{fig:slow_flow_-y} and \ref{sf:desing_complete_blowup}. Consequently, we will focus on a neighborhood of $\zeta_2$ for the transition into the rescaling chart $\kappa_{\eps}$. Further, we remark that in principle trajectories might arrive at $\zeta_1$ through a center manifold from inside $\mathcal{S}^{\txta}_0$. Since we focus on a neighborhood of $\zeta_2$ this has no effect on our analysis and we do not investigate a center manifold at $\zeta_1$ further. However, we note that Lemma \ref{zeta1} shows that there do not exist trajectories on the blow-up sphere emanating from $\zeta_1$ into the rescaling chart.
\end{remark}

\subsection{Transition from $\kappa_{-y}$ to $\kappa_{\eps}$}
\label{sec:entry_trans}

Recall that we are interested in the evolution of trajectories in attracting slow manifolds close to the hyperbolic umbilic point. From Lemma \ref{lemma:lin_along_critical_manifold}, these trajectories are contained in the center manifold $W^c(\mathcal{S}^{\txta}_{0,1})$ . In the singular limit, the only slow flow trajectory reaching the hyperbolic umbilic is $\sigma$, which we denote by $\sigma_1$ in blown-up coordinates $\kappa_{-y}$. From the previous section we know that $\sigma_1$ reaches the blow-up sphere at $\zeta_2$. Thus, $\zeta_2$ will be our ``base point'' to study transitions inside $W^c(\mathcal{S}^{\txta}_{0,1})$ into the rescaling chart $\kappa_{\eps}$. In other words, we aim to track the flow close to $\zeta_2$ in $W^c(\mathcal{S}_{0,1}^{\txta})$.

\begin{lemma}
	\label{lemma:CM_red}
	Let $\tilde{x}_1 = x_1 + \sqrt{B_0/C_0}$, such that $\zeta_2$ is translated to the origin.
	Then the flow in $W^c(\mathcal{S}_{0,1}^{\txta})\rvert_{\eps_1 > 0}$ close to $\zeta_2 = (\tilde{x}_1 = 0, a_1 = 0, r_1 = 0, \eps_1 = 0)$ is determined up to orbital equivalence by
	\begin{equation}
		\begin{split}
			\label{CM_red2}
			\tilde{x}_1' &= 2 \tilde{x}_1 + \Lambda\\
			a_1' &= a_1 + r_1 \frac{\tilde{g_a}}{G} + r_1^3 \eps_1 f_1\\
			r_1' &= - r_1\\
			\eps_1' &= 3 \eps_1,
		\end{split}
	\end{equation}
	where $\Lambda=\Lambda(\tilde x_1, a_1,r_1,\eps_1)$ satisfies $\Lambda(0) = \partial_{\tilde{x}_1}\Lambda(0) = 0$, $\tilde{g}_a$ is a shorthand recycled notation for $\tilde{g}_a \circ \Theta_{-y}$, where $\tilde{g}_a$ is defined in \eqref{sf:desing}, and $G$ is a sufficiently smooth function with $G(0)>0$.
\end{lemma}
\begin{proof}
	Consider the system \eqref{entry:rectification} with its linearization \eqref{entry_lin}. The critical manifold is given by $\{u = v = \eps_1 = 0\}$. Close to $\zeta_2$ the center manifold $W^c(\mathcal{S}_{0,1}^{\txta})$ is given as a graph $(u,v) = h(\tilde{x}_1,a_1,r_1,\eps_1) = (h_1,h_2)$. It follows that the map $h$ factors $\eps_1$. Further $h$ is tangent to the center eigenspace $E^c$ of the linearization \eqref{entry_lin} at $\zeta_2$. The eigenspace $E^c$ is given as the graph of the linear function $(u,v) = (\eps_1 \sqrt{B_0C_0}/2, \eps_1 C_0/2)$. This implies that $h_2$ is of the form $h_2 = \eps_1 \tilde{h}_2$, with $\tilde{h}_2(\zeta_2) = C_0/2 > 0$. We now plug $(u,v) = h$ into \eqref{entry:rectification}, where we first translate $\zeta_2$ to the origin, to obtain the flow in $W^c(\mathcal{S}_{0,1}^{\txta})$ close to $\zeta_2$. This leads to a vector field of the form
	\begin{equation}
		\begin{split}
			\label{vf:proof_red}
			\tilde{x}_1' &= (\tilde{x}_1 - \sqrt{B_0/C_0})(\eps_1 \tilde{h}_2 + r_1 \eps_1 f_y) + \eps_1 \tilde{h}_1 + r_1 \eps_1 f_x \\
			a_1' &= a_1 (\eps_1 \tilde{h}_2 + r_1 \eps_1 f_y) + r_1 \eps_1 g_a \\
			r_1' &= - r_1 (\eps_1 \tilde{h}_2 + r_1 \eps_1 f_y)\\
			\eps_1' &= 3 \eps_1 (\eps_1 \tilde{h}_2 + r_1 \eps_1 f_y).
		\end{split}
	\end{equation}
	For clarity we omit all function arguments.
	For $\eps_1 > 0$ the factor $\eps_1( \tilde{h}_2 + r_1 f_y)$ is positive sufficiently close to $\zeta_2$. Therefore we divide \eqref{vf:proof_red} by $\eps_1( \tilde{h}_2 + r_1 f_y)$. 
	\begin{equation}
		\begin{split}
			\label{vf:proof_red2}
			\tilde{x}_1' &= (\tilde{x}_1 - \sqrt{B_0/C_0}) + \frac{\tilde{h}_1}{\tilde{h}_2 + r_1 f_y} + r_1\frac{f_x}{\tilde{h}_2 + r_1 f_y} \\
			a_1' &= a_1 + r_1 \frac{g_a}{\tilde{h}_2 + r_1 f_y} \\
			r_1' &= - r_1 \\
			\eps_1' &= 3 \eps_1 .
		\end{split}
	\end{equation}
	Note that this uncovers a flow in the critical manifold $W^c(\mathcal{S}_{0,1}^{\txta})\rvert_{\eps_1 = 0}$ close to $\zeta_2$. In fact, we must recover, up to orbital equivalence, the desingularized slow flow \eqref{sf:-y} near $\zeta_2$ for $\eps_1 = 0$. We want to simplify \eqref{vf:proof_red2} further. First, note that by the shorthand $g_a$ we actually mean the function $(g_a \circ \Phi_{-y} \circ T^{-1})(\tilde{x}_1 - \sqrt{B_0/C_0},r_1,a_1,u,v,\eps_1)$, where further $u = \eps_1 \tilde{h}_1$ and $v = \eps_1 \tilde{h}_2$ are plugged in. Written out, this equals $g_a(r_1x_1,-r_1,r_1a_1,r_1^2(\eps_1 \tilde{h}_1 - x_1^2 + a_1),r_1^2(\eps_1 \tilde{h}_2 - 1 - a_1 x_1), r_1^3 \eps_1)$.
	By splitting off $g_a\rvert_{\eps = 0}$ in original coordinates by a Taylor argument, this can be rewritten as $(\tilde{g_a} \circ \Theta_{-y})(x_1,r_1,a_1) + r_1^3 \eps_1 f_1$, where $f_1$ is a sufficiently regular function (which is bounded in a neighborhood of the origin in original coordinates) and $\tilde{g}_a(x,y,a) = g_a(x,y,a,-x^2-ay,-y^2-ax,0)$. Let $G = \tilde{h}_2 + r_1 f_y$, where $f_y$ is a shorthand for $(f_y \circ \Phi_{-y} \circ T^{-1})(x_1,r_1,a_1,\eps_1 \tilde{h}_1,\eps_1 \tilde{h}_2,\eps_1)$. For the first equation, the constant term cancels out with $\tilde{h}_1(0)/\tilde{h}_2(0)$ and employing the invariance equation for the graph of the center manifold with the ansatz $\tilde{h}_1 = \sqrt{B_0C_0}/2 + \xi_1 \tilde{x}_1 + \dots$ and $\tilde{h}_2 = C_0/2 + \xi_2 \tilde{x}_1 + \dots$ implies the assertion.
\end{proof}

Note that the linearization of \eqref{CM_red2} at $\zeta_2$ has resonant eigenvalues $2,1,-1,3$. Let $W^{c}(\zeta_2)$ denote the intersection of (any instance of) $W^c(\mathcal{S}_{0,1}^{\txta})$ with the blow-up sphere $\{ r_1 = 0\}$ in $\kappa_{-y}$.

\begin{lemma}
	 The branch $\eps_1 > 0$ of the 3-dimensional attracting center manifold $W^{c}(\zeta_2)$ is unique.
\end{lemma}
\begin{proof}
	By Lemma \ref{lemma:CM_red}, the equilibrium $\zeta_2$ is a source in $W^{c}(\zeta_2)\rvert_{\eps_1 \ge 0}$. Moreover, $W^{c}(\zeta_2)$ is attracting. Then, the asymptotic rate fibration with base being the center manifold implies uniqueness of the branch for $\eps_1 > 0$.
\end{proof}

$W^{c}(\zeta_2)$ contains the ``singular continuations'' of $\sigma_1$, which start at $\zeta_2$ and emanate on the blow-up sphere into the rescaling chart $\kappa_{\eps}$. In general, we are interested in the transition from trajectories in attracting slow manifolds starting at $\Delta^{\en}$ close to the hyperbolic umbilic point. The section $\Delta^{\en}$ appears in $\kappa_{-y}$ as
\begin{equation}
	\begin{split}
		\Delta_1^{\en} = \{ r_1 = \nu \}.
	\end{split}
\end{equation}
The attracting slow manifolds are contained in $W^c(\mathcal{S}^{a}_{0,1})$ and their evolution close to $\zeta_2$ is determined by \eqref{CM_red2}. To enter the rescaling chart $\kappa_\eps$ we set up a section
\begin{equation}
	\begin{split}
		\Delta^{\en \to \eps}_1 = \{ \eps_1 = \delta \},
	\end{split}
\end{equation}
for small $\delta > 0$, and aim to study a transition of the form $\Delta_1^{\en} \cap W^c(\mathcal{S}^{a}_{0,1})\rvert_{\eps_1 > 0} \to \Delta^{\en \to \eps}_1$ near $\zeta_2$ via the flow of \eqref{CM_red2}. Note that $\zeta_2$ is a source on the blow-up sphere and the $r_1$-direction is decreasing. The eigenvalue of the linearization of \eqref{CM_red2} on the blow-up sphere transverse to the equator is 3, so the expansion is strongest in this direction. Since the 1-d stable manifold of $\zeta_2$ is given by $\sigma_1 \subset \{ \eps_1 = 0\}$, we expect that initial conditions close to $\sigma_1$ will lead to solution trajectories with least blow-up. Recall that we want to see the evolution of slow manifolds \emph{near} the hyperbolic umbilic, which is located in $a = 0$ in original coordinates. In fact, also from the analysis in the rescaling chart it will be clear that we need control on the variable $a_1$.

Consider a transition map $\Pi_1 \colon I_1 \subset \Delta_1^{\en} \cap W^c(\mathcal{S}^{a}_{0,1})\rvert_{\eps_1 > 0} \to \Delta^{\en \to \eps}_1$. Take an initial condition $(\tilde{x}_i,a_i,\nu,\eps_i)$. From \eqref{CM_red2} we have that $r_1(t) = \nu e^{-t}, \eps_1(t) = \eps_i e^{3t}$, where we use $t$ to indicate the time parametrization in the current chart. This gives the transition time $T = \log(\delta/\eps_i)/3$ for $\Pi_1$. Assume for a moment that the function $g_a\rvert_{\eps = 0}$ factors the variable $a$. Then we have

\begin{lemma}
	\label{lemma:CM_trans1}
	Assume that $g_a\rvert_{\eps = 0} = a \cdot \bar{g}_a$ for some sufficiently smooth $\bar{g}_a$. Then the transition map $\Pi_1 \colon I_1 \subset \Delta_1^{\en} \cap W^c(\mathcal{S}^{a}_{0,1})\rvert_{\eps_1 > 0} \to \Delta^{\en \to \eps}_1$ is of the form
	\begin{equation}
		\begin{split}
			\Pi_1 \colon \begin{bmatrix}
				x_i\\ a_i\\ \nu\\ \eps_i
			\end{bmatrix} \mapsto
			\begin{bmatrix}
				\Pi_{1 x_1}(x_i,a_i,\nu,\eps_i)\\[2ex] 
				\left(\dfrac{\delta}{\eps_i}\right)^{1/3} O(a_i) + O\left(\eps_i^{2/3} \log \eps_i\right)\\[2ex] 
				\nu \left(\dfrac{\eps_i}{\delta}\right)^{1/3}\\[2ex]
				\delta
			\end{bmatrix}
		\end{split}
	\end{equation}
\end{lemma}
\begin{proof}
	Let $a_1(t) = e^{t}(a_i + z_a(t))$, such that $z_a(0) = 0$. It follows from \eqref{CM_red2} that $z_a' = e^{-t} r_1^2 a_1 \bar{g}_a + r_1^3\eps_1 f_1$, for sufficiently smooth functions $\bar{g}_a,f_1$. This implies
	\begin{equation}
		\begin{split}
			z_a(T) = \int_0^T \nu^2 e^{-2t} a_i \bar{g}_a \d t + \int_0^T \nu^2 e^{-2t} z_a \bar{g}_a \d t + \int_0^T \eps_i \nu^3 f_1 \d t,
		\end{split}
	\end{equation}
	where the functions are evaluated along the corresponding integral curve. Recall that $\bar{g}_a,f_1$ are bounded close to the blow-up sphere and that $T = \log(\delta/\eps_i)/3$. The last integral is of order $O(\eps_i \log \eps_i )$. The first integral is of order $O(a_i)$. Application of the Grönwall inequality gives that the second integral is also $O(a_i)$. Hence $a_1(T)$ is of the form $a_1(T) = (\frac{\delta}{\eps_i})^{1/3}(a_i + O(a_i) + O(\eps_i \log \eps_i))$.
\end{proof}

In the general case that $g_a\rvert_{\eps = 0}$ does not factor the variable $a$, we have the rough estimate $a_1(T) = (\frac{\delta}{\eps_i})^{1/3}(a_i + O(\nu))$ in context of the above Lemma.

\begin{remark}
The assumption that $g_a\rvert_{\eps = 0} = a \cdot \bar{g}_a$ allows to control $a_1(T)$ solely in terms of $a_i$ and $\eps_i$. In particular, Lemma \ref{lemma:CM_trans1} implies that the domain $I_1$ of $\Pi_1$ must be of width $O(\eps_i^{1/3})$ in $a_1$-direction, to have $a_1(T)$ small. That is, for small, fixed $\eps_i$, the domain $I_1\rvert_{\eps_i}$ is a thin (but bounded) strip of width $O(\eps_i^{1/3})$. In the singular limit, it contains the trajectory $\sigma_1$, which lies in the invariant plane $\{ a_1 = 0\}$ in this case. In the general setting, i.e.\ without the assumption $g_a\rvert_{\eps = 0} = a \cdot \bar{g}_a$, the domain $I_1$ needs to be chosen sufficiently small in order to guarantee that $a_1(T)$ is small. We do not have an estimate of the size of the domain with respect to the initial conditions in this case.  We shall choose $I_{1}$ sufficiently small in section \ref{sec:main_proof} by a time reversion argument. In the singular limit, $I_1$ coincides with the single point $p := \sigma_1 \cap \Delta^{\en}_1$. 
\end{remark}

\subsection{Analysis in the rescaling chart $\kappa_{\eps}$}
\label{sec:rescaling}

Blowing up $X$ via $\Phi_{\eps}$ and rescaling time accordingly leads to
\begin{equation}
	\begin{split}
		\label{blowup_eps}
		x_2' &= x_2^2 + a_2 y_2 + b_2 + O(r_2)\\
		y_2' &= y_2^2 + a_2 x_2 + c_2 + O(r_2)\\
		a_2' &= A_0 r_2 + O(r_2^2)\\
		b_2' &= B_0 + O(r_2)\\
		c_2' &= C_0 + O(r_2)\\
		r_2' &= 0.
	\end{split}
\end{equation}

In $\{r_2 = 0\}$ the vector field \eqref{blowup_eps} reduces to a coupled system of two Riccati-type equations
\begin{equation}
	\begin{split}
		\label{riccati_on_sphere}
		x_2' &= x_2^2 + a_2 y_2 + b_2 \\
		y_2' &= y_2^2 + a_2 x_2 + c_2 \\
		a_2' &= 0\\
		b_2' &= B_0 \\
		c_2' &= C_0,
	\end{split}
\end{equation}
whose phase space is foliated into invariant sets $\{a_2 = \text{const}\}$. Note that $a_2 = a_1 \eps_1^{-1/3}$, which gives the invariant sets on the blow-up sphere in $\kappa_{-y}$, also visible in \eqref{blowup_-y}.
Further note that the sets $\{C_0 b_2 - B_0 c_2 = \text{const}\}$ are invariant. Recall the attracting center manifold $W^c(\zeta_2)$ from the previous section. We aim to follow trajectories in $W^{c}(\zeta_2)\rvert_{\eps_1 > 0}$, which emanate from $\zeta_2$ on the blow-up sphere into the rescaling chart. We expect that these orbits will transition onto the fast regime, that is, connect to the appropriate equilibria ``on the other side'' of the blow-up sphere. In $\{a_2 = 0\}$ the coupling of \eqref{riccati_on_sphere} is absent. For $a_2$ close to 0 the flow of \eqref{riccati_on_sphere} is a regular perturbation of the flow in $\{a_2 = 0\}$. The full flow \eqref{blowup_eps} is in turn a regular perturbation of \eqref{riccati_on_sphere}. Therefore, we begin with the analysis of the invariant set $\{a_2 = 0\}$ of \eqref{riccati_on_sphere}. Note that this essentially leads to a pair of decoupled Riccati equations, as they appear in the rescaling chart for the non-degenerate fold \cite{krupa2001extending}. The task is now to identify the trajectories being backward asymptotic to $\zeta_2$ and to determine their forward asymptotics.

For $a_2 = 0$ the vector field \eqref{riccati_on_sphere} is of product structure $(Z_1,Z_2)$, where both $Z_i$ are vector fields on $\R^2$. In fact, each $Z_i$ has a unique ``dividing solution'', denote it by $\gamma_i$ for the moment, which is an integral curve dividing the phase plane in two classes of solutions with different asymptotic behavior, itself having distinct behavior \cite{krupa2001extending,mishchenko1980}. Only combinations of these $\gamma_i$  lead to trajectories being backward asymptotic to $\zeta_2$, as shall be apparent from the proof of Proposition \ref{res:unique_gamma} below.

\begin{remark}\label{rem:qs}
	To state the following result in a convenient way, we anticipate the main objects visible in the exit chart $\kappa_{\ex}$. These are three equilibria denoted by $q_4,q_5,q_6$, which are located on the equator of the blow-up sphere. As will turn out later, they arise from the interaction of the fast subsystem with the hyperbolic umbilic singularity and will serve as transition points onto the fast regime. Refer to section \ref{s:exit_chart}, in particular figure \ref{fast_blowup}.
\end{remark}

\begin{prop}
	\label{res:unique_gamma}
	Let $B_0, C_0 > 0$. In $\{a_2 = 0\}$ the vector field \eqref{riccati_on_sphere} has a 1-parameter family of integral curves $\gamma_s$, $s \in \R$, given by 
	\begin{equation}
	\begin{split}
		\label{the_family_gamma}
		\gamma_s(t) = (\gamma_{x_2}(t),\gamma_{y_2}(t+s), a_2 = 0, B_0 t, C_0 (t + s)),
	\end{split}
\end{equation}
	which are the only trajectories in $\{a_2 = 0\}$ backward asymptotic to $\zeta_2$. In \eqref{the_family_gamma}, $\gamma_{x_2}$ is given by \eqref{d_infty_parabola} and $\gamma_{y_2}$ by \eqref{gamma_y2}. The family $\gamma_s$ determines an invariant 2-manifold $\Gamma$ in $\{r_2 = a_2 = 0\}$.
	The trajectories $\gamma_s$ have the following forward asymptotic behavior:
	\begin{itemize}
		\item For $s > s_0$, $\gamma_s$ is forward asymptotic to $q_4$,
		\item $\gamma_{s_0}$ is forward asymptotic to $q_5$,
		\item for $s < s_0$, $\gamma_s$ is forward asymptotic to $q_6$,
	\end{itemize}
	where $q_4,\, q_5,$ and $q_6$ are equilibria visible in the chart $\kappa_\ex$ (see Remark \ref{rem:qs}), $s_0 = (B_0^{-1/3} - C_0^{-1/3})z_0$, and $z_0<0$ is the first zero of the Airy function $\Ai$ of the first kind, see e.g.\ \cite[9.9(i)]{NIST:DLMF}.
\end{prop}
\begin{proof}
	In $\{a_2 = 0\}$, system \eqref{riccati_on_sphere} reduces to two decoupled Riccati equations, as they appear in the rescaling chart of the regular fold \cite{krupa2001extending}.
	First, consider
	\begin{equation}
		\begin{split}
			\label{one_riccati}
			x_2' &= x_2^2 + b_2\\
			b_2' &= B_0.
		\end{split}
	\end{equation}
	In \cite[Chapter II, 9.]{mishchenko1980} the asymptotic behavior of a Riccati equation is studied by reduction to a Bessel equation. Here, we choose to reduce to an Airy equation.
	Letting $x_2 = - \frac{u'}{u}$ and $b_2 = B_0 t$ transforms \eqref{one_riccati} into
	\begin{equation}
		\begin{split}
			u'' &= - B_0 t u\\
			t' &= 1,
		\end{split}
	\end{equation}
	which has the two linearly independent solutions $\Ai(-B_0^{1/3} t), \Bi(-B_0^{1/3} t)$. Hence, we obtain the general solution for \eqref{one_riccati} given by
	\begin{equation}
		\label{x2omegas}
		x_2(t) = \frac{B_0^{1/3} \omega_1 \Ai'(-B_0^{1/3} t) + B_0^{1/3} \omega_2 \Bi'(-B_0^{1/3} t)}{ \omega_1 \Ai(-B_0^{1/3} t) +  \omega_2 \Bi(-B_0^{1/3} t)},
	\end{equation}
	where $\omega_1$ and $\omega_2$ are constants that do not vanish simultaneously. It is convenient to re-write the solutions \eqref{x2omegas} in terms of a single parameter $d \in (-\infty,\infty]$ as
	\begin{equation}
		\begin{split}
			\label{riccati_solution}
			x_2(t)&= \begin{cases}
				{\tilde{B}} \frac{d \Ai'(-\tilde{B} t) + \Bi'(-\tilde{B} t)}{d \Ai(-\tilde{B} t) + \Bi(- \tilde{B} t)}, & d<\infty,\\
				{\tilde{B}} \frac{\Ai'(-\tilde{B}t)}{\Ai(-\tilde{B}t)}, & d = +\infty,
			\end{cases}
		\end{split}
	\end{equation}
	where $d=\frac{\omega_1}{\omega_2}$, and $\tilde{B} = B_0^{1/3}$.
	Each choice of $d$ defines a curve with asymptotes at certain values of $t$, due to the zeros of the denominator. The solutions \eqref{riccati_solution} are sketched in figure \ref{ricc_plot}. The case $d = +\infty$ corresponds to the so called ``dividing solution''.
\begin{figure}[h]
	\centering
	\begin{overpic}[width=0.7\textwidth]{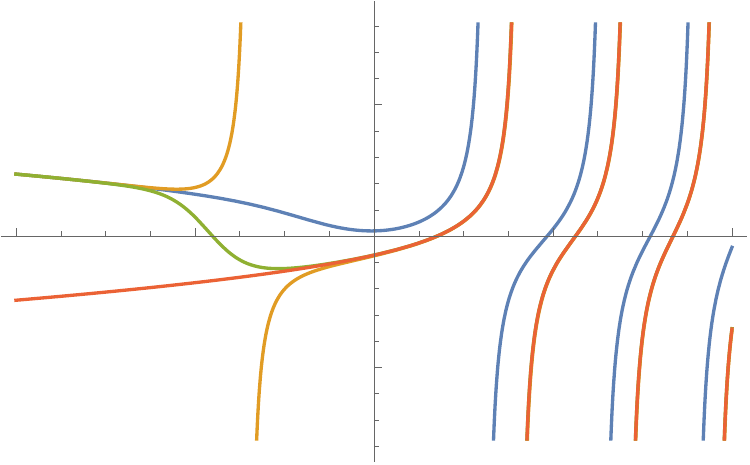}
		\put(98,33){\textcolor{gray}{$t$}}
		\put(45,58){\textcolor{gray}{$x_2$}}
		\put(35,3){\rotatebox{87}{\tiny $d = -100$}}
		\put(29,48){\rotatebox{87}{\tiny $d = -100$}}
		\put(61,53){\rotatebox{87}{\tiny $d = 1$}}
		\put(2,20){\rotatebox{5}{\tiny $d = +\infty$}}
		\put(26,34){\rotatebox{-42}{\tiny $d = 500$}}
	\end{overpic}
	\caption{Qualitative plot of the solutions \eqref{riccati_solution} to the Riccati equation \eqref{one_riccati} with $t = b_2 / B_0$ and $d = -100,1,500,+\infty$. Any solution is defined for almost every $t \in \R$, countably many vertical asymptotes arise from the properties of the Airy functions. Note the appearance of the parabola $B_0 t = - x_2^2$, which opens to the left and which all solutions approach as $t \to -\infty$.}
	\label{ricc_plot}
\end{figure}
	By Lemma \ref{coordinate_changes}, we write the equations of the critical manifold from $\kappa_{-y}$ (and similarly in $\kappa_{-x}$) in $\kappa_\eps$-coordinates, which yields
	\begin{equation}
		\begin{split}
			\label{critical_mfld_backward_asymptotic_in_-x_-y}
			b_2 &= -a_2y_2 - x_2^2\\
			c_2 &= -a_2x_2 - y_2^2.
		\end{split}
	\end{equation}
	Again by Lemma \ref{coordinate_changes}, solutions of \eqref{riccati_on_sphere} are backward asymptotic to $\mathcal{S}_{0,1}\eval_{r_1 = 0}$ in $\kappa_{-y}$ if and only if they asymptotically satisfy \eqref{critical_mfld_backward_asymptotic_in_-x_-y} for $y_2 \to -\infty$. 
	We seek for solutions in $\{a_2 = 0\}$ being backward asymptotic to $\zeta_2$. Hence, we seek solutions of \eqref{riccati_on_sphere} satisfying $b_2 =- x_2^2$ as $x_2 \to -\infty$ and $c_2 =- y_2^2$ as $y_2 \to -\infty$ asymptotically in backward time. The requirement $b_2 =- x_2^2$ as $x_2 \to -\infty$ is only satisfied by the ``dividing solution'', i.e. \eqref{riccati_solution} with $d=+\infty$:\\
	The Airy functions $\Ai, \Bi$ do not have zeros on the positive half-line and have asymptotic expansions \cite[9.7(ii)]{NIST:DLMF}. Employing such asymptotic expansions, it follows that $x_2 \to -\infty$ in \eqref{riccati_solution} as $t \to -\infty$ only in the case $d = +\infty$: for $t \to -\infty$ one has
\begin{equation}
	\begin{split}
		\label{d_infty_parabola}
		\gamma_{x_2}(t) := {\tilde{B}} \frac{\Ai'(-\tilde{B}t)}{\Ai(-\tilde{B}t)} \sim - \sqrt{-B_0 t},
	\end{split}
\end{equation}
so $\gamma_{x_2}$ is backward asymptotic to the branch $x_2 < 0$ of the parabola $b_2 = B_0 t = -x_2^2$, see figure \ref{ricc_plot}. 
	Note that solutions obtained from \eqref{riccati_solution}, which are defined on finite time intervals $(t_1,t_2)$, may also show $x_2 \to -\infty$ for $t \downarrow t_1$. However, in this case, $b_2 \not\to - \infty$, so these solutions cannot satisfy $b_2 =- x_2^2$ as $x_2 \to -\infty$.\\
	An identical analysis as above applies to the second half of \eqref{riccati_on_sphere} in $\{a_2 = 0\}$, which is the system
	\begin{equation}
		\begin{split}
			\label{2nd_riccati}
			y_2' &= y_2^2 + c_2\\
			c_2' &= C_0.
		\end{split}
	\end{equation}
	The ``dividing solution'' for \eqref{2nd_riccati} is given by
\begin{equation}
	\begin{split}
		\label{gamma_y2}
		\gamma_{y_2}(t) := {\tilde{C}} \frac{\Ai'(-\tilde{C}t)}{\Ai(-\tilde{C}t)},
	\end{split}
\end{equation}
where $\tilde{C} = C_0^{1/3}$. By the product structure of \eqref{riccati_on_sphere} in $\{a_2 = 0\}$ we obtain all integral curves built from the ``dividing solutions'' $\gamma_i$ essentially by putting $(\gamma_1(t),\gamma_2(t+s))$ for a free parameter $s \in \R$.
Thus, we obtain a 1-parameter family of solutions in $\{a_2 = 0\}$, determined by
\begin{equation}
	\begin{split}
		\gamma_s(t) = (\gamma_{x_2}(t),\gamma_{y_2}(t+s), a_2 = 0, B_0 t, C_0 (t + s)),
	\end{split}
\end{equation}
where $s \in \R$. The above argument then implies that the trajectories $\gamma_s$ are precisely the integral curves, which are backward asymptotic to the critical manifold. Then, employing asymptotic expansions of the Airy functions \cite[9.7(ii)]{NIST:DLMF} and Lemma \ref{coordinate_changes}, one checks that
\begin{equation}
	\begin{split}
		\frac{\gamma_{x_2}(t)}{- \gamma_{y_2}(t+s)} \to - \sqrt{B_0/C_0},
	\end{split}
\end{equation}
as $t \to -\infty$. This implies that the trajectories $\gamma_s$ are precisely the solutions of \eqref{riccati_on_sphere} in $\{a_2 = 0\}$, which are backward asymptotic to $\zeta_2$.

What remains to be checked is the forward asymptotics of the trajectories $\gamma_s$. We will employ $\kappa_{\eps \to \ex}$ from Lemma \ref{coordinate_changes} and the fact that the Airy function $\Ai$ does have a largest zero \cite[9.9(i)]{NIST:DLMF}, which we denote by $z_0$. Let $M,N$ be the first zeros of the denominators of $\gamma_{x_2}$ and $\gamma_{y_2}$, in the sense that $\Ai(-B_0^{1/3}t)$ is defined on $(-\infty,M)$ and $\Ai(-C_0^{1/3}t)$ is defined on $(-\infty,N)$, both domains maximal. Thus $z_0 = -B_0^{1/3}M = -C_0^{1/3}N$. Then $\gamma_{x_2}$ is defined on $(-\infty,M)$ and $\gamma_{y_2}(\cdot + s)$ is defined on $(-\infty,N_s)$, where $N_s = N - s$. Since $\Ai$ is positive before its first zero $z_0$ and $\Ai'(z_0) > 0$, it follows that $\gamma_{x_2}(t) \to +\infty$ for $t \uparrow M$ and $\gamma_{y_2}(t+s) \to +\infty$ for $t \uparrow N_s$. Each $\gamma_s(t)$ is defined for all $t < \min \{M,N_s\}$.
The ratio of $M$ and $N_s$ determines the forward asymptotics of $\gamma_s$. 
The relevant component of the coordinate change $\kappa_{\eps \to \ex}$ is
\begin{equation}
	\begin{split}
		\label{x5_change}
		x_5 = \frac{x_2 - y_2}{x_2 + y_2},
	\end{split}
\end{equation}
since by Lemma \ref{coordinate_changes} all other components converge to zero along every $\gamma_s$ in forward time. By \eqref{the_qs} we have that $q_4$ corresponds to $x_5 = -1$, $q_5$ to $x_5 = 0$ and $q_6$ to $x_5 = 1$. Consider the three cases $M > N_s$, $M < N_s$ and $M = N_s$:\\

If $M > N_s$ then $\gamma_{y_2}(t+s)$ blows up as $t \uparrow N_s$, so
\begin{equation}
	\begin{split}
		x_5 = \left(\frac{\gamma_{x_2}(t)}{\gamma_{y_2}(t+s)} - 1\right)\left(\frac{\gamma_{x_2}(t)}{\gamma_{y_2}(t+s)} + 1\right)^{-1} \to -1 \quad\text{for } t \uparrow N_s < M.
	\end{split}
\end{equation}
If $M < N_s$ then $\gamma_{x_2}(t)$ blows up as $t \uparrow M$, so
\begin{equation}
	\begin{split}
		x_5 = \left( 1 - \frac{\gamma_{y_2}(t+s)}{\gamma_{x_2}(t)}\right)\left(1 + \frac{\gamma_{y_2}(t+s)}{\gamma_{x_2}(t)}\right)^{-1} \to 1 \quad\text{for } t \uparrow M < N_s.
	\end{split}
\end{equation}
If $M = N_s \iff s = N - M$, both $\gamma_{x_2}(t) \to +\infty$ and $\gamma_{y_2}(t + s) \to +\infty$ as $t \uparrow M$. In the case of $B_0 = C_0$, and therefore $s = 0$, we directly see from \eqref{x5_change} that $\gamma_0$ is forward asymptotic to $q_5$. To treat the general case $s = N - M$, consider
\begin{equation}
	\begin{split}
		x_5 = \left(\frac{\gamma_{x_2}(t)}{\gamma_{y_2}(t+N-M)} - 1\right)\left(\frac{\gamma_{x_2}(t)}{\gamma_{y_2}(t+N-M)} + 1\right)^{-1}.
	\end{split}
\end{equation}
To find the limit of $\frac{\gamma_{x_2}(t)}{\gamma_{y_2}(t+N-M)}$ as $t \uparrow M$, note that a first order Taylor expansion gives $\Ai(-\tilde{B}t) \sim -\tilde{B}\Ai'(z_0)$ as $t \to M$. Similarly $\Ai(-\tilde{C}(t+N-M)) \sim -\tilde{C}\Ai'(z_0)$ as $t \to M$. This implies that $\frac{\gamma_{x_2}(t)}{\gamma_{y_2}(t+N-M)} \to 1$. Hence, $\gamma_{N-M}$ is forward asymptotic to $q_5$, and we let $s_0 = N - M$.
\end{proof}

\begin{remark}
	\label{remark_asymptotics}
	Solutions $(x_2,y_2,b_2,c_2)$ in $\{r_2 = 0\}$ need to asymptotically satisfy  $b_2 = -a_2y_2 - x_2^2, c_2 = -a_2x_2 - y_2^2$ in order to be asymptotic to the critical manifold. Further they need to show $y_2 \to -\infty$ in order to arrive at the equator in $\kappa_{-y}$, $x_2 \to -\infty$ to arrive at the equator in $\kappa_{-x}$, et cetera. Suppose that such solutions are asymptotic to slow flow equilibria $\zeta_{2,3,4,5}$ or fast subsystem equilibria $q_{1,2,3,4,5,6}$ (in fact, there do not exist solutions of \eqref{riccati_on_sphere} being asymptotic to $\zeta_{1,6}$, see Remark \ref{rem:resc:kappa_a}). It is then possible to determine the forward/backward limits of all solutions in $\{r_2 = 0, a_2 = 0\}$ by proceeding similarly to the proof of Proposition \ref{res:unique_gamma}.
\end{remark}

\begin{remark}
	\label{rem:resc:kappa_a}
	At this point we remark that the charts $\kappa_{\pm a}$ can be used to supply an analysis of the flow of \eqref{riccati_on_sphere} for sufficiently large $\abs{a_2}$. More precisely, the fast-slow systems \eqref{fs:-a} and \eqref{fs:a} in section \ref{sec:further} are instances of \eqref{riccati_on_sphere} for large $\abs{a_2}$. In fact, in the singular limit analysis of \eqref{fs:-a} and \eqref{fs:a}, Remark \ref{kappa_a_asymptotics} can be employed to show that the relevant candidate trajectories coming from $\zeta_2$ are jumping and arriving to $q_4$ or $q_6$, depending on the location of the jump. Moreover, in $\kappa_{\pm a}$ one sees that there do not exist trajectories of \eqref{riccati_on_sphere} being forward/backward asymptotic to $\zeta_1$, see e.g.\ Lemma \ref{zeta1}. An analogous statement can be obtained for $\zeta_6$.
\end{remark}

The family $\gamma_s$ from Proposition \ref{res:unique_gamma} gives all trajectories in $W^c(\zeta_2)\rvert_{a_2 = 0}$, which enter the rescaling chart $\kappa_\eps$ and are backward asymptotic to $\zeta_2$. Recall that $W^c(\zeta_2)$ is contained in the blow-up sphere. The section $\Delta^{-y \to \eps}_1$, which appears in $\kappa_{\eps}$ as $\Delta^{-y \to \eps}_2 = \{ y_2 = - \delta^{-1/3} \}$, has transverse intersection with $\Gamma$, where $\Gamma$ is the invariant manifold determined by the family $\gamma_s$.

For $a_2$ sufficiently close to zero, the flow of \eqref{riccati_on_sphere} is a regular perturbation of the flow in $\{ a_2 = 0\}$. Perturbations of $\gamma_s$ give a 2-parameter family of integral curves following $\Gamma$ closely through the rescaling chart. Furthermore, the $a_2$-perturbations of any $\gamma_s$ are backward asymptotic to $\zeta_2$: Since these regular perturbations are $O(a_2)$-uniformly close to the unperturbed solutions on any finite time interval, the $a_2$-perturbations must also be contained in $W^c(\zeta_2)$, because in backward time they cannot get exponentially repelled from $W^c(\zeta_2)$. Moreover, by Lemma \eqref{lemma:CM_red}, the equilibrium $\zeta_2$ appears in $W^c(\zeta_2)$ as source, which implies the backward asymptotic behavior of the perturbations.
Let $\hat{\Gamma}$ denote the invariant 3-manifold obtained by regular $a_2$-perturbations of trajectories in $\Gamma$ and note that $\hat{\Gamma} = \{ r_2 = 0, \abs{a_2} < L \}$ for some small $L>0$.
For trajectories in the submanifold $\Gamma \subset \hat{\Gamma}$ we have established the forward asymptotic behavior. The forward asymptotic behavior of trajectories in $\hat{\Gamma}$ will be established in the exit chart $\kappa_{\ex}$ after Lemma \ref{lemma:transverse}.

For small $r_2 > 0$ the full system \eqref{blowup_eps} is again a regular perturbation of the flow in $\{r_2 = 0\}$, i.e.\ \eqref{riccati_on_sphere}. In particular, trajectories in $\hat{\Gamma}$ get regularly perturbed for small $r_2 > 0$. These are the trajectories in attracting slow manifolds close to the hyperbolic umbilic. We will track these further in the exit chart $\kappa_\ex$. Recall from Lemma \ref{coordinate_changes} that we can change coordinates to $\kappa_{\ex}$, as soon as $x_2 + y_2 > 0$.

We conclude this section with a simple statement on the structure of a regular transition map $\Pi_2$ from a suitable subset of $\Delta_2^{-y \to \eps}$ to (one choice of) suitable sections close to the equilibria $q_{4,5,6}$, which will be specified later.

\begin{lemma}
	\label{resc:trans}
	 Let $V \subset \Delta_2^{-y \to \eps}$ be an open set corresponding to initial conditions obtained from $r_2$-perturbations of $\hat{\Gamma}$ and let $\Omega$ be any hypersurface, which intersects all trajectories with initial condition in $V$. Then the transition $\Pi_2 \colon V \to \Omega$ is of the form
\begin{equation}
	\begin{split}
		\label{map:rescaling_transition}
		\Pi_2 \colon \begin{bmatrix}
			x_j\\
			-\delta^{-1/3}\\
			a_j\\
			b_j\\
			c_j\\
			r_j
		\end{bmatrix}
		\mapsto
		\begin{bmatrix}
			x_e\\
			y_e\\
			0\\
			b_e\\
			c_e\\
			0
		\end{bmatrix} + 
		\begin{bmatrix}
			O(a_j)\\
			O(a_j)\\
			a_j\\
			O(a_j)\\
			O(a_j)\\
			0
		\end{bmatrix} +
		\begin{bmatrix}
			O(r_j)\\
			O(r_j)\\
			O(r_j)\\
			O(r_j)\\
			O(r_j)\\
			r_j
		\end{bmatrix},
	\end{split}
\end{equation}
where $a_j$ and $r_j$ are sufficiently small and the map $(x_j,-\delta^{-1/3},b_j,c_j) \mapsto (x_e,y_e,b_e,c_e)$ is determined by the flow in $\Gamma$, its integral curves are given by \eqref{the_family_gamma}.
\end{lemma}
\begin{proof}
	By the above considerations and regular perturbation theory.
\end{proof}

\begin{remark}
	$C_0 b_2 - B_0 c_2$ is a constant of motion for \eqref{riccati_on_sphere}. In fact, for $\gamma_s \in \Gamma$ we have $C_0 b_2 - B_0 c_2 = - B_0 C_0 s$. That is, the $b_2$- and $c_2$-component of the initial condition of $\gamma_s$ determine $s$ and therefore its forward limit. Hence $b_j$ and $c_j$ in \eqref{map:rescaling_transition} determine the family member $\gamma_s$, from which the transition \eqref{map:rescaling_transition} is a regular perturbation.
\end{remark}

\subsection{Analysis in the exit chart $\kappa_{\ex}$}
\label{s:exit_chart}

Let us briefly explain how the chart $\kappa_{\ex}$ and its blow-up $\Phi_{\ex}$ are obtained. Recall that $\Phi_{\ex}$ is defined by
\begin{equation}
	\begin{split}
		\label{phi_ex}
		\Phi_{\ex}\colon x = r_5(1 + x_5),\, y = r_5(1-x_5), \,a = r_5 a_5,\, b = r_5^2 b_5, \, c = r_5^2 c_5,\, \eps = r_5^3 \eps_5.
	\end{split}
\end{equation}
Let $\tilde{x} = \frac{x-y}{2}$ and $\tilde{y} = \frac{x+y}{2}$. Up to scaling the coordinates, this is a $45^\circ$ rotation. In these new coordinates $\tilde{x},\tilde{y}$, the map $\Phi_{\ex}$ is a usual weighted blow-up in direction of $\tilde{y}$. Further, note that \eqref{phi_ex} implies $\frac{x-y}{2} = r_5x_5$ and $\frac{x+y}{2} = r_5$.
We can also transform the original vector field $X$ \eqref{hyp_umb_exp} according to $\tilde{x} = \frac{x-y}{2}$ and $\tilde{y} = \frac{x+y}{2}$.
The first two components of $X$ then transform to
\begin{equation}
	\begin{split}
		\label{tilde}
		\tilde{x}' &= 2\tilde{x}\tilde{y} - a\tilde{x} + \frac{b-c}{2} + O(\eps)\\
		\tilde{y}' &= \tilde{x}^2 + \tilde{y}^2 + a\tilde{y} + \frac{b+c}{2} + O(\eps).
	\end{split}
\end{equation}
For $\eps = 0$, this is a gradient vector field in the fast variables $(\tilde{x},\tilde{y})$ arising from the potential
\begin{equation}
	\label{alternative_potential}
	\begin{split}
		\tilde{V} = \frac{1}{3}\tilde{y}^3 + \tilde{y}\tilde{x}^2 + \tilde{a}(\tilde{y}^2 - \tilde{x}^2) + \tilde{b}\tilde{x} + \tilde{c}\tilde{y},
	\end{split}
\end{equation}
with parameters $\tilde{a} = \frac{a}{2}, \tilde{b} = \frac{b-c}{2}$ and $\tilde{c} = \frac{b+c}{2}$. In fact, the potential $\tilde{V}$ is another universal unfolding of the hyperbolic umbilic catastrophe. This is shown in appendix \ref{app:universal_unfolding}.
With this universal unfolding $\tilde{V}$, the double cone of singularities is given by $\tilde{y}^2 = \tilde{a}^2 + \tilde{x}^2$, i.e.\ the rotational axis of the double cone is the $\tilde{y}$-axis.

\begin{remark}
	For the prior analysis we have chosen the universal unfolding $V$ leading to \eqref{hyp_umb}, since the expressions of the desingularized slow flow and the vector field in the rescaling chart are more convenient to analyze.
\end{remark}

Blowing up $X$ with $\Phi_{\ex}$ and dividing out a common factor $r_5$ leads to
\begin{equation}
	\begin{split}
		\label{vf:exit_chart}
		x_5' &= 2x_5 - a_5x_5 + \frac{b_5 - c_5}{2} - x_5 F + r_5\eps_5\frac{f_x - f_y}{2}\\
		r_5' &= r_5 F\\
		a_5' &= -a_5 F + r_5\eps_5 g_a\\
		b_5' &= - 2b_5 F + \eps_5 g_b\\
		c_5' &= - 2c_5 F + \eps_5 g_c\\
		\eps_5' &= - 3\eps_5 F,
	\end{split}
\end{equation}
where $F = F(x_5,r_5,a_5,b_5,c_5,\eps_5) = 1 + x_5^2 + a_5 + \frac{b_5 + c_5}{2} + r_5\eps_5 \frac{f_x + f_y}{2}$ and the shorthand $f_x$ for $f_x \circ \Phi_{\ex}$ is used, similarly for $f_y,g_a,g_b,g_c$.

The subspaces $\{r_5 = 0\}$ and $\{\eps_5 = 0\}$ are invariant. Furthermore, $\{r_5 = 0, a_5 = 0\}$, $\{\eps_5 = 0, a_5 = 0\}$, $\{\eps_5 = 0, b_5 = 0\}$ and $\{ \eps_5 = 0, c_5 = 0\}$ (and any of their intersections) are also invariant.

Consider \eqref{vf:exit_chart} restricted to the invariant subspace $\{r_5 = 0, \eps_5 = 0\}$. Similarly to the entry charts, this leads to a 3-manifold of equilibria determined by
\begin{equation}
	\begin{split}
		1 + x_5^2 + a_5 + \frac{b_5 + c_5}{2} &= 0\\
		2x_5 - a_5x_5 + \frac{b_5 - c_5}{2} &= 0,
	\end{split}
\end{equation}
which is the critical manifold appearing in the exit chart. Besides that, there are three isolated equilibria 
\begin{equation}
	\begin{split}
		\label{the_qs}
		q_4 &= (x_5 = -1, a_5 = 0, b_5 = 0, c_5 = 0),\\
		q_5 &= (x_5 = 0, a_5 = 0, b_5 = 0, c_5 = 0),\\
		q_6 &= (x_5 = 1, a_5 = 0, b_5 = 0, c_5 = 0),
	\end{split}
\end{equation}
which are precisely the equilibria on the blow-up sphere arising from the fast subsystem.

We now further motivate the role of the equilibria $q_{4,5,6}$. Recall the fast subsystem obtained from $X$ \eqref{hyp_umb_exp} by letting $\eps = 0$. Trajectories in the fast subsystem can only approach the origin in forward or backward time if $a = b = c = 0$. Hence, in the singular limit, the fast dynamics interacting with the hyperbolic umbilic singularity are given by the system
\begin{equation}
	\begin{split}
		\label{fast_restricted}
		x' &= x^2\\
		y' &= y^2,
	\end{split}
\end{equation}
where the hyperbolic umbilic point appears as nilpotent equilibrium. Also refer to figure \ref{fig:main_thm_sing_limit}, which contains the phase portrait of \eqref{fast_restricted}. The blow-up of \eqref{fast_restricted} is contained in the blow up of the vector field $X$ \eqref{hyp_umb_exp} by restricting to the subspace $\{\bar{a} = \bar{b} = \bar{c} = \bar{\eps} = 0\}$. In $\kappa_{\ex}$ this corresponds to the invariant subspace $\{\eps_5 = 0, a_5 = 0, b_5 = 0, c_5 = 0\}$, which we will analyze shortly in the following. 

The blow-up of \eqref{fast_restricted} is visualized in figure \ref{fast_blowup}. From Proposition \ref{res:unique_gamma} we know that the family of integral curves $\gamma_{s}$ in $\{r_2 = a_2 = 0\} \subset \kappa_{\eps}$ connects in forward time to $q_{4,5,6}$. More precisely, only $\gamma_{s_0}$ connects to $q_5$, whereas the remaining members of the family $\gamma_s$ connect to either $q_4$ or $q_6$. In the singular limit, from $q_{4,5,6}$, such trajectories continue onto the fast regime and away from the blow-up sphere.
This motivates that trajectories of $X$, which are contained in attracting slow manifolds and come close to the hyperbolic umbilic, will transition onto the fast regime near the equilibria $q_{4,5,6}$. The saddle $q_5$ leads to trajectories fanning out in-between the positive $x$- and $y$-axis. The main goal of this section is to investigate this transition away from the blow-up sphere.

\begin{figure}[h]
	\centering
	\begin{overpic}[width = 0.5\textwidth]{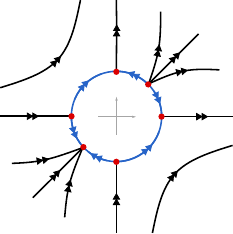}
		\put(34,49){$q_3$}
		\put(38,39){$q_2$}
		\put(48,35){$q_1$}
		\put(62,49){$q_6$}
		\put(59,59){$q_5$}
		\put(48,63){$q_4$}
		\put(55,45){\textcolor{gray}{$\bar{x}$}}
		\put(45,55){\textcolor{gray}{$\bar{y}$}}
	\end{overpic}
	\caption{Blow-up of \eqref{fast_restricted}, i.e. the fast subsystem interacting with the hyperbolic umbilic point. The blow-up sphere $\mathbb S^5$ restricted to $\bar{a} = \bar{b} = \bar{c} = \bar{\eps} = 0$ appears as $\mathbb{S}^1$ (blue), with six equilibria $q_{1,2,3,4,5,6}$ (red), where $q_{4,5,6}$ are visible in $\kappa_{\ex}$, $q_{1,2}$ appear in $\kappa_{-y}$ and $q_{2,3}$ in $\kappa_{-x}$. Note that the unstable manifolds at $q_{4,5,6}$ also appear in this figure. As we show in this section, in the singular limit the trajectories coming from the rescaling chart $\kappa_\eps$ escape towards the fast regime along the aforementioned unstable manifolds.}
	\label{fast_blowup}
\end{figure}

\begin{remark}
	Due to the rotated choice of $\Phi_{\ex}$ \eqref{phi_ex}, the three equilibria $q_{4,5,6}$ and all ``fast escape directions'' are visible in the single chart $\kappa_{\ex}$.
\end{remark}

We briefly look at the linearization of \eqref{vf:exit_chart} at $q_{4,5,6}$.
This directly leads to the following result.
\begin{lemma}
	\label{exit_lin_at_qs}
	At $q_4,q_5,q_6$, the resonant linearizations of \eqref{vf:exit_chart} are of the following upper-triangular form:
	\begin{align*}
		\text{At }q_4,q_6\colon \begin{bmatrix}
			 -2 & 0 & & & & \\
			 & 2 & & & * & \\
			 & &  -2 & & & \\
			 & & & -4 & & \\
			 & & & & -4 & \\
			 & & & & & -6
		\end{bmatrix},
		\text{ at }q_5\colon \begin{bmatrix}
			 1 & 0 & & & & \\
			 & 1 & & & * & \\
			 & &  -1 & & & \\
			 & & & -2 & & \\
			 & & & & -2 & \\
			 & & & & & -3
		\end{bmatrix}.
	\end{align*}
\end{lemma}
The equilibria $q_{4,5,6}$ are hyperbolic, but have resonant eigenvalues. Note that in the invariant subspace $\{ a_5 = b_5 = c_5 = \eps_5 = 0\}$ the system \eqref{vf:exit_chart} reduces to \begin{equation}
		\begin{split}
			\label{exit_plane}
			x_5' &= x_5 (1 - x_5^2)\\ 
			r_5' &= r_5 (1 + x_5^2).
		\end{split}
	\end{equation}
	The vector field \eqref{exit_plane} is precisely a directional blow-up of the fast subsystem, in fact blown up by $\Phi_{\ex}$. Its flow is visualized in figure \ref{fig:exit:setting}. Note that the flow of \eqref{exit_plane} appears also in figure \ref{fast_blowup} in the region $\{ \bar{x} + \bar{y} > 0 \}$.

An immediate consequence is the following.
\begin{lemma}
	\label{exit_lin_2}
	The unique 2-d unstable manifold at $q_5$ is given by the plane $\{x_5 \in (-1,1),a_5 = b_5 = c_5 = \eps_5 = 0\}$. The unstable manifold at $q_4$ is given by the line $\{x_5 = -1,a_5 = b_5 = c_5 = \eps_5 = 0\}$, and the unstable manifold at $q_6$ is given by the line $\{x_5 = 1,a_5 = b_5 = c_5 = \eps_5 = 0\}$. The parts of these unstable manifolds with $r_5 > 0$ correspond, via $\Phi_{\ex}$, to the positive $y$-axis and to the positive $x$-axis. The equilibria $q_{4,6}$ appear in $\{ r_5 = 0 \}$ as sinks. The 4-d stable manifold $W^{s}(q_5)$ of $q_5$ is contained in $\{r_5 = 0\}$ and is transverse to the $x_5$-axis in $\{r_5 = 0\}$. The equilibria $q_{4}$ and $q_6$ are sinks in $\{r_5 = 0\}$.
\end{lemma}
\begin{proof}
	The unstable manifolds follow from \eqref{exit_plane}. The identification of the axes follows from \eqref{phi_ex}. Lemma \ref{exit_lin_at_qs} gives $W^{s}(q_5)$ and its transversality.
\end{proof}

We now briefly look at the singular limit dynamics on the blow-up sphere in the exit chart. Restrict the blown-up vector field \eqref{vf:exit_chart} to $\{r_5 = 0, \eps_5 > 0\}$. There we can identify the objects from the analysis in $\kappa_\eps$. Recall that the family of integral curves $\gamma_s$ in the rescaling chart forms an invariant 2-manifold $\Gamma$ on the blow-up sphere. The manifold $\Gamma$ is contained in $\{r_2 = 0, a_2 = 0\}$, which corresponds to $\{r_5 = 0, a_5 = 0, \eps_5 > 0\}$ in the exit chart. In the following, we focus on the dynamics close to $q_{4,5,6}$ and $\Gamma$.

\begin{figure}[h]
	\centering
	\begin{overpic}[width=.8\textwidth]{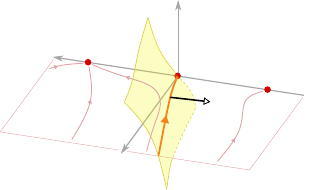}
		\put(66,26.5){$v$}
		\put(53,22){\textcolor{orange}{$\gamma_{s_0}$}}
		\put(28,42){$q_6$}
		\put(56.5,37){$q_5$}
		\put(84,33){$q_4$}
		\put(19,28){\textcolor{rosapastell}{$\gamma_{s < s_0}$}}
		\put(77,20){\textcolor{rosapastell}{$\gamma_{s > s_0}$}}
		\put(5,19){\textcolor{rosapastell}{$\Gamma$}}
		\put(16,43){\textcolor{gray}{$x_5$}}
		\put(36,9){\textcolor{gray}{$\eps_5$}}
		\put(57.5,57){\textcolor{gray}{$a_5 = a_2 \eps_5^{1/3}$}}
		\put(33.5,48){\textcolor{ly}{$W^{s}(q_5)$}}
	\end{overpic}
	\caption{Sketch of the transverse intersection of $\Gamma$ and $W^{s}(q_5)$ inside $\{r_5 = 0, \eps_5 > 0\}$ close to $q_5$. The vector $v(t)$ is a tangent vector of $\Gamma$ located along $\gamma_{s_0}$, being a derivative of the family $\gamma_s$ in direction of the family parameter $s$.}
	\label{fig:eps_to_ex}
\end{figure}

The $x_5$-axis, which is invariant under the flow of \eqref{vf:exit_chart}, serves up to $q_{4,6}$ as unstable manifold for $q_5$. The stable manifold $W^s(q_5)$ at $q_5$ is 4-dimensional and transverse to the $x_5$-axis in $\{r_5 = 0\}$. 
The manifold $W^{s}(q_5)$ is a hypersurface in $\{r_5 = 0\}$ and separates the dynamics between the sinks $q_4$ and $q_6$. The intersection of $\Gamma$ and $W^{s}(q_5)$ is nonempty, due to the existence of $\gamma_{s_0}$. More precisely, we have the following lemma.
\begin{lemma}
	\label{lemma:transverse}
	The manifolds $\Gamma$ and $W^{s}(q_5)$ intersect transversely in $\{r_5 = 0, \eps_5 > 0\}$ near $q_5$. 
\end{lemma}
\begin{proof}
	Consider the tangent space of $\Gamma$ along $\gamma_{s_0}$, denoted by $T_{\gamma_{s_0}}\Gamma$. Let $v(t) = \partial_s\rvert_{(t,s_0)} \gamma_{s} \in T_{\gamma_{s_0}(t)}\Gamma$, which is the tangent vector of $\Gamma$ at $\gamma_{s_0}(t)$ given by the derivative of $\gamma_s$ in direction of the family parameter $s$. For simplicity, denote the $\kappa_\eps$-components of $\gamma_{s_0}(t)$ by $(x_2,y_2,0,b_2,c_2)$. Transforming $v$ into $\kappa_\ex$-coordinates results in
	\begin{equation}
		\begin{split}
			\left( -2x_2 \frac{y_2^2 + c_2}{(x_2 + y_2)^2},0, -8b_2 \frac{y_2^2 + c_2}{(x_2 + y_2)^3}, -8c_2\frac{y_2^2 + c_2}{(x_2+ y_2)^{3}} + \frac{4C_0}{(x_2 + y_2)^{2}} , -24\frac{y_2^2 + c_2}{(x_2 + y_2)^4}\right).
		\end{split}
	\end{equation}
	We know that $\gamma_{s_0}$ converges to $q_5$ for $t \uparrow M$. As $t \uparrow M$ we have $x_2,y_2 \to +\infty$, $x_2/y_2 \to 1$ and $b_2,c_2$ converge to finite values. Scaling $v$ by the scalar $1/x_2$ does not affect the direction of $v$ close to $q_5$. Employing $x_2,y_2 \to +\infty$ and $x_2/y_2 \to 1$ implies that $x_2^{-1} v \to (-1/2,0,0,0,0)$ for $t \uparrow M$, which is a vector along the $x_5$-axis. In $\{ r_5 = 0\}$, the manifold $W^{s}(q_5)$ is of codimension 1  and transverse to the $x_5$-axis. Thus, sufficiently close to $q_5$, the sum of the tangent spaces of $\Gamma$ and $W^{s}(q_5)$ span $\{ r_5 = 0\}$.
\end{proof}
Recall that the invariant 3-manifold $\hat{\Gamma}$ is obtained by perturbation of any $\gamma_s$ to small $\abs{a_2}$.
By transversality, the intersection of $\Gamma$ and $W^{s}(q_5)$ perturbs to the leaves $\{r_2 = 0, a_2 = \text{const}\}$ for $a_2$ close to zero, that is, to $\hat{\Gamma}$. 
Since $q_4$ and $q_6$ are hyperbolic sinks in $\{r_5 = 0\}$, the perturbations of $\gamma_s$ for $s \neq s_0$ in $\hat{\Gamma}$ are also forward asymptotic to $q_4$ or $q_6$, respectively. The perturbations of $\gamma_{s_0}$ must be contained in $W^{s}(q_5)$, due to transversality. In particular, the intersection of $\hat{\Gamma}$ and $W^s(q_5)$ is an invariant 2-manifold, which we denote by $\hat{\Gamma}_{s_0}$.
In summary, the perturbed trajectories in $\hat{\Gamma}$ inherit their forward asymptotic behavior from $\Gamma$.

Recall from the rescaling chart that for small $r_2 > 0$, we want to track perturbations of trajectories in $\hat{\Gamma}$. The manifold $\hat{\Gamma}$ is foliated by trajectories that emanate from $\zeta_2$, and converge to the three equilibria $q_{4,5,6}$. The evolution of the perturbations of $\hat{\Gamma}$ through the chart $\kappa_{\ex}$ will clearly be affected by the hyperbolic, resonant equilibria $q_{4,5,6}$. 

\begin{figure}[h]
	\centering
	\begin{overpic}[width=.65\textwidth]{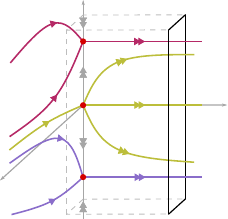}
		\put(84,87){$\Delta^{\ex}_5$}
		\put(97,45){\textcolor{gray}{$r_5$}}
		\put(30,94){\textcolor{gray}{$x_5$}}
		\put(3,15){\textcolor{gray}{$\eps_5$}}
		\put(39,74){$q_6$}
		\put(39,15){$q_4$}
		\put(39,46){$q_5$}
		\put(54,91){\textcolor{gray}{$K$}}
		\put(11,59){\textcolor{q6red}{$\hat{\Gamma}_{s < s_0}$}}
		\put(0,28){\textcolor{q5yellow}{$\hat{\Gamma}_{s_0}$}}
		\put(13,15){\textcolor{q4violet}{$\hat{\Gamma}_{s > s_0}$}}
		\put(91,77){\textcolor{gray}{$\simeq$ $x$-axis}}
		\put(91,17.4){\textcolor{gray}{$\simeq$ $y$-axis}}
	\end{overpic}
	\caption{Sketch of the dynamics near $q_{4,5,6}$, in which the coordinates $a_5,b_5,c_5$ are suppressed. In the $(r_5,x_5)$-plane the flow \eqref{exit_plane} is shown, i.e.\ $a_5,b_5,c_5$ are set to zero. The $(\eps_5,x_5)$-plane represents the 5-d invariant set $\{r_5 = 0\}$, in which trajectories in $\hat{\Gamma}$ approach $q_{4,5,6}$, as sketched.}
	\label{fig:exit:setting}
\end{figure}

The setup of the equilibria $q_{4,5,6}$ is shown in figure \ref{fig:exit:setting}. We are interested in the transition of perturbed trajectories from $\hat{\Gamma}$ away from the blow-up sphere to the section
\begin{equation}
	\begin{split}
		\label{resc:to_ex}
		\Delta_5^{\ex} = \{ r_5 = \nu \}
	\end{split}
\end{equation}
close to $q_{4,5,6}$ via the flow of \eqref{vf:exit_chart}. Due to hyperbolicity of $q_{4,5,6}$ the qualitative behavior of the flow close to these equilibria is organized by the stable and unstable manifolds. In particular, we expect trajectories starting with small $r_5 > 0$ to escape from the saddles $q_{4,5,6}$ close to the unstable manifolds contained in the $(x_5,r_5)$-plane. That is, close to the singular limit flow \eqref{exit_plane}. The key difference between $q_5$ and $q_{4,6}$ is that at $q_5$ a second unstable direction is present, namely $x_5$.
Therefore, we expect $q_5$ to have the effect of fanning out trajectories in-between $x_5 = \pm 1$, i.e.\ the positive $x$- and $y$-axis. Furthermore, we expect that most trajectories will transition to the fast regime close to $q_{4,6}$ and escape along the positive $x$- and $y$-axis, since on the blow-up sphere $\hat{\Gamma}_{s_0}$ is of codimension one in $\hat{\Gamma}$. Also refer to figure \ref{fig:exit:setting}.

The relevant vector field for the transition is \eqref{vf:exit_chart}. The value of the function $f_x + f_y$ is constant in $\{r_5 = 0\}$. Hence, the function $F = 1 + x_5^2 + a_5 + (b_5 + c_5)/2 + r_5\eps_5 (f_x + f_y)/2$ is positive close to the $x_5$-axis. Dividing out $F$ gives, close to the $x_5$-axis, an orbitally equivalent vector field of the form
\begin{equation}
	\begin{split}
		\label{F_divided_out}
		x_5' &= -x_5 + \frac{2x_5 - a_5x_5 + (b_5 - c_5)/2}{F} + r_5\eps_5 f_1\\
		r_5' &= r_5 \\
		a_5' &= -a_5 + r_5\eps_5 g_a/F\\
		b_5' &= - 2b_5 + \eps_5 g_b/F\\
		c_5' &= - 2c_5 + \eps_5 g_c/F\\
		\eps_5' &= - 3\eps_5 ,
	\end{split}
\end{equation}
for a sufficiently regular function $f_1$. As we have seen above, trajectories in $\hat{\Gamma} \subset \{ r_5 = 0\}$ converge to one of $q_{4,5,6}$. We now aim to track $r_2$-perturbations of trajectories in $\hat{\Gamma}$ via the flow of \eqref{F_divided_out}. 

Let us choose a compact box $\tilde{K}$ about the $x_5$-axis, such that $q_{4,5,6}$ are contained in its interior. On $\tilde{K}$ the functions $g_a/F,g_b/F,g_c/F$ are bounded. Inside $\tilde{K}$, consider another box of the form 
\begin{equation}
	K = \{ x_5 \in [-\Xi,\Xi], r_5 \in [0,\nu], \eps_5 \in [0,\delta], a_5 \in [-a_M,a_M], b_5 \in [-b_M,b_M], c_5 \in [-c_M,c_M] \}
\end{equation}
We can find sufficiently small $\nu,\delta,a_M,b_M,c_M > 0$ such that integral curves of \eqref{F_divided_out} can exit $K$ only through $\{r_5 = \nu\}$ or $\{ x_5 = \pm \Xi \}$. Note that the $x_5$-axis is invariant and the flow on it is directed towards $q_{4,6}$ for $\abs{x_5} > 1$. Hence we can even find small $\nu,\delta,a_M,b_M,c_M > 0$ and $\Xi > 1$ such that the flow can exit $K$ only through $\{r_5 = \nu\}$.

Next, consider a trajectory with initial condition $(x_k,r_k,a_k,b_k,c_k,\eps_k) \in K$. This trajectory needs to exit $K$ through $\Delta^{\ex}_5 = \{r_5 = \nu\}$. Assume $r_k \in (0,\nu)$. We have from \eqref{F_divided_out} that $r_5(t) = r_k e^{t},\; \eps_5(t) = \eps_k e^{-3t}$. This implies a travel duration $T = \log(\frac{\nu}{r_k})$ to $\Delta^{\ex}_5$. Let us now propose the ansatz
\begin{equation}
	\begin{split}
		a_5(t) &= e^{-t}(a_i + z_a(t)),\\
		b_5(t) &= e^{-2t}(b_i + z_b(t)),\\
		c_5(t) &= e^{-2t}(c_i + z_c(t)),
	\end{split}
\end{equation}
with $z_a(0) = z_b(0) = z_c(0) = 0$. It follows that 
\begin{equation}
	\begin{split}
		z_a' &= e^{-t} r_k \eps_k \frac{g_a}{F},\\
		z_b' &= e^{-t} \eps_k \frac{g_b}{F},\\
		z_c' &= e^{-t} \eps_k \frac{g_c}{F},
	\end{split}
\end{equation}
where the functions are evaluated along the corresponding trajectory. Since
\begin{equation}
	\begin{split}
		z_a(T) = r_k \eps_k \int_0^T   \frac{g_a}{F}e^{-t} \d t,
	\end{split}
\end{equation}
we directly obtain $z_a(T) = O(r_k \eps_k)$. Similarly
\begin{equation}
	\begin{split}
		z_b(T) = \eps_k \int_0^{T} \frac{g_b}{F}e^{-t} \d t,
	\end{split}
\end{equation}
which is bounded as $r_k \downarrow 0$ and its bound scales with factor $\eps_k$. In fact, we can write $g_b/F = B_0/F_1 + r_5 f_2$, where $F_1 = F\rvert_{r_5 = 0}$ and $f_2$ is a sufficiently regular function. This leads, for $r_k \to 0$, to
\begin{equation}
	\begin{split}
		z_b(T) = \eps_k \int_0^{T}  \frac{B_0}{F_1} e^{-t} \d t + \eps_k r_k \int_0^{T} \frac{f_2}{F} \d t = O(1) + O(r_k \log r_k).
	\end{split}
\end{equation}
A similar calculation holds for $z_c(T)$. Note that by the above considerations, we obtain the structure of \emph{any} transition in $K$ to $\Delta^{\ex}_5$ in the variables $r_5,a_5,b_5,c_5,\eps_5$. We are still missing the $x_5$-direction, so we proceed as follows: In $\{a_5 = b_5 = c_0 = \eps_5 = 0\}$ the vector field \eqref{F_divided_out} reduces to
\begin{equation}
	\begin{split}
		\label{x-dir}
		x_5' = \frac{x_5(1-x_5^2)}{1 + x_5^2},
	\end{split}
\end{equation}
which is separable and can be integrated by employing a partial fraction decomposition. Let $\bar{x}(t,x_k)$ denote the integral curve of \eqref{x-dir} with $\bar{x}(0,x_k) = x_k$. Clearly $\bar{x}(t,-x_k) = - \bar{x}(t,x_k)$, and $\bar{x}$ is explicitly given by
\begin{equation}
	\begin{split}
		\bar{x}(t,x_k) = \begin{cases}
			\dfrac{\abs{1-x_k^2}}{2\abs{x_k}}e^{-t} + \sqrt{\dfrac{\abs{1-x_k^2}^2}{4\abs{x_k}^2}e^{-2t} + 1}, &\quad x_k > 1\\[2ex]
			1, &\quad x_k = 1\\[2ex]
			-\dfrac{\abs{1-x_k^2}}{2\abs{x_k}}e^{-t} + \sqrt{\dfrac{\abs{1-x_k^2}^2}{4\abs{x_k}^2}e^{-2t} + 1}, &\quad x_k \in (0,1)\\[2ex]
			0, &\quad x_k = 0.
		\end{cases}
	\end{split}
\end{equation}
For the $x_5'$ equation in \eqref{F_divided_out} we take the ansatz $x_5(t) = \bar{x}(t) + z_x(t)$ with $z_x(0) = 0$. Rewriting the first equation of \eqref{F_divided_out} to
\begin{equation}
	\begin{split}
		x_5' = -x_5 + \frac{2x_5 - a_5 x_5 + (b_5 - c_5)/2}{1 + x_5^2 + a_5 + (b_5 + c_5)/2} + r_5\eps_5 f_2,
	\end{split}
\end{equation}
for some sufficiently regular $f_2$, and employing a Taylor expansion at $(x_5,0,0,0,0,0)$ in $\{r_5 = \eps_5 = 0\}$ leads to the equation $z_x' = a_5 f_3 + b_5 f_4 + c_5 f_5 + r_5\eps_5 f_2$, where $f_3,f_4,f_5$ are sufficiently regular functions. Hence
\begin{equation}
	\begin{split}
		z_x(T) = \int_0^T a_5 f_3 \d t + \int_0^T b_5 f_4 \d t + \int_0^T c_5 f_5 \d t + \int_0^T r_5 \eps_5 f_2 \d t,
	\end{split}
\end{equation}
where the functions are evaluated along the corresponding integral curve. Bounding the functions, plugging in $a_5(t) = e^{-t}(a_k + z_a(t)), b_5(t), c_5(t), r_5(t), \eps_5(t)$ and, where necessary, integrating by parts, shows that $z_x(T) = O(r_k\eps_k,a_k,b_k,c_k,\eps_k)$. Thus, with the analysis performed above, we obtain the following result that summarizes the transition in the present chart $\kappa_\ex$.
\begin{prop}
	\label{ex:trans}
	Any transition $\Pi_5\colon K \to \Delta_5^{\ex} = \{ r_5 = \nu\}$ via the flow of \eqref{F_divided_out} in $K$ is of the form
	\begin{equation}
		\begin{split}
			\begin{bmatrix}
				x_k\\[2ex]
				r_k\\[2ex]
				a_k\\[2ex]
				b_k\\[2ex]
				c_k\\[2ex]
				\eps_k
			\end{bmatrix} \mapsto
			\begin{bmatrix}
				\Pi_{x_5}(x_k) + O(r_k\eps_k,a_k,b_k,c_k,\eps_k)\\[1ex]
				\nu\\[1ex]
				\dfrac{r_k}{\nu}a_k + O(r_k^2\eps_k)\\[2ex]
				\left(\dfrac{r_k}{\nu}\right)^2(b_k + O(\eps_k))\\[2ex]
				\left(\dfrac{r_k}{\nu}\right)^2(c_k + O(\eps_k))\\[2ex]
				\left(\dfrac{r_k}{\nu}\right)^3\eps_k
			\end{bmatrix},
		\end{split}
	\end{equation}
	where the $O(\eps_k)$ terms are bounded as $r_k \to 0$, $\Pi_{x_5}(-x_k) := - \Pi_{x_5}(x_k)$ and
	\begin{equation}
		\begin{split}
			\label{Pix5}
			\Pi_{x_5}(x_k) = \begin{cases}
				\frac{\abs{1-x_k^2}}{2\abs{x_k}}\frac{r_k}{\nu} + \sqrt{\frac{\abs{1-x_k^2}^2}{4\abs{x_k}^2}(\frac{r_k}{\nu})^2 + 1}, & x_k > 1\\
				1, & x_k =  1\\
				- \frac{\abs{1-x_k^2}}{2\abs{x_k}}\frac{r_k}{\nu} + \sqrt{\frac{\abs{1-x_k^2}^2}{4\abs{x_k}^2}(\frac{r_k}{\nu})^2 + 1}, & x_k \in (0,1)\\
				0, & x_k = 0.
			\end{cases}
		\end{split}
	\end{equation}
\end{prop}

\subsection{Proof of Theorem \ref{thm:main_thm}}
\label{sec:main_proof}
In the following proof we collect the results from sections \ref{sec:entry} to \ref{s:exit_chart} and ``glue'' them together. Recall that we essentially work inside the continuation of $W^c(\mathcal{S}^{\txta}_{0,1})$, which is a 4-dimensional invariant manifold containing the 3-d attracting slow manifolds.
\begin{proof}[Proof of Theorem \ref{thm:main_thm}]
We begin by choosing an open set of initial conditions $V_1 \subset \Delta^{-y \to \eps} \cap W^c(\mathcal{S}^{\txta}_{0,1})\rvert_{r_1 > 0}$. The analysis in the rescaling chart dictates, which trajectories with initial condition in $V_1$ we are able to track. After shrinking $V_1$ suitably, we can assume the following two assertions: First, the flow from Lemma \ref{lemma:CM_red} takes $V_1$ in backward time to $\Delta^{\en}_1$. Secondly, the set $V_2 = \kappa_{-y \to \eps}(V_1) \subset \Delta^{-y \to \eps}_2$ is lead by the flow of \eqref{blowup_eps} along the invariant 3-manifold $\hat{\Gamma} \subset \{ r_2 = 0 \}$ through the rescaling chart close to $q_{4,5,6}$. Refer to Lemma \ref{resc:trans} and before for the definition of $\hat{\Gamma}$. The trajectories in $\hat{\Gamma}$ are forward asymptotic to one of the hyperbolic saddles $q_{4,5,6}$. Hence, by choosing $r_2 > 0$ sufficiently small (i.e.\ shrinking $V_1$ accordingly), we can assume that $\Pi_2$ maps $V_2$ to (faces of) $K$. Here $\Pi_2$ is the transition map from Lemma \ref{resc:trans} and $K$ the compact box defined for Proposition \ref{ex:trans} in which the transition map $\Pi_5$ is defined. 
Recall that $V_1$ lies in $W^c(\mathcal{S}^{\txta}_{0,1})$, and hence we can let $V_1$ flow in backward time via \eqref{CM_red2} to $\Delta^{\en}_1 \cap W^{c}(\mathcal{S}^{\txta}_{0,1})$. This gives the domain $I_1$ for the transition $\Pi_1$ inside the center manifold. In backward time \eqref{CM_red2} leads to exponential contraction towards the 1-d unstable manifold, which is given by $\sigma_1 \subset \mathcal{S}^{\txta}_{0,1}$. Notice that $I_1$ then degenerates to the single point $p = \sigma_1 \cap \Delta^{\en}_1$ for $\eps_1 \to 0$.

We now distinguish three regions in $I_1$, which give qualitatively different jumping behavior. Recall that $\hat{\Gamma}_{s_0} := \hat{\Gamma} \cap W^{s}(q_5)$ is an invariant 2-manifold on the blow-up sphere, consisting of perturbations of $\gamma_{s_0}$. Similarly, let $\hat{\Gamma}_{s > s_0}$ and $\hat{\Gamma}_{s < s_0}$ denote the invariant submanifolds of the 3-manifold $\hat{\Gamma}$ consisting of perturbations of some $\gamma_{s > s_0}$ or $\gamma_{s < s_0}$, respectively. For small $r_2 > 0$ the manifold $\hat{\Gamma}_{s_0}$ gets perturbed away from the blow-up sphere to an invariant 3-manifold, let us denote it by $\Sigma$. Recall that $W^c(\mathcal{S}^{\txta}_{0,1})$ is 4-dimensional. Then $V_2 \subset \Delta^{-y \to \eps}_2 \cap W^c(\mathcal{S}^{\txta}_{0,1})$ is a 3-manifold. Consequently $\Sigma$ intersects $V_2$ in a 2-manifold. Limiting to small $r_2 > 0$, $\Sigma \cap V_2$ splits $V_2$ into two disjoint open subsets $U^{q_4}$ and $U^{q_6}$, where $U^{q_4}$ contains perturbations of $\hat{\Gamma}$ that connect to $q_4$. Furthermore, let $U^{q_5}$ denote a (small) neighborhood of $\Sigma \cap V_2$ in $V_2$. The sets $U^{q_4},\, U^{q_5},\, U^{q_6}$ are an open cover of $V_2$ and correspond to distinct jumping behaviors. The set $U^{q_5}$ leads to trajectories jumping and fanning out close to $q_5$, whereas $U^{q_4}$ and $U^{q_6}$ lead to trajectories jumping close to $q_4$, respectively $q_6$, with contraction towards the 1-d unstable manifolds $W^{u}(q_4)$ (the $y$-axis in original coordinates), and $W^{u}(q_6)$ (the $x$-axis in original coordinates) respectively. 

Next, we let the aforementioned three regions $U^{q_4}$,  $U^{q_5}$, and  $U^{q_6}$ flow in backward time to $I_1$, which gives three open sets $R^{q_4},R^{q_5},R^{q_6}$ covering $I_1$. In backward time, i.e.\ $\Pi_1^{-1}$, the flow contracts towards $\sigma_1$. Hence, in the singular limit $\eps_1 \to 0$, the set $I_{1}$ collapses in wedge-like fashion into the single point corresponding to $\sigma_1 \cap \Delta^{\en}$: by this we mean that $\eps$-slices through $I_1$ converge in Hausdorff distance to the single point, after discarding the $\eps$-direction.
Similar statements holds for any of the 3-manifolds $R^{q_4},R^{q_5},R^{q_6}$. We refer to figure \ref{fig:complete_transition} for an overview sketch of the situation. Finally, we ``blow-down'' by considering the transition map $\hat{\Pi}\colon I \to \Delta^{\ex}$ defined by
\begin{equation}
	\begin{split}
		\label{trans_complete}
		\hat{\Pi} = \Phi_{\ex} \circ \Pi_{5} \circ \kappa_{\eps \to \ex} \circ \Pi_2 \circ \kappa_{-y \to \eps} \circ \Pi_1 \circ \Phi_{-y}^{-1},
	\end{split}
\end{equation} 
where $I = \Phi_{-y}(I_1)$ is a 3-manifold. By construction the map $\hat{\Pi}$ is the transition map from $\Delta^{\en}$ to $\Delta^{\ex}$ along the flow of $X$ given by  \eqref{hyp_umb_exp}. Now let $I_{\eps}$ denote a slice through $I$ for $\eps = \text{const}$. Then $\Pi$ from Theorem \ref{thm:main_thm} is given by $\hat{\Pi}\rvert_{I_{\eps}}$ and $I_\eps$ is a 2-manifold. Clearly, we discard the $\eps$-direction when we take slices at fixed value of $\eps$. It follows that $I_\eps$ converges in Hausdorff distance to a single point. Each $I_{\eps}$ also contains the corresponding $\eps$-slices through $R^{q_4},R^{q_5},R^{q_6}$. Let us denote these slices by $L^{y}_{\eps}, R_{\eps}$ and $L^{x}_{\eps}$, such that $L^{y}_{\eps}$ corresponds to an $\eps$-slice through $R^{q_5}$, similarly $R_{\eps}$ to an $\eps$-slice through $R^{q_5}$ and $L^{x}_\eps$ to an $\eps$-slice through $R^{q_6}$. We refer to $L^{x}_\eps$ and $L^{y}_\eps$ as the lateral regions, which correspond to initial conditions escaping along the $x$-axis, respectively the $y$-axis. The sets $L^{x}_\eps$ and $L^{y}_\eps$ are disjoint, and $I_{\eps} = L^{x}_\eps \cup R_{\eps} \cup L^{y}_\eps$, where all sets are open. 

Recall that $J$ is the closed line segment defined by $\Delta^{\ex} \cap \{ x \ge 0, y \ge 0, a = b = c = 0\}$ and let $j_x$, respectively $j_y$, denote its points of intersection with the $x$-axis, respectively $y$-axis. We need to argue that $\Pi(L^{x}_\eps) \to j_x$, $\Pi(L^{y}_\eps) \to j_y$ and $\Pi(R_{\eps}) \to J$ in Hausdorff distance for $\eps \to 0$. Let us assume for a moment that the scalings $a = O(\eps^{1/3}),\, b = O(\eps^{2/3}),\, c = O(\eps^{2/3})$ for $(x,y,a,b,c) \in \Pi(I_\eps)$ as $\eps \to 0$ hold. Then the assertions follow from the definition of the sets $U^{q_4},\, U^{q_5},\, U^{q_6}$ and from the fact that $\Pi$ is continuous. Indeed, $U^{q_5}$ is a neighborhood of $\Sigma \cap V_2$ in $V_2$ and the structure of the transition map $\Pi_5$ in Proposition \ref{ex:trans}, in particular \eqref{Pix5} for $r_k \to 0$, imply the assertions.

\begin{figure}[h]
	\begin{overpic}[width=.85\textwidth]{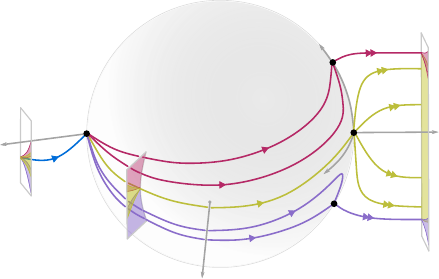}
		\put(25,55){\textcolor{gray}{$\mathbb{S}^{5}$}}
		\put(74,54){\textcolor{gray}{$x_5$}}
		\put(99,30){\textcolor{gray}{$r_5$}}
		\put(48,0){\textcolor{gray}{$r_2$}}
		\put(0,33){\textcolor{gray}{$r_1$}}
		\put(72,22){\textcolor{gray}{$\eps_5$}}
		\put(15,26){\textcolor{light-blue}{$\sigma_1$}}
		\put(7,39){\textcolor{light-gray}{$\Delta^{\en}$}}
		\put(29,30){\textcolor{light-gray}{$\Delta^{-y \to \eps}$}}
		\put(95,58){\textcolor{light-gray}{$\Delta^{\ex}$}}
		\put(16,35){$\zeta_2$}
		\put(76,13){$q_4$}
		\put(83,31){$q_5$}
		\put(78,48){$q_6$}
		\put(52,23.8){\textcolor{q6red}{$\hat{\Gamma}_{s < s_0}$}}
		\put(60,19.8){\textcolor{q5yellow}{$\hat{\Gamma}_{s_0}$}}
		\put(49,12.5){\textcolor{q4violet}{$\hat{\Gamma}_{s > s_0}$}}
		\put(8,28.5){\textcolor{q6red}{$R^{q_6}$}}
		\put(0,26){\textcolor{q5yellow}{$R^{q_5}$}}
		\put(8,21){\textcolor{q4violet}{$R^{q_4}$}}
		\put(33,24){\textcolor{q6red}{$U^{q_6}$}}
		\put(24,14){\textcolor{q5yellow}{$U^{q_5}$}}
		\put(24,10){\textcolor{q4violet}{$U^{q_4}$}}
	\end{overpic}
	\caption{Rough sketch of the transition $\Pi_{5} \circ \kappa_{\eps \to \ex} \circ \Pi_2 \circ \kappa_{-y \to \eps} \circ \Pi_1$ for initial conditions in $I_1 = R^{q_4} \cup R^{q_5} \cup R^{q_6}$ in blow-up space ``inside'' the continuation of the 4-d invariant manifold $W^c(\mathcal{S}^{\txta}_{0,1})$. We remark that the dimensions of the sketched objects have been reduced for visualization reasons. The colored regions inside the sections indicate the 3-manifolds $R^{q_4},R^{q_5},R^{q_6} \subset \Delta^{\en}$, respectively $U^{q_4}, U^{q_5}, U^{q_6} \subset \Delta^{-y \to \eps}$, and their evolution steered by singular limit candidate trajectories up to $\Delta^{\ex}$. Note that the fanning-out behavior of $R^{q_5}$ is visible. Slices of the form $\eps = \text{const}$ through $R^{q_5}$ give $R_{\eps}$ in Theorem \ref{thm:main_thm}. The sketched coordinate axes correspond by abuse of notation to the euclidean coordinates transported to blow-up space via the charts.}
	\label{fig:complete_transition}
\end{figure}

Next, we prove the scalings $a = O(\eps^{1/3})$, $b = O(\eps^{2/3})$ and $c = O(\eps^{2/3})$ for $(x,y,a,b,c) \in \Pi(I_{\eps})$. Observe that Proposition \ref{ex:trans} already suggests these scalings. However, one should note at this point, that in general $\Pi_2$ from Lemma \ref{resc:trans} passes $a_2$ unaffected to leading order through the rescaling chart. This implies that in the exit chart, the scaling of $a_1$, caused by $\Pi_1$ in the entry chart, remains present. The structure of \eqref{CM_red2} suggests that $\Pi_{1a}(x_i,a_i,\nu,\eps_i) = O(\eps_i^{-1/3})$, which would cancel out with the leading order scaling of $\Pi_5$. Thus, heuristically, we should not see any a priori scaling in the $a$-variable. However, in our setting of $\Pi(I_{\eps})$, we expect to see ``artificial scalings'' caused by our choice of domain $I_{\eps}$, which shrinks to a single point $p$ for $\eps \to 0$. Indeed, as we will shortly show, the scalings $a = O(\eps^{1/3})$, $b = O(\eps^{2/3})$ and $c = O(\eps^{2/3})$ appear because of our choice of initial conditions in $I_{\eps}$.

We build the composition $\Pi_5 \circ \kappa_{\eps \to \ex} \circ \Pi_2 \colon V \to \Delta^{\ex}_5$ to see the scalings formally. From Lemma \ref{resc:trans} we get $\Pi_2(x_j,-\delta^{-1/3},a_j,b_j,c_j,r_j) = (x_e + O(a_j,r_j),y_e + O(a_j,r_j),a_j + O(r_j),b_e + O(a_j,r_j), c_e + O(a_j,r_j), r_j) \in \kappa_{\ex \to \eps}(\partial K \setminus \{ r_5 = \nu \})$.
According to Lemma \ref{coordinate_changes} we can write $ x_e + y_e + O(a_j,r_j) = 2\eps_k^{-1/3}$ for $\eps_k \in (0,\delta]$, which implies
\begin{equation}
	\begin{split}
		(\kappa_{\eps \to \ex} \circ \Pi_2)(x_j,-\delta^{-1/3},a_j,b_j,c_j,r_j) = 
		\begin{bmatrix}
			\frac{1}{2}\eps_k^{1/3}(x_e - y_e + O(a_j,r_j))\\[1ex]
			r_j \eps_k^{-1/3}\\[1ex]
			\eps_k^{1/3}(a_j + O(r_j))\\[1ex]
			\eps_k^{2/3}(b_e + O(a_j,r_j))\\[1ex]
			\eps_k^{2/3}(c_e + O(a_j,r_j))\\[1ex]
			\eps_k
		\end{bmatrix}.
	\end{split}
\end{equation}
Applying Proposition \ref{ex:trans} leads to $\Pi_5 \circ \kappa_{\eps \to \ex} \circ \Pi_2 \colon V_2 \to \Delta^{\ex}_5$ given by
\begin{equation}
	\begin{split}
		\label{final:transition25}
		\begin{bmatrix}
			x_j\\
			-\delta^{-1/3}\\
			a_j\\
			b_j\\
			c_j\\
			r_j
		\end{bmatrix} \mapsto
		\begin{bmatrix}
			\ast\\
			\nu\\
			\frac{r_j}{\nu}(a_j + O(r_j^2))\\[1ex]
			(\frac{r_j}{\nu})^2 (b_e + O(a_j,r_j) + O(\eps_k^{2/3}))\\[1ex]
			(\frac{r_j}{\nu})^2 (c_e + O(a_j,r_j) + O(\eps_k^{2/3}))\\[1ex]
			(\frac{r_j}{\nu})^3
		\end{bmatrix},
	\end{split}
\end{equation}
where the $O(\eps_k^{2/3})$ terms are bounded as $r_k \to 0$ (and $r_j \to 0$) and we leave the entry $\ast$ unspecified. First, recall that $r_j = \nu \eps_i^{1/3} = \eps^{1/3}$. Observe that $a_j$ appears in the right hand side of \eqref{final:transition25}, so in general can possibly affect the scalings if $a_j$ blows up for $\eps \to 0$. In our setting, however, we have chosen the domain $I_\eps$ in such a way, that we only encounter small, in fact bounded values of $a_j$ as $\eps \to 0$. This implies the scalings in $a$, $b$ and $c$ as claimed in Theorem \ref{thm:main_thm}.

It remains to consider the case that $g_a$ is factored by $a$. For this case we have Lemma \ref{lemma:CM_trans1}, which implies
\begin{equation}
	\begin{split}
		a_j = \delta^{-1/3} \Pi_{1a}(x_i,a_i,\nu,\eps_i) = \eps_i^{-1/3} O(a_i) + O(\eps_i^{2/3} \log \eps_i).
	\end{split}
\end{equation}
This implies that $I_1\rvert_{\eps_i}$ is of width $a_i = O(\eps_i^{1/3})$ as $\eps_i \to 0$. Thus, for the case of $g_a$ being factored by $a$, this gives that $I_\eps$ is a thin strip of width $O(\eps^{1/3})$ in $a$-direction, for $\eps \to 0$.
\end{proof}

\section{Further observations} 
\label{sec:further}

In this last section we briefly present further observations, obtained by analyzing the blow-up $\bar{X}$ of \eqref{hyp_umb_exp} in the two charts $\kappa_{\pm a}$. The main observation is that in the blow-up $\bar{X}$, one generically encounters fast-slow systems on the blow-up sphere with folded saddles or folded centers. These organize the flow on the blow-up sphere close to the equator $\bar{\eps} = 0$.

To formulate the observations, we look at $\bar{X}$ in the two charts $\kappa_{\pm a}$. The blow-up of $X$ \eqref{hyp_umb_exp} in the chart $\kappa_a$ is, after suitable time rescaling, given by
\begin{equation}
	\begin{split}
		\label{blowup_a}
		x_3' &= x_3^2 + y_3 + b_3  + O(r_3 \eps_3)\\
		y_3' &= y_3^2 + x_3 + c_3 + O(r_3 \eps_3)\\
		r_3' &= A_0 r_3^2\eps_3 + O(r_3^3\eps_3)\\
		b_3' &= B_0\eps_3 + O(r_3 \eps_3)\\
		c_3' &= C_0 \eps_3 + O(r_3 \eps_3)\\
		\eps_3' &= - 3 A_0 r_3 \eps_3^2 + O(r_3^2 \eps_3^2).
	\end{split}
\end{equation}

Similarly, in the chart $\kappa_{-a}$,  we obtain the blown-up vector field
\begin{equation}
	\begin{split}
		\label{blowup_-a}
		x_4' &= x_4^2 - y_4 + b_4  + O(r_4^2 \eps_4)\\
		y_4' &= y_4^2 - x_4 + c_4  + O(r_4^2 \eps_4)\\
		r_4' &= - A_0 r_4^2\eps_4 + O(r_4^3\eps_4)\\
		b_4' &= B_0\eps_4 + O(r_4 \eps_4)\\
		c_4' &= C_0 \eps_4 + O(r_4 \eps_4)\\
		\eps_4' &=  3 A_0 r_4 \eps_4^2 + O(r_4^2 \eps_4^2).
	\end{split}
\end{equation}
We restrict the analysis of \eqref{blowup_a} and \eqref{blowup_-a} to the invariant blow-up sphere, i.e.\ we put $r_3 = 0$ and $r_4 = 0$, respectively. For brevity we give our observations mainly for \eqref{blowup_-a} and only indicate the similar statements for \eqref{blowup_a}. In $\{r_4 = 0\}$ \eqref{blowup_-a} simplifies to
\begin{equation}
	\begin{split}
		\label{fs:-a}
		x_4' &= x_4^2 - y_4 + b_4  \\
		y_4' &= y_4^2 - x_4 + c_4  \\
		b_4' &= B_0\eps_4 \\
		c_4' &= C_0 \eps_4 \\
		\eps_4' &=  0,
	\end{split}
\end{equation}
which is a fast-slow system near the equator $\{ \bar{\eps} = 0\} \cap \mathbb{S}^5$, i.e.\ for small $\eps_4 > 0$.
The 2-dimensional critical manifold of \eqref{fs:-a} is given by $E_4 := \{x_4^2 - y_4 + b_4 = 0,\, y_4^2 - x_4 + c_4 = 0\}$, which is the intersection of the blown-up critical manifold of $\bar{X}$ with the blow-up sphere in $\kappa_{-a}$. Via $\kappa_{-y \to -a}$ the manifold $\mathcal{S}_{0,1}\eval_{r_1 = 0,a_1 < 0}$ corresponds to $E_4\eval_{y_4 < 0}$, and so in particular, if $B_0 C_0 > 0$, the folded saddle $\zeta_1$ from \eqref{zetas} is present in \eqref{fs:-a}. Furthermore, we have $\eps_4 = - a_2^{-3}$, so the invariant sets $\{r_2 = 0, a_2 = \text{const}\}$ of the coupled Riccati system \eqref{riccati_on_sphere} in $\kappa_{\eps}$ appear in form of the fast-slow system \eqref{fs:-a} in $\kappa_{-a}$.
The invariant sets $\{ \eps_4 = \text{const} \}$ for \eqref{fs:-a} imply that there do not exist trajectories in $\{ r_2 = 0\}$, which approach in forward or backward time points on the half-equator $\{ \bar{\eps} = 0, \bar{a} < 0 \} \cap \mathbb{S}^5$. In particular, this implies

\begin{lemma}
	\label{zeta1}
	Let $B_0 C_0 > 0$.
	Then no integral curve of \eqref{riccati_on_sphere} is forward or backward asymptotic to $\zeta_1$.
\end{lemma}

In case of $B_0 C_0 < 0$, the folded equilibrium $\zeta_1$ is a folded center and visible in $\kappa_{a}$ instead of $\kappa_{-a}$. Hence, to obtain a similar statement in this case, we need to employ \eqref{blowup_a} in $\{r_3 = 0\}$, which explicitly is the fast-slow system
\begin{equation}
	\begin{split}
		\label{fs:a}
		x_3' &= x_3^2 + y_3 + b_3  \\
		y_3' &= y_3^2 + x_3 + c_3 \\
		b_3' &= B_0\eps_3 \\
		c_3' &= C_0 \eps_3 \\
		\eps_3' &= 0,
	\end{split}
\end{equation}
for small $\eps_3 > 0$. A similar argument then implies that Lemma \ref{zeta1} also holds for $B_0 C_0 < 0$. In fact, in each of these cases, a second folded equilibrium is present, which we denote by $\zeta_6$. That is, $\zeta_6$ is the folded singularity of \eqref{fs:a} in case of $B_0 C_0 > 0$, but in case of $B_0 C_0 < 0$ $\zeta_6$ appears as folded singularity of \eqref{fs:-a}. Analogous results to Lemma \ref{zeta1} hold for $\zeta_6$ instead of $\zeta_1$ if $B_0 C_0 \neq 0$.

\subsection{Folded singularities in $\bar{X}$ on the blow-up sphere}
The fast-slow systems \eqref{fs:-a} and \eqref{fs:a} appear in the blow-up $\bar{X}$ in the charts $\kappa_{\mp a}$ on the blow-up sphere near the equator $\{\bar{\eps} = 0\} \cap \mathbb{S}^5$. In fact, the coupled Riccati system \eqref{riccati_on_sphere} on the blow-up sphere in $\kappa_{\eps}$ corresponds to \eqref{fs:-a} for $a_2 < 0$ and to \eqref{fs:a} for $a_2 > 0$. As noted already, \eqref{fs:-a} and \eqref{fs:a}, have folded singularities. Here, we present a brief analysis of \eqref{fs:-a} and comment on the analogous results for \eqref{fs:a} afterwards.

The critical manifold $E_4$ of \eqref{fs:-a} is parametrized by $(x_4,y_4)$ and has fold points along $4 x_4 y_4 = 1$ and a cusp point located at $x_4 = y_4 = -1/2$, which is implied by Lemma \ref{classif_singular_points} restricted to $\{ a = -1 \}$.
The desingularized slow flow of \eqref{fs:-a} is given by
\begin{equation}
	\begin{split}
		\label{fs:-a:sf:desing}
		U \colon
		\begin{bmatrix}
			\dot{x}_4\\ \dot{y}_4
		\end{bmatrix} &=
		\begin{bmatrix}
			0 & -2B_0\\
			-2 C_0 & 0
		\end{bmatrix}
		\begin{bmatrix}
			x_4\\ y_4
		\end{bmatrix}
		- \begin{bmatrix}
			C_0\\ B_0
		\end{bmatrix}
	\end{split}
\end{equation}
with reversed time orientation of orbits in the saddle-type region $\{ 4x_4 y_4 < 1\}$. The desingularized slow flow $U$ is analyzed in the following Lemma. A sketch of the phase portrait is given in figure \ref{sf:-a:portrait}.

\begin{lemma}
	\label{lemma:-a:sf}
	Assume $B_0C_0\neq0$. The desingularized slow flow $U$ \eqref{fs:-a:sf:desing} of \eqref{fs:-a} has a single folded singularity. If $B_0 C_0 > 0$, a folded saddle is located on the negative branch of $4 x_4 y_4 = 1$, which corresponds to $\zeta_1$. There exists a singular canard passing from $\mathcal{S}_0^{\txta}$ to $\mathcal{S}_0^{\txts}$ and a singular faux canard passing from $\mathcal{S}_0^{\txts}$ to $\mathcal{S}_0^{\txta}$. In particular, if $B_0 = C_0$, $\zeta_1$ is located at the cusp point. If $B_0 C_0 < 0$, a folded center is located on the positive branch of $4 x_4 y_4 = 1$, which corresponds to $\zeta_6$ in this case.
\end{lemma}
\begin{proof}
	The equilibrium of $U$ is located at $x_4 = -\frac{B_0}{2C_0},\, y_4 = - \frac{C_0}{2B_0}$, which is on the lower fold curve if $B_0,C_0$ have the same sign. If $B_0 = C_0$ it coincides with the cusp point. Eigenvalues of the linearization are $\pm 2\sqrt{B_0C_0}$, the situation partly resembles Lemma \ref{sf:origin_lin}. For $B_0,C_0 > 0$, the corresponding eigenvectors are \begin{equation}
		\begin{split}
			-2\sqrt{B_0C_0}\colon (\sqrt{B_0/C_0}, 1 ), \quad 2 \sqrt{B_0C_0}\colon (-\sqrt{B_0C_0},1).
		\end{split}
	\end{equation}
	For $B_0,C_0 < 0$ the eigenvectors swap, similar to Lemma \ref{sf:origin_lin}.
	An eigenvector is tangent to the fold curves if and only if $B_0 = C_0$. This is precisely the case when the folded saddle coincides with the cusp point. If $B_0,C_0 > 0$ and $B_0 \neq C_0$, the stable/unstable manifolds of $\zeta_1$ are transverse to the fold curve. Finally, since $U$ is a desingularized slow flow, time reversion has to be applied in the saddle-type region $\{4x_4 y_4 < 1\}$.
\end{proof}

\begin{figure}[h]
	\centering
	\begin{overpic}[width=0.8\textwidth]{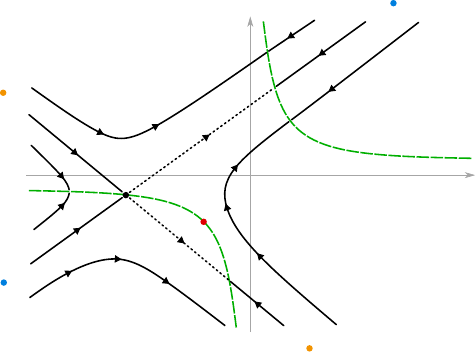}
		\put(26,28){$\zeta_1$}
		\put(100,34){\textcolor{gray}{$x_4$}}
		\put(48,70){\textcolor{gray}{$y_4$}}
		\put(13,9){attracting}
		\put(80,50){repelling}
		\put(70,20){saddle-type}
		\put(2,16){$\zeta_2$}
		\put(61,2){$\zeta_3$}
		\put(2,51){$\zeta_4$}
		\put(78,71){$\zeta_5$}
	\end{overpic}
	\caption{Phase portrait of the desingularized slow flow $U$ \eqref{fs:-a:sf:desing} on the critical manifold of \eqref{fs:-a} with reversed time orientation in the saddle-type region and $B_0 > C_0 > 0$. The fold curves are indicated dashed green, the red point represents the cusp. In case of $C_0 > B_0 > 0$, the folded saddle $\zeta_1$ is located on the branch ``below'' the cusp point.}
	\label{sf:-a:portrait}
\end{figure}

A similar analysis as above holds for the fast-slow system \eqref{fs:a}. In particular, the desingularized slow flow of \eqref{fs:a} on the corresponding critical manifold $E_3 = \{x_3^2 + y_3 + b_3 = 0, y_3^2 + x_3 + c_3 = 0\}$ is given by
\begin{equation}
	\begin{split}
		\label{fs:a:sf:desing}
		Q\colon
		\begin{bmatrix}
			\dot{x}_3\\ \dot{y}_3
		\end{bmatrix} &=
		\begin{bmatrix}
			0 & -2B_0\\
			-2 C_0 & 0
		\end{bmatrix}
		\begin{bmatrix}
			x_3\\ y_3
		\end{bmatrix}
		+ \begin{bmatrix}
			C_0\\ B_0
		\end{bmatrix},
	\end{split}
\end{equation}
where the time orientation of orbits in the saddle type region $\{4x_3y_3 < 1\}$ needs to be reversed. Lemma \ref{lemma:-a:sf} holds mutatis mutandis for the desingularized slow flow $Q$ \eqref{fs:a:sf:desing} of \eqref{fs:a}. In particular, if $B_0 C_0 > 0$, $Q$ has a folded saddle corresponding to $\zeta_6$ on the upper fold curve. If $B_0 C_0 < 0$, then $Q$ has a folded center located on the lower fold curve, which corresponds to $\zeta_1$ in this case. A phase portrait of $Q$ for the case $C_0 > B_0 > 0$ is given in figure \ref{portrait_kappa_a}.
\begin{figure}[h]
	\centering
	\begin{overpic}[width=0.8\textwidth]{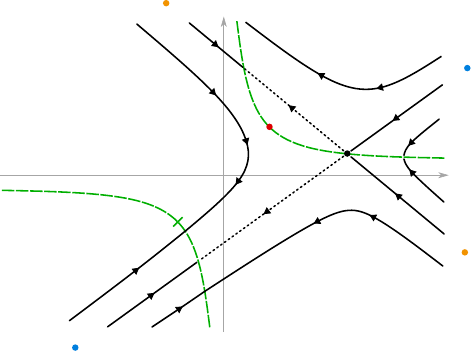}
		\put(5,16){attracting}
		\put(73,64){repelling}
		\put(20,45){saddle-type}
		\put(95,34){\textcolor{gray}{$x_3$}}
		\put(44,72){\textcolor{gray}{$y_3$}}
		\put(73,45){$\zeta_6$}
		\put(18,2){$\zeta_2$}
		\put(36,71){$\zeta_4$}
		\put(95,23){$\zeta_3$}
		\put(96,57){$\zeta_5$}
	\end{overpic}
	\caption{Phase portrait of the desingularized slow flow $Q$ \eqref{fs:a:sf:desing} on the critical manifold of \eqref{fs:a} with reversed time orientation in the saddle-type region and $B_0 > C_0 > 0$. The fold curves are indicated dashed green, the red point represents the cusp. In case of $C_0 > B_0 > 0$, $\zeta_6$ is located on the branch ``above'' the cusp point.}
	\label{portrait_kappa_a}
\end{figure}

\subsection{Dynamics of $\bar{Y}$ on the blow-up sphere}
\label{ss:further:y}

The vector fields $U$ \eqref{fs:-a:sf:desing} and $Q$ \eqref{fs:a:sf:desing} play not only an important role in the analysis of the fast-slow systems \eqref{fs:-a} and \eqref{fs:a}, being their desingularized slow flows on the critical manifold $E_4$, respectively $E_3$. In fact, $U$ also appears by blowing up the complete desingularized slow flow $Y$ \eqref{sf:desing} via
\begin{equation}
	\begin{split}
		\label{phi_-a_restriction}
		\Theta_{-a}\colon x = r_4 x_4, y = r_4 y_4, a = -r_4,
	\end{split}
\end{equation}
and restricting to the blow-up sphere $\{r_4 = 0\}$. Note that \eqref{phi_-a_restriction} is just the restriction of $\Phi_{-a}$ to the coordinates $(x_4,y_4,r_4)$. In other words, $U$ naturally appears in the blow-up $\bar{Y}$, in the restricted chart $\kappa_{-a}$. Similarly $Q$ arises in the blow-up of $Y$ in the restricted chart $\kappa_{a}$.

We summarize this, in case of $U$, in the following remark.

\begin{remark}
	\label{comm_diag}
	Let $(\Theta_{-a})_{\ast} \bar{Y}^{-a} = Y$, where $Y$ is the desingularized slow flow of $X$ and $\Theta_{-a}$ as in \eqref{phi_-a_restriction}. Then $U = \bar{Y}^{-a}\rvert_{r_4 = 0}$. Denote the vector field \eqref{blowup_-a} by $\bar{X}^{-a}$, such that $\bar{X}^{-a}\rvert_{r_4 = 0}$ gives the fast-slow system \eqref{fs:-a}. Then we have a commutative diagram:
		\begin{center}
			\begin{tikzcd}[row sep=large, column sep=5cm]
				X \arrow{r}{\text{slow flow desingularization}} \arrow{d}[left]{\text{blow-up via $\Phi_{-a}$}} & Y \arrow{d}{\text{blow-up via $\Theta_{-a}$}} \\
							\bar{X}^{-a}\rvert_{r_4 = 0}	\arrow{r}{\text{slow flow desingularization}} & \bar{Y}^{-a}\rvert_{r_4 = 0} = U
			\end{tikzcd}
		\end{center}
\end{remark}

As should be clear from the preceeding observations, one can employ $U$ and $Q$ to obtain the flow of the blown-up desingularized slow flow $\bar{Y}$ on the blow-up sphere. That is, in the following we show how the flow on the blow-up sphere in figure \ref{sf:desing_complete_blowup} was obtained. We consider the case $B_0, C_0 > 0$. The folded saddle of $U$ corresponding to $\zeta_1$ organizes the dynamics on the critical manifold $E_4$. Moreover, it organizes the dynamics of $\bar{Y}$ on one half of the 2-d blow-up sphere, where $\bar{Y}$ is the suitably rescaled blown-up desingularized slow flow \eqref{slow_flow_complete_blowup}. In other words, by the following Lemma we obtain the dynamics on the half-sphere $\bar{a} < 0$ in figure \ref{sf:desing_complete_blowup}, if we take the time reversion in the saddle-type region into account.

\begin{lemma}
	\label{kappa_-a_heterocl}
	Let $B_0, C_0 > 0$. Any trajectory of $U$ has asymptotic behavior as indicated by $\zeta_{2,3,4,5}$ in figure \ref{sf:-a:portrait}. That is, neglecting the time reversion for $U$ in the saddle-type region, the stable manifold of $\zeta_1$ is asymptotic to $\zeta_2$ for $y_4 \to -\infty$ and asymptotic to $\zeta_5$ for $y_4 \to +\infty$. The unstable manifold of $\zeta_1$ is asymptotic to $\zeta_4$ for $y_4 \to +\infty$ and $\zeta_3$ for $y_4 \to -\infty$. This determines the asymptotic behavior of any trajectory of $U$.
\end{lemma}
\begin{proof}
	Neglect the time reversion for $U$ in the saddle-type region. The stable and unstable manifolds of $\zeta_1$ are given by straight lines of the form
	\begin{equation}
		\begin{split}
			\label{s_u_mflds}
			x_4 = \pm \sqrt{B_0/C_0} y_4 + \text{const}.
		\end{split}
	\end{equation}
	Employing the appropriate coordinate changes from Lemma \ref{coordinate_changes} to \eqref{s_u_mflds}, one obtains the asymptotics in the assertion. Since $U$ is a linear system, the saddle $\zeta_1$ determines the forward and backward asymptotic behavior of all orbits. Hence, in each of the four sectors determined by the stable and unstable manifolds, any trajectory constitutes a heteroclinic connection with the corresponding behavior.
\end{proof}

A similar analysis applies to $Q$. Combining the phase portraits of $U$ and $Q$ leads to the qualitative flow of the blown-up desingularized slow flow $\bar{Y}$ on the 2-d blow-up sphere: The case $B_0 > C_0 > 0$ of $\bar{Y}$ visualized in figure \ref{sf:desing_complete_blowup} is obtained by glueing the phase portraits from figures \ref{portrait_kappa_a} and \ref{sf:-a:portrait}. That is, each desingularized slow flow in $\kappa_{\pm a}$ determines the flow on the half-spheres $\{\bar{a} > 0\}$ and $\{ \bar{a}<0 \}$. The glueing takes place along the equator, on which heteroclinic orbits run. These appear in figure \ref{fig:slow_flow_-y} along the $x_1$-axis, similarly in $\kappa_{-x},\kappa_x,\kappa_y$.

\begin{remark}
	Restricted to the slow flow coordinates $(x,y,a)$, the charts $\kappa_{\pm a}$ in fact correspond to gnomonic projections, through which great circles on the blow-up sphere correspond to straight lines. Therefore the (un-)stable manifolds of $\zeta_{1,6}$ correspond to great circles, as depicted in figure \ref{sf:desing_complete_blowup}.
\end{remark}

\clearpage

\appendix

\section{$\quad$}

\subsection{Takens' constrained differential equations}
\label{app:takens}

Here we give Takens' definition of constrained differential equations and their solutions, taken from \cite{takens1976}. After this, we briefly compare Takens' notions to the singular limit of \eqref{start}.

\begin{defn}[{{\cite{takens1976}}}]
	Let $\pi\colon E \to B$ be a smooth fiber bundle. A constrained differential equation on $E$ is a pair $(V,X)$ with smooth potential $V\colon E \to \R$ and smooth vector field $X \colon E \to TB$ with $X(e) \in T_{\pi(e)}B$ such that
	\begin{enumerate}[a),noitemsep]
		\item $V$ restricted to any fiber is proper and bounded from below,
		\item the set $S_V = \{ e \in E \,\rvert\, V\rvert_{\pi^{-1}(\pi(e))} \text{ has a critical point in } e \}$ is locally compact in the following sense: for each compact $K \subset B$ the set $S_V \cap \pi^{-1}(K)$ is compact.
	\end{enumerate}
	Further, define
	\begin{align*}
		S_{V,\text{min}} = \{ e \in S_V\, \rvert\, \text{ the Hessian of } V\rvert_{\pi^{-1}(\pi(e))} \text{ is positive (semi-)definite at } e \}.
	\end{align*}
\end{defn}

\begin{defn}[{{\cite{takens1976}}}]
	Let $(V,X)$ be a constrained differential equation on $E$ as above. A curve $\gamma \colon (\alpha,\beta) \to E$ is a solution of $(V,X)$ if
	\begin{enumerate}[i),noitemsep]
		\item $\gamma(t_0^{+}) := \lim_{t \downarrow t_0} \gamma(t)$ and $\gamma(t_0^{-}) := \lim_{t \uparrow t_0} \gamma(t)$ exist for all $t_0 \in (\alpha,\beta)$ and satisfy $\pi(\gamma(t_0^{+})) = \pi(\gamma(t_0^{-}))$ and $\gamma(t_0^{+}),\gamma(t_0^{-}) \in S_{V,\text{min}}$, 
		\item for all $t \in (\alpha,\beta)$, $X(\gamma(t^{-}))$ is the left derivative of $\pi \circ \gamma$ and $X(\gamma(t^{+}))$ is the right derivative of $\pi \circ \gamma$ at $t$,
		\item whenever $\gamma(t^{-}) \neq \gamma(t^{+})$ for some $t \in (\alpha,\beta)$, there is a curve in $\pi^{-1}(\pi(\gamma(t^{+})))$ from $\gamma(t^{-})$ to $\gamma(t^{+})$ along $V$ decreases monotonically.
	\end{enumerate}
\end{defn}

Let us briefly compare Takens' notion of a constrained differential equation to the singular limit of \eqref{start}: The fibers of the bundle can be viewed as the fast subsystems. $S_{V,\text{min}}$ can be viewed as the attracting region of the critical manifold with its surrounding singularities. Takens' solutions on $S_{V,\text{min}}$ are essentially determined by the vector field $X$. This is similar to the slow flow of the slow subsystem of \eqref{start}, which is well defined in regular points of the critical manifold since the critical manifold is a graph over the slow variables. Note that in the slow subsystem of \eqref{start} we allow dynamics on the full critical manifold, not only the attracting region and nearby singularities. When jumps occur in Takens' setting, they do so ``infinitely fast'' (producing discontinuities) along a fiber, such that the potential decreases along these jumps. In contrast to this, \eqref{start} has dynamics in the fast fibers, in fact gradient dynamics arising from a potential. The gradient dynamics force jumps along paths of steepest descent/ascent, depending on the used sign.

\subsection{Coordinate changes between the charts $\kappa_{\bullet}$}

Let $\kappa_{\bullet \to \circ}$ denote the coordinate change from $\kappa_{\bullet}$ to $\kappa_{\circ}$, defined on the overlap of the involved charts. That is, $\kappa_{\bullet \to \circ} = \kappa_{\circ} \circ \kappa_{\bullet}^{-1}$. The overlap is apparent from our notation: For example, to switch from $\kappa_{-y}$ to $\kappa_\eps$, we need $\eps_1 > 0$ in $\kappa_{-y}$. The formula for $\kappa_{-y \to \eps}$ can then be computed by relating $\Phi_{-y}$ and $\Phi_{\eps}$. The chart $\kappa_{\ex}$ plays an exceptional role and intuitively corresponds to the ``direction $x + y$''. For example, to apply $\kappa_{\eps \to \ex}$ we need points in the overlap with $x_2 + y_2 > 0$ in $\kappa_\eps$. For convenience of the reader, the following Lemma states the relevant coordinate changes, most of them used extensively and without reference throughout the analysis in section \ref{sec:blowup}.

\begin{lemma}
	\label{coordinate_changes}
	The relevant coordinate changes $\kappa_{\bullet \to \circ}$ are given by the following maps:
	\begingroup
	\allowdisplaybreaks
	\begin{align*}
			\kappa_{-y \to \eps} &\colon \begin{cases}
				x_2 &=  x_1 \eps_1^{-1/3}\\
				y_2 &=  - \eps_1^{-1/3}\\
				a_2 &= a_1 \eps_1^{-1/3}\\
				b_2 &= b_1 \eps_1^{-2/3}\\
				c_2 &= c_1 \eps_1^{-2/3}\\
				r_2 &= r_1 \eps_1^{1/3},
			\end{cases}
			&&\kappa_{-y \to -x} \colon \begin{cases}
				r_0 &= - r_1 x_1\\
				y_0 &= x_1^{-1}\\
				a_0 &= a_1 (-x_1)^{-1}\\
				b_0 &= b_1 (-x_1)^{-2}\\
				c_0 &= c_1 (-x_1)^{-2}\\
				\eps_0 &= \eps_1 (-x_1)^{-3},
			\end{cases}  \\
			\kappa_{-y \to a} &\colon \begin{cases}
				x_3 &= x_1 a_1^{-1}\\
				y_3 &= -a_1^{-1}\\
				r_3 &= r_1a_1\\
				b_3 &= b_1a_1^{-2}\\
				c_3 &= c_1a_1^{-2}\\
				\eps_3 &= \eps_1a_1^{-3}
			\end{cases}
			&&\kappa_{-y \to -a} \colon \begin{cases}
				x_4 &= x_1 (-a_1)^{-1}\\
				y_4 &= a_1^{-1}\\
				r_4 &= -r_1 a_1\\
				b_4 &= b_1(-a_1)^{-2}\\
				c_4 &= c_1(-a_1)^{-2}\\
				\eps_4 &= \eps_1(-a_1)^{-3}
			\end{cases}\\
			\kappa_{\eps \to a} &\colon \begin{cases}
				x_3 &= x_2 a_2^{-1}\\
				y_3 &= y_2 a_2^{-1}\\
				r_3 &= r_2 a_2\\
				b_3 &= b_2 a_2^{-2}\\
				c_3 &= c_2 a_2^{-2}\\
				\eps_3 &= a_2^{-3}
			\end{cases}
			&&\kappa_{\eps \to -a} \colon \begin{cases}
				x_4 &= x_2 (-a_2)^{-1}\\
				y_4 &= y_2 (-a_2)^{-1}\\
				r_4 &= -r_2 a_2\\
				b_4 &= b_2(-a_2)^{-2}\\
				c_4 &= c_2(-a_2)^{-2}\\
				\eps_4 &= (-a_2)^{-3}
			\end{cases}\\
			\kappa_{\eps \to -y} &\colon \begin{cases}
				x_1 &= x_2 (-y_2)^{-1}\\
				r_1 &= -r_2 y_2\\
				a_1 &= a_2(-y_2)^{-1}\\
				b_1 &= b_2 (-y_2)^{-2}\\
				c_1 &= c_2 (-y_2)^{-2}\\
				\eps_1 &= (-y_2)^{-3},
				\end{cases}
			&&\kappa_{\eps \to \ex} \colon \begin{cases}
				x_5 &= (x_2 - y_2)(x_2 + y_2)^{-1}\\
				r_5 &= r_2(x_2 + y_2)/2 \\
				a_5 &= 2a_2(x_2 + y_2)^{-1}\\
				b_5 &= 4b_2(x_2 + y_2)^{-2}\\
				c_5 &= 4c_2(x_2 + y_2)^{-2}\\
				\eps_5 &= 8(x_2 + y_2)^{-3}
			\end{cases}\\
			\kappa_{a \to \ex} &\colon \begin{cases}
				x_5 &=(x_3-y_3)(x_3+y_3)^{-1}\\
				r_5 &=r_3(x_3+y_3)/2\\
				a_5 &=2(x_3+y_3)^{-1}\\
				b_5 &=4b_3(x_3+y_3)^{-2}\\
				c_5 &=4c_3(x_3+y_3)^{-2}\\
				\eps_5 &= 8\eps_3(x_3+y_3)^{-3}
			\end{cases}
			 &&\kappa_{-a \to \ex} \colon \begin{cases}
				x_5 &= (x_4 - y_4)(x_4 + y_4)^{-1}\\
				r_5 &= r_4(x_4 + y_4)/2\\
				a_5 &= 2 (x_4 + y_4)^{-1}\\
				b_5 &= 4b_4(x_4 + y_4)^{-2}\\
				c_5 &= 4 c_4(x_4 + y_4)^{-2}\\
				\eps_5 &= 8\eps_4(x_4 + y_4)^{-3}
			\end{cases}
	\end{align*}%
	\endgroup
\end{lemma}

\subsection{Another universal unfolding for the hyperbolic umbilic}
\label{app:universal_unfolding}

Consider 
\begin{equation}
	\begin{split}
		\tilde{V} = \frac{1}{3}\tilde{y}^3 + \tilde{y}\tilde{x}^2 + \tilde{a}(\tilde{y}^2 - \tilde{x}^2) + \tilde{b}\tilde{x} + \tilde{c}\tilde{y}
	\end{split}
\end{equation}
from \eqref{alternative_potential}. Here we briefly show that $\tilde{V}$ is a universal unfolding (among others) for the hyperbolic umbilic catastrophe, compare to table \ref{the7}.

We briefly recall: Let $\mathcal{E}_n$ denote the local ring of germs at $0 \in \R^n$. Let $\mathfrak{m}_n$ denote the ideal of germs vanishing at 0. $\mathfrak{m}_n$ is the maximal ideal in $\mathcal{E}_n$, which means that there does not exist a proper ideal of $\mathcal{E}_n$ strictly containing $\mathfrak{m}_n$.
Let $\Delta(f)$ denote the Jacobi ideal of $f \in \mathfrak{m}_n^2$, that is, the ideal generated by the partial derivatives of $f$ over $\mathcal{E}_n$. Let the codimension of $f$ be defined by the dimension of $\mathfrak{m}_n / \Delta(f)$ as an $\R$-vector space. The following theorem can be found in \cite{michor1985cat}, for example.

\begin{thm}
	Let $f \in \mathfrak{m}_n^2$ be a germ with finite codimension and let $b_1,\dots,b_d$ be representatives of a basis for $\mathfrak{m}_n / \Delta(f)$. Then $V(z,\alpha) = f(z) + \sum_{i = 1}^{d} b_i(z) \alpha_i $ is a universal unfolding of $f$.
\end{thm}

Consider the germ $\tilde{V}$ unfolds, that is, let $f = \frac{1}{3}y^3 + yx^2$. The hyperbolic umbilic germ $x^3 + y^3$ is right-equivalent to $2 \tilde{x}^3 + 6\tilde{x}\tilde{y}^2$ by putting $x = \tilde{x} + \tilde{y}$ and $y = \tilde{x} - \tilde{y}$. By scaling the variables, the germ $2 \tilde{x}^3 + 6\tilde{x}\tilde{y}^2$ is right equivalent to $f$. Hence $f$ is a valid germ to represent the hyperbolic umbilic catastrophe. Now consider $\mathfrak{m}_2 / \Delta(f)$. The Jacobi-ideal $\Delta(f)$ is generated by $x^2 + y^2$ and $2xy$ over $\mathcal{E}_2$. Surely $x$ and $y$ are not contained in $\Delta(f)$. But $x^k = x^{k-2}(x^2 + y^2) - \frac{1}{2}x^{k-1} y (2 x y) $ for $k \ge 3$, and similarly $y^k$ for $k \ge 3$ is contained in $\Delta(f)$. Certain quadratic germs are still missing in $\Delta(f)$. We can choose $y^2 - x^2$ together with $x$ and $y$ to form a basis for the $\R$-vector space $\mathfrak{m}_2 / \Delta(f)$. This gives the universal unfolding $\tilde{V}$.

\section*{Acknowledgements}

CK and MS acknowledge funding by the Deutsche Forschungsgemeinschaft (DFG, German Research Foundation) - KU 3333/7-1.
CK acknowledges partial support by a Lichtenberg Professorship funded by the VolkswagenStiftung.
CK and MS acknowledge support by TUM International Graduate School of Science and Engineering (IGSSE).
MS would like to thank the University of Groningen, where part of this research was conducted. The authors thank the reviewers and editors for their comments and remarks that helped to improve the manuscript.

\printbibliography

@article{hayes2016geometric,
	author = {Michael G. Hayes and Tasso J. Kaper and Peter Szmolyan and Martin Wechselberger},
	doi = {https://doi.org/10.1016/j.indag.2015.11.005},
	issn = {0019-3577},
	journal = {Indagationes Mathematicae},
	note = {Dynamics and Geometry},
	number = {5},
	pages = {1184-1203},
	title = {Geometric desingularization of degenerate singularities in the presence of fast rotation: A new proof of known results for slow passage through Hopf bifurcations},
	url = {https://www.sciencedirect.com/science/article/pii/S0019357715001020},
	volume = {27},
	year = {2016}}

@article{mitry2017folded,
	author = {John Mitry and Martin Wechselberger},
	doi = {10.1137/15m1045065},
	journal = {{SIAM} Journal on Applied Dynamical Systems},
	month = {1},
	number = {1},
	pages = {546--596},
	publisher = {Society for Industrial {\&} Applied Mathematics ({SIAM})},
	title = {Folded Saddles and Faux Canards},
	url = {https://doi.org/10.1137%2F15m1045065},
	volume = {16},
	year = 2017}

@book{arnold2012singularities,
	author = {Arnol'd, V. I. and Gusein-Zade, S. M. and Varchenko, A. N.},
	isbn = {9780817683405},
	lccn = {2012938547},
	publisher = {Birkh{\"a}user Boston},
	series = {Modern Birkh{\"a}user Classics},
	title = {Singularities of Differentiable Maps, Volume 1: Classification of Critical Points, Caustics and Wave Fronts},
	year = {2012}}

@inproceedings{takens1976,
	address = {Berlin, Heidelberg},
	author = {Takens, Floris},
	booktitle = {Structural Stability, the Theory of Catastrophes, and Applications in the Sciences},
	editor = {Hilton, Peter},
	isbn = {978-3-540-38254-6},
	pages = {143--234},
	publisher = {Springer Berlin Heidelberg},
	title = {Constrained equations; a study of implicit differential equations and their discontinuous solutions},
	year = {1976}}

@book{kuehn2015multiple,
	author = {Kuehn, Christian},
	isbn = {978-3-319-12315-8},
	publisher = {Springer International Publishing},
	series = {Applied Mathematical Sciences},
	title = {Multiple Time Scale Dynamics},
	year = {2015}}

@book{demazure2000,
	author = {Michel Demazure},
	publisher = {Springer-Verlag Berlin Heidelberg},
	title = {Bifurcations and Catastrophes: Geometry Of Solutions To Nonlinear Problems},
	year = {2000}}

@book{michor1985cat,
	author = {Michor, Peter W.},
	publisher = {Universitatea din Timisoara},
	series = {Monografii Matematice 24},
	title = {Elementary catastrophe theory},
	year = {1985}}

@article{szmolyan2001canards,
	author = {Peter Szmolyan and Martin Wechselberger},
	doi = {https://doi.org/10.1006/jdeq.2001.4001},
	issn = {0022-0396},
	journal = {Journal of Differential Equations},
	number = {2},
	pages = {419 - 453},
	title = {Canards in R3},
	url = {http://www.sciencedirect.com/science/article/pii/S002203960194001X},
	volume = {177},
	year = {2001}}

@book{mishchenko1980,
	author = {Mishchenko, E. F. and Rozov, N.Kh.},
	isbn = {9780306392535},
	lccn = {78004517},
	publisher = {Plenum Press},
	series = {Mathematical Concepts and Methods in Science and Engineering},
	title = {Differential Equations with Small Parameters and Relaxation Oscillations},
	year = {1980}}

@misc{NIST:DLMF,
	howpublished = {http://dlmf.nist.gov/, Release 1.1.1 of 2021-03-15},
	note = {F.~W.~J. Olver, A.~B. {Olde Daalhuis}, D.~W. Lozier, B.~I. Schneider, R.~F. Boisvert, C.~W. Clark, B.~R. Miller, B.~V. Saunders, H.~S. Cohl, and M.~A. McClain, eds.},
	title = {{NIST Digital Library of Mathematical Functions}},
	url = {http://dlmf.nist.gov/}}

@article{boardman1967,
	author = {Boardman, John. M.},
	journal = {Publications Math\'ematiques de l'IH\'ES},
	language = {en},
	mrnumber = {37 \#6945},
	pages = {21--57},
	publisher = {Institut des Hautes \'Etudes Scientifiques},
	title = {Singularities of differentiable maps},
	url = {www.numdam.org/item/PMIHES_1967__33__21_0/},
	volume = {33},
	year = {1967},
	zbl = {0165.56803}}

@article{kojakhmetov2014constrained,
	author = {H. Jard{\'o}n-Kojakhmetov and Henk W. Broer},
	doi = {https://doi.org/10.1016/j.jde.2014.04.022},
	issn = {0022-0396},
	journal = {Journal of Differential Equations},
	number = {4},
	pages = {1012-1055},
	title = {Polynomial normal forms of constrained differential equations with three parameters},
	url = {https://www.sciencedirect.com/science/article/pii/S0022039614001739},
	volume = {257},
	year = {2014}}

@article{krupa2001extending,
	author = {Krupa, Martin and Szmolyan, Peter},
	journal = {SIAM journal on mathematical analysis},
	number = {2},
	pages = {286--314},
	publisher = {SIAM},
	title = {Extending geometric singular perturbation theory to nonhyperbolic points---fold and canard points in two dimensions},
	volume = {33},
	year = {2001}}

@article{broer2013cusp,
	author = {Broer, Henk W. and Kaper, Tasso J and Krupa, Martin},
	journal = {Journal of Dynamics and Differential Equations},
	number = {4},
	pages = {925--958},
	publisher = {Springer},
	title = {Geometric desingularization of a cusp singularity in slow--fast systems with applications to Zeeman's examples},
	volume = {25},
	year = {2013}}

@article{kojakhmetov2016cusp,
	author = {H. Jard{\'o}n-Kojakhmetov and Henk W. Broer and R. Roussarie},
	doi = {https://doi.org/10.1016/j.jde.2015.10.045},
	issn = {0022-0396},
	journal = {Journal of Differential Equations},
	number = {4},
	pages = {3785-3843},
	title = {Analysis of a slow--fast system near a cusp singularity},
	url = {https://www.sciencedirect.com/science/article/pii/S0022039615005884},
	volume = {260},
	year = {2016}}

@article{krupa2001transcritical,
	author = {M. Krupa and P. Szmolyan},
	doi = {10.1088/0951-7715/14/6/304},
	journal = {Nonlinearity},
	month = {9},
	number = {6},
	pages = {1473--1491},
	publisher = {{IOP} Publishing},
	title = {Extending slow manifolds near transcritical and pitchfork singularities},
	url = {https://doi.org/10.1088/0951-7715/14/6/304},
	volume = {14},
	year = 2001}

@article{chicone1979,
	author = {Carmen C. Chicone},
	doi = {https://doi.org/10.1016/0022-0396(79)90085-8},
	issn = {0022-0396},
	journal = {Journal of Differential Equations},
	number = {2},
	pages = {159-166},
	title = {Quadratic gradients on the plane are generically Morse-Smale},
	url = {https://www.sciencedirect.com/science/article/pii/0022039679900858},
	volume = {33},
	year = {1979}}

@inproceedings{zeeman1976,
	address = {Berlin, Heidelberg},
	author = {Zeeman, E. C.},
	booktitle = {Structural Stability, the Theory of Catastrophes, and Applications in the Sciences},
	editor = {Hilton, Peter},
	isbn = {978-3-540-38254-6},
	pages = {328--366},
	publisher = {Springer Berlin Heidelberg},
	title = {The umbilic bracelet and the double-cusp catastrophe},
	year = {1976}}

@incollection{guckenheimer1973,
	author = {John Guckenheimer},
	booktitle = {Dynamical Systems},
	editor = {M.M. Peixoto},
	pages = {95-109},
	publisher = {Academic Press},
	title = {Bifurcation and Catastrophe},
	year = {1973}}

@article{khesin1990,
	author = {Khesin, B. A.},
	doi = {https://doi.org/10.1007/BF01095251},
	issn = {0022-0396},
	journal = {Journal of Soviet Mathematics},
	number = {2},
	pages = {159-166},
	title = {Bifurcations in gradient dynamic systems},
	volume = {33},
	year = {1990}}

@incollection{vegter1982,
	author = {Vegter, Gert},
	booktitle = {Bifurcation, th\'eorie ergodique et applications - 22-26 juin 1981},
	language = {en},
	mrnumber = {724443},
	number = {98-99},
	pages = {39-73},
	publisher = {Soci\'et\'e math\'ematique de France},
	series = {Ast\'erisque},
	title = {Bifurcations of gradient vectorfields},
	url = {http://www.numdam.org/item/AST_1983__98-99__39_0/},
	year = {1982},
	zbl = {0526.58034}}

@book{poston1996catastrophe,
	author = {Poston, T. and Stewart, I.},
	isbn = {9780486692715},
	lccn = {96021795},
	publisher = {Dover Publications},
	series = {Dover books on mathematics},
	title = {Catastrophe Theory and Its Applications},
	year = {1996}}

@article{fenichel1979,
	author = {Neil Fenichel},
	doi = {https://doi.org/10.1016/0022-0396(79)90152-9},
	issn = {0022-0396},
	journal = {Journal of Differential Equations},
	number = {1},
	pages = {53-98},
	title = {Geometric singular perturbation theory for ordinary differential equations},
	url = {https://www.sciencedirect.com/science/article/pii/0022039679901529},
	volume = {31},
	year = {1979}}

@incollection{jones1995,
	author = {Jones, Christopher KRT},
	booktitle = {Dynamical systems},
	pages = {44--118},
	publisher = {Springer},
	title = {Geometric singular perturbation theory},
	year = {1995}}

@book{wiggins1994normally,
	author = {Wiggins, Stephen},
	publisher = {Springer Science \& Business Media},
	title = {Normally hyperbolic invariant manifolds in dynamical systems},
	volume = {105},
	year = {1994}}

@article{arnold1973normal,
	author = {V. I. Arnol{\textquotesingle}d},
	doi = {10.1007/bf01077644},
	journal = {Functional Analysis and Its Applications},
	number = {4},
	pages = {254--272},
	publisher = {Springer Science and Business Media {LLC}},
	title = {Normal forms for functions near degenerate critical points, the Weyl groups of Ak, Dk, Ek and Lagrangian singularities},
	url = {https://doi.org/10.1007/bf01077644},
	volume = {6},
	year = {1973}}

@book{broecker1975germs,
  doi = {10.1017/cbo9781107325418},
  url = {https://doi.org/10.1017/cbo9781107325418},
  year = {1975},
  month = {7},
  publisher = {Cambridge University Press},
  author = {Theodor Br\"{o}cker},
  title = {Differentiable Germs and Catastrophes}
}

@inproceedings{hjk2021survey,
  author = {Hildeberto Jard{\'o}n-Kojakhmetov and Christian Kuehn},
  title = {A survey on the blow-up method for fast-slow systems},
  doi = {10.1090/conm/775},
  url = {https://doi.org/10.1090/conm/775},
  year = {2021},
  pages = {115--160},
  month = dec,
  volume = {775},
  publisher = {American Mathematical Society},
  editor = {Fernando Galaz-Garc{\'i}a and Cecilia Gonz{\'a}lez-Tokman and Juan Pardo Mill{\'a}n},
  booktitle = {Mexican Mathematicians in the World: Trends and Recent Contributions},
  series = {Contemporary Mathematics}
}

@book{kuznetsov2004,
  doi = {10.1007/978-1-4757-3978-7},
  url = {https://doi.org/10.1007/978-1-4757-3978-7},
  year = {2004},
  publisher = {Springer New York},
  author = {Yuri A. Kuznetsov},
  title = {Elements of Applied Bifurcation Theory}
}

@article{wechselberger2012apropos,
  doi = {10.1090/s0002-9947-2012-05575-9},
  url = {https://doi.org/10.1090/s0002-9947-2012-05575-9},
  year = {2012},
  month = jun,
  publisher = {American Mathematical Society ({AMS})},
  volume = {364},
  number = {6},
  pages = {3289--3309},
  author = {Martin Wechselberger},
  title = {{\`{A}} propos de canards (Apropos canards)},
  journal = {Transactions of the American Mathematical Society}
}

@inproceedings{takens1971manifolds,
	title = {A Solution - An example as requested in the problem of R. Thom, D 4c},
	author = {F. Takens},
	pages = {231},
  booktitle={Manifolds - Amsterdam 1970: Proceedings of the Nuffic Summer School on Manifolds Amsterdam, August 17 - 29, 1970},
  editor={Kuiper, N.H.},
  series={Lecture Notes in Mathematics},
  year={1971},
  publisher={Springer Berlin Heidelberg}
}

@book{de2021canard,
  title={Canard Cycles},
  author={De Maesschalck, Peter and Dumortier, Freddy and Roussarie, Robert},
  year={2021},
  publisher={Springer}
}

@book{wechselberger2020geometric,
  title={Geometric singular perturbation theory beyond the standard form},
  author={Wechselberger, Martin},
  year={2020},
  publisher={Springer Nature}
}

@article{khesin1986,
	author = {Khesin, B. A.},
	doi = {10.1007/BF01078483},
	issn = {1573-8485},
	journal = {Functional Analysis and Its Applications},
	number = {3},
	pages = {250--252},
	title = {Bifurcation of singular points of gradient dynamical systems},
	url = {https://doi.org/10.1007/BF01078483},
	volume = {20},
	year = {1986},
}

@article{perez2023slow,
  title={Slow-fast normal forms arising from piecewise smooth vector fields},
  author={Perez, Otavio Henrique and Rond{\'o}n, Gabriel and da Silva, Paulo Ricardo},
  journal={Journal of Dynamical and Control Systems},
  volume={29},
  number={4},
  pages={1709--1726},
  year={2023},
  publisher={Springer}
}

@article{colombo2012bifurcations,
  title={Bifurcations of piecewise smooth flows: Perspectives, methodologies and open problems},
  author={Colombo, Alessandro and di Bernardo, Mario and Hogan, SJ and Jeffrey, MR},
  journal={Physica D: Nonlinear Phenomena},
  volume={241},
  number={22},
  pages={1845--1860},
  year={2012},
  publisher={Elsevier}
}

@article{jeffrey2011geometry,
  title={The geometry of generic sliding bifurcations},
  author={Jeffrey, Mike R and Hogan, Stephen John},
  journal={SIAM review},
  volume={53},
  number={3},
  pages={505--525},
  year={2011},
  publisher={SIAM}
}

@article{panazzolo2017regularization,
  title={Regularization of discontinuous foliations: Blowing up and sliding conditions via Fenichel theory},
  author={Panazzolo, Daniel and da Silva, Paulo R},
  journal={Journal of Differential Equations},
  volume={263},
  number={12},
  pages={8362--8390},
  year={2017},
  publisher={Elsevier}
}

@article{llibre2009study,
  title={Study of singularities in nonsmooth dynamical systems via singular perturbation},
  author={Llibre, Jaume and Da Silva, Paulo R and Teixeira, Marco A},
  journal={SIAM Journal on Applied Dynamical Systems},
  volume={8},
  number={1},
  pages={508--526},
  year={2009},
  publisher={SIAM}
}

\end{document}